\documentclass[12pt,a4paper]{amsart}
\usepackage{amsthm,amsfonts,amsmath,amssymb,latexsym}
\usepackage{epsfig,graphics}
\usepackage[dvipsnames]{xcolor}
\usepackage{cancel}
\usepackage[all]{xy}
\usepackage{hyperref}
\usepackage{stmaryrd}
\usepackage{verbatim}
\usepackage{bm}
\usepackage{mathabx}
\usepackage{enumitem}
\usepackage{pgf,amsmath,tikz,type1cm,fix-cm}
\usetikzlibrary{positioning}
 \usepackage[normalem]{ulem}

\setlist[itemize,1]{leftmargin=\dimexpr 26pt-.1in}

\newtheorem{PARA}{}[section]
\newtheorem{theorem}[PARA]{Theorem}
\newtheorem{corollary}[PARA]{Corollary}
\newtheorem{lemma}[PARA]{Lemma}
\newtheorem{proposition}[PARA]{Proposition}
\newtheorem{definition}[PARA]{Definition}
\newtheorem{definition-proposition}[PARA]{Definition-Proposition}
\newtheorem{conjecture}[PARA]{Conjecture}
\newtheorem{question}[PARA]{Question}
\theoremstyle{definition}
\newtheorem{remark}[PARA]{Remark}
\newtheorem{remarks}[PARA]{Remarks}
\theoremstyle{theorem}
\newtheorem{example}[PARA]{Example}

\newcommand{\para}{\begin{PARA}\rm}
\newcommand{\arap}{\end{PARA}\rm}
\newcommand{\dfn}{\begin{definition}\rm}
\newcommand{\nfd}{\end{definition}\rm}
\newcommand{\rmk}{\begin{remark}\rm}
\newcommand{\kmr}{\end{remark}\rm}
\newcommand{\xmpl}{\begin{example}\rm}
\newcommand{\lpmx}{\end{example}\rm}
\newcommand{\mrm}[1]{{\mathrm{#1}}}
\newcommand{\brat}[1]{{\langle#1\rangle}}
\newcommand{\const}{{\mrm{const}}}

\newcommand{\cL}{\mathcal{L}}

\newcommand{\cP}{\mathcal{P}}

\newcommand{\cS}{\mathcal{S}}

\newcommand{\al}{{\alpha}}

\newcommand{\C}{{\mathbb{C}}}
\newcommand{\Ca}{{\mathbb{O}}}

\newcommand{\F}{{\mathbb{F}}}
\renewcommand{\H}{{\mathbb{H}}}
\newcommand{\bK}{{\mathbb{K}}}

\newcommand{\N}{{\mathbb{N}}}
\newcommand{\Q}{{\mathbb{Q}}}
\newcommand{\R}{{\mathbb{R}}}

\newcommand{\Z}{{\mathbb{Z}}}
\newcommand{\Cr}{\mathrm{ Cr}\, }  
\newcommand{\ccr}{\mathrm{ cr}\, }  
\newcommand{\im}{\mathrm{im}\,}        
\newcommand{\Gr}{\mathrm{ Gr}\, }  

\newcommand{\ev}{\mathrm{ev}}

\newcommand{\Hom}{\mathrm{Hom}}

\newcommand{\Sp}{\mathrm{Sp}}
\newcommand{\eps}{{\varepsilon}}

\def\NABLA#1{{\mathop{\nabla\kern-.5ex\lower1ex\hbox{$#1$}}}}
\def\Nabla#1{\nabla\kern-.5ex{}_{#1}}
\def\Tabla#1{\Tilde\nabla\kern-.5ex{}_{#1}}
\renewcommand{\Tilde}{\widetilde}

\newcommand{\p}{{\partial}}

\newcommand{\la}{\langle}
\newcommand{\ra}{\rangle}
\newcommand{\wh}{\widehat}
\newcommand{\ol}{\overline}

\newcommand{\K}{\mathbb{K}}

\newcommand{\kai}{\color{blue}}
\newcommand{\alex}{\color{magenta}}

\newcommand{\boldmu}{{\boldsymbol{\mu}}}

\newcommand{\boldeps}{{\boldsymbol{\eps}}}

\newcommand{\boldp}{{\boldsymbol{p}}}

\begin{document}

\title[Resonances and string point invertibility for CROSS]{Resonances and string point invertibility for compact rank one symmetric spaces}
\author{K.~Cieliebak}
\address{Universit\"at Augsburg \newline Universit\"atsstrasse 14, D-86159 Augsburg, Germany}
\email{kai.cieliebak@math.uni-augsburg.de}
\author{A.~Oancea}
\address{Institut de Recherche Math\'ematique Avanc\'ee, IRMA \newline Universit\'e de Strasbourg \newline 7 Rue Descartes, 67084 Strasbourg Cedex, France}
\email{oancea@unistra.fr}
\author{E.~Shelukhin}
\address{D\'epartement de math\'ematiques et de statistique\newline Universit\'e de Montr\'eal \newline Pavillon Andr\'e-Aisenstadt\newline C.P. 6128 Succ.  Centre-ville \newline Montr\'eal (Qu\'ebec) H3C 3J7, Canada}
\email{shelukhin@dms.umontreal.ca}
\date{\today}


\begin{abstract} 
{We calculate the Batalin--Vilkovisky (BV) algebra\break structure of Rabinowitz loop homology for compact rank one symmetric spaces. As a consequence, we prove that such a space satisfies a natural homological condition called string point invertibility if and only if its Euler characteristic is equal to the characteristic of the coefficient field for loop homology. This implies certain cases of Viterbo's conjecture on a uniform bound on the spectral norm of exact Lagrangian submanifolds in cotangent disk bundles. Furthermore, we prove that whenever a compact rank one symmetric space is string point invertible, the critical levels of its loop homology classes with respect to an arbitrary Riemannian metric satisfy a resonance condition with respect to degrees and a density condition for closed geodesics.
This generalizes results of Hingston and Rademacher for spheres to a broader class of compact rank one symmetric spaces. }  	
\end{abstract}

\maketitle


\setcounter{tocdepth}{1}
{\footnotesize{\tableofcontents}}


\section{Introduction}\label{sec:introduction}

This article is motivated by three phenomena, two in Finsler geometry and one in symplectic geometry. 
The first two are the {\em resonance theorem} and the {\em density theorem} for Finsler metrics on spheres by Hingston and Rademacher~\cite{Hingston-Rademacher}.
The third one is a conjecture by Viterbo~\cite{Viterbo23} on the spectral norm of Lagrangians in cotangent bundles and its recent proof, over certain manifolds, by the third named author~\cite{Shelukhin19}. 

All three phenomena rely on string topology. More precisely, for an $n$-dimensional closed manifold $M$ consider its free loop space $\Lambda=C(S^1,M)$, where $S^1=\R/\Z$. Its degree shifted homology 
$\H_*\Lambda=H_{*+n}\Lambda$ carries the structure of a Batalin--Vilkovisky (BV) algebra defined by the loop product and the BV operator arising from the reparametrization circle action~\cite{CS}. Similarly, the degree shifted cohomology $H^{1-n-*}\Lambda$ is a BV algebra with the cohomology product and the dual BV operator~\cite{Goresky-Hingston}.\footnote{
The loop product and cohomology product are also known as the Chas--Sullivan and Goresky--Hingston product, respectively. }
The proofs in~\cite{Hingston-Rademacher} are based on the computation of these BV algebras for spheres, while in~\cite{Shelukhin19} Viterbo's conjecture is proved if $\H_*\Lambda$ is {\em string point invertible}, i.e., the fundamental class of $M$ is obtained from the point class by repeated loop brackets with other classes (see~\S\ref{sec:string_point_invertibility}).

In this article, we pursue a unified approach to these phenomena for {\em compact rank one symmetric spaces (CROSS)}, i.e., the manifolds $S^n$, $\R P^n$, $\C P^d$, $\H P^d$, and $\Ca P^2$. 
Our main tool is {\em Rabinowitz loop homology} $\wh{H}_n\Lambda$, introduced in~\cite{CHO-PD} as the Rabinowitz Floer homology of the unit sphere cotangent bundle $S^*M$. Its degree shifted version $\wh{\H}_*\Lambda=\wh{H}_{*+n}\Lambda$ admits for $n\geq 3$ a splitting
$$
  \wh{\H}_*\Lambda = \ol{\H}_*\Lambda \oplus \ol{H}^{1-n-*}\Lambda
$$
into reduced loop homology $\ol{\H}^*\Lambda=\H_*(\Lambda)/\chi(M)[\mathrm{pt}]$ and reduced loop cohomology. It carries a BV algebra structure having both summands as subalgebras, and a Poincar\'e duality isomorphism exchanging the two factors and preserving the BV algebra structure. 

Now we explore the fact that each CROSS $M$ with its homogeneous metric is an {\em SC-manifold}, i.e., there exists $L>0$ such that all geodesics of length $L$ are simple and closed. Hence, the Reeb flow on $S^*M$ for this metric is $L$-periodic. In this situation, Uebele's theorem~\cite{Uebele-products} provides a powerful tool for computing the ring structure on $\wh{\H}_*\Lambda$, and therefore on $\ol{\H}_*\Lambda$ and $\ol H^{1-n-*}\Lambda$ via the splitting above. Using this, we obtain the following results. 

In~\S\ref{sec:BV} we prove some foundational statements about the BV algebra structure of Rabinowitz Floer homology. 

In~\S\ref{sec:SC-manifolds} we recall the Bott-Samelson theorem about manifolds all of whose geodesics are closed (SC-manfolds), we formulate a conjectural generalization of this theorem, and we give evidence for this conjecture. 

In~\S\ref{sec:BV-CHCa} we compute the BV algebra structure on Rabinowitz loop homology $\wh{\H}_*\Lambda$ for $\C P^d$, $\H P^d$ and $\Ca P^2$ for each coefficient field $\K$ (Theorems~\ref{thm:Rabinowitz_BV_projective_spaces_char_divides}, \ref{thm:Rabinowitz_BV_projective_spaces_char_does_not_divide}, \ref{thm:Rabinowitz_BV_S2_char_does_not_divide} and \ref{thm:Rabinowitz_BV_S2_char_2}). 
Here we rely on previous computations of the loop homology BV algebra with integer coefficients~\cite{Cadek-Moravec, Chataur-Le_Borgne, Hepworth, Yang}. To compute the BV operator on $\wh{\H}_*\Lambda$ we use the fact that it vanishes in local homology at action level $0$.
As a consequence, we obtain the loop cohomology BV algebra with arbitrary field coefficients.
Here and throughout this paper the case $\C P^1=S^2$, which is the one sphere not covered in~\cite{Hingston-Rademacher}, always requires separate treatment. 

Based on these computations, we then address the three phenomena above. Since the case of spheres is covered in~\cite{Hingston-Rademacher}, we first consider the $\C P^d$, $\H P^d$, and $\Ca P^2$. 

In~\S\ref{sec:string_point_invertibility} we characterize string point invertibility:

{\bf Theorem A} {\sl (cf.~Theorem~\ref{thm:string-point}). 
Let $M$ be one of the CROSS $\C P^d$, $\H P^d$, or $\Ca P^2$ (set $d=2$ in this last case). Then $M$ is string point invertible over a field of characteristic $p$ if and only $p$ equals the Euler characteristic $\chi(M)=d+1$ (in particular, $d+1$ must be prime). }

Note that this theorem implies some instances of Viterbo's conjecture, which are however covered by more general results in~\cite{Shelukhin22,Viterbo2022-inverse, Guillermou2022}. 

In~\S\ref{sec: res} we extend the resonance theorem of~\cite{Hingston-Rademacher} from spheres to other CROSS:

{\bf Theorem B} {\sl (cf.~Theorems~\ref{thm: resonance} and~\ref{thm: resonance CP1}). 
Let $M$ be one of the CROSS $\C P^d$, $\H P^d$, or $\Ca P^2$ (where we set $d=2$), 
such that $d+1$ is prime, and let $\K$ be a coefficient field of characteristic ${\rm char}(\K) = d+1$.
Then, for every Finsler metric $g$
on $M$, there exist positive constants $\lambda,C$ depending only on $g$ and $\K$ 
such that for all homogeneous elements $X \in H_*(\Lambda;\K) \setminus \{0\}$ and $x\in H^*(\Lambda;\K)\setminus \{0\}$ we have
$$
  |\lambda \cdot \deg(X) - \Cr(X)| \leq C \quad\text{and}\quad
  |\lambda \cdot \deg(x) - \ccr(x)| \leq C.
$$ }

Here $\Cr(X)$ and $\ccr(x)$ are the critical values of $X$ and $x$ with respect to $g$ (see~\S\ref{sec: res}). In fact, we prove this result more generally for arbitrary Reeb flows on the unit cotangent bundles of $\C P^d$, $\H P^d$ or $\Ca P^2$ (Theorems~\ref{thm: resonance-Liouville} and~\ref{thm: resonance CP1-Liouville}).

In~\S\ref{sec:density} we extend the density theorem of~\cite{Hingston-Rademacher} from spheres to other CROSS.
For a closed Reeb orbit $\gamma$ we consider its length $\ell(\gamma)$, its index $i(\gamma)$, its {\em average index} 
$\hat{i}(\gamma) := \lim_{k \to \infty} i(\gamma^k)/k$,
and its {\em mean frequency}
$\alpha(\gamma) := \hat{i}(\gamma)/\ell(\gamma)$.
We set $i=2,4,8$ for $M=\C P^d$, $\H P^d$ and $\Ca P^2$, respectively.
The following is a special case of Theorem~\ref{thm: dense} for Reeb flows on unit cotangent bundles. 

{\bf Theorem C. } {\sl  
Let $M$, $g$, $\bK$, $\lambda$ be as in Theorem B and define the {\em global mean frequency} $\alpha = \lambda^{-1}$. For $\eps > 0$ denote by $\cS_\eps$ the set of geometrically distinct simple closed geodesics $\gamma$ with mean frequency
$\alpha(\gamma) \in (\alpha-\eps, \alpha+\eps)$. Then
\[\sum_{\gamma\in\cS_\eps} \frac{1}{\hat{i}(\gamma)} \geq \frac{1}{n+i-2}.\] }

In~\S\ref{sec:finite_quotients_of_spheres} we prove that resonance and density theorems pass to finite quotients with suitable local coefficient systems (Theorem~\ref{thm: resonance-Liouville-quotient}).
Combining this with the results in~\cite{Hingston-Rademacher}, we obtain resonance and density theorems for finite quotients of spheres (Corollary~\ref{cor:density-quotients-spheres}). For real projective spaces $\R P^n$ these complement for $n$ odd the results of Stegemeyer~\cite{Stegemeyer} for $n$ even.

The two appendices contain the proof of an important lemma that characterizes critical levels, and the proof of the fact that the Kronecker pairing between loop cohomology and loop homology is determined by the Frobenius pairing on Rabinowitz loop homology $\wh H_*\Lambda$.

{\bf Acknowledgements. }
We wish to express our gratitude to Nancy Hingston, whose geometric insights and far reaching vision led to this paper. 
We also thank Mihai Damian, Hans-Bert Rademacher, Maximilian Stegemeyer, and Nathalie Wahl for fruitful discussions. 
K.C.~was partially supported by the Deutsche Forschungsgemeinschaft (DFG, German Research
Foundation) grant 517480394.
A.O.~was partially supported by ANR grant COSY 21-CE40-000 and by the University of Strasbourg Institute for Advanced Study (USIAS).
E.S.~was partially supported by an NSERC Discovery grant, the Courtois chair in fundamental research, a Fonds de Recherche du Qu\'ebec Nature et Technologies teams grant, and a Sloan research fellowship. 

\section{BV structures} \label{sec:BV}

We recall in this section following~\cite{CHO-PD} the BV algebra structure on Rabinowitz loop homology $\wh{\H}_*\Lambda$ and the formulation of Poincar\'e duality in this setting.

The $S^1$-action $r:S^1\times \Lambda\to\Lambda$ given by repara\-me\-trization at the source manifests itself algebraically in the form of the degree $+1$ \emph{Batalin-Vilkovisky, or BV, operator} 
$$
\Delta:\H_*\Lambda\to \H_{*+1}\Lambda,\qquad \Delta(A)=r_*([S^1]\times A).
$$
It was already observed in~\cite{CS} that, in the orientable case, the homology $\H_*\Lambda$ endowed with the Chas-Sullivan product and the operator $\Delta$ is a BV algebra. According to~\cite{Getzler-BV}, this means that $\Delta$ is a degree $1$ operator satisfying $\Delta^2=0$ and whose defect from being a derivation is a \emph{Gerstenhaber bracket}, i.e., the equation 
\begin{equation}\label{eq:Gerstenhaber-bracket-intro}
\{A,B\}=(-1)^{|A|}\Delta(AB) - (-1)^{|A|}(\Delta A)B - A \Delta B
\end{equation}
defines a degree $1$ Lie bracket which is a derivation in each variable with respect to the product. This means that  
\begin{equation} \label{eq:Gerst-0}
\{A,\{B,C\}\}=\{\{A,B\},C\} + (-1)^{(|A|+1)(|B|+1)}\{B,\{A,C\}\},
\end{equation}
\begin{equation} \label{eq:Gerst-1}
\{AB,C\}=A\{B,C\} + (-1)^{(|C|+1)|B|}\{A,C\}B,
\end{equation}
\begin{equation} \label{eq:Gerst-2}
\{A,BC\}=\{A,B\}C + (-1)^{|B|(|A|+1)}B\{A,C\}.
\end{equation}
Note that the degree $1$ skew-symmetry condition writes $\{A,B\}=-(-1)^{(|A|+1)(|B|+1)}\{B,A\}$. It was proved in~\cite{Getzler-BV} that the operation $\{\cdot,\cdot\}$ is a Gerstenhaber bracket iff the {\em 7-term relation} holds for the BV-operator:
\begin{align}\label{eq:7-term} 
\Delta(ABC) = & \, \Delta (AB) C + (-1)^{|A|}A\Delta(BC) + (-1)^{(|A|+1)|B|}B\Delta(AC) \cr
& - (\Delta A)BC - (-1)^{|A|}A (\Delta B) C - (-1)^{|A|+|B|}AB(\Delta C). 
\end{align}
Abouzaid extended in~\cite[\S10, \S11]{Abouzaid-cotangent} this result to the case of nonorientable manifolds and to certain local systems of coefficients. The outcome is a \emph{twisted BV structure}~\cite[\S10]{Abouzaid-cotangent}, meaning that 
in addition to the $\Z$-grading $|\cdot|$ one is also given a $\Z/2$-grading $w$ and the following hold: the product is \emph{twisted graded commutative}, i.e.,  $AB=(-1)^{|A||B| + w(A)w(B)}BA$, and the \emph{twisted 7-term relation} for the BV-operator $\Delta$ reads
\begin{align*}
\Delta& (ABC) \\
& = \Delta (AB) C + (-1)^{|A|}A\Delta(BC) + (-1)^{(|A|+1)|B| + w(A)w(B)}B\Delta(AC) \\
& \qquad - (\Delta A)BC - (-1)^{|A|}A (\Delta B) C - (-1)^{|A|+|B|}AB(\Delta C).
\end{align*}
The twisted 7-term relation is equivalent to the fact that the bilinear operation defined by~\eqref{eq:Gerstenhaber-bracket-intro} is a \emph{twisted Gerstenhaber bracket}, satisfying suitable derivation identities, see~\cite[Exercise~10.5.7]{Abouzaid-cotangent}.

The relevant notion of local system is the following. 

\begin{definition}[Abouzaid~\cite{Abouzaid-cotangent}]
A local system on $\Lambda$ is \emph{$S^1$-equivariant} if there exists an isomorphism between the two local systems on $S^1\times \Lambda$ obtained by pulling back via the projection on the second factor, respectively via the reparametrization action $r$. (The isomorphism is considered to be part of the data defining an $S^1$-equivariant local system.)
\end{definition} 

\begin{definition} A \emph{BV local system} on $\Lambda$ is an $S^1$-equivariant local system of the form $\nu\otimes \mu$, where $\mu=\ev_0^*|M|^{-1}$ and $\nu$ is a local system which is compatible with products in the sense of~\cite[Appendix A]{CHO-MorseFloerGH}.
\end{definition}

From a symplectic perspective the BV operator on symplectic homology is induced by the action on Hamiltonian Floer homology of the fundamental class of the moduli space of genus zero Riemann surfaces with two punctures, asymptotic markers at the punctures, and one marked point aligned with the first asymptotic marker (Seidel~\cite[\S8]{Seidel07}). This moduli space is canonically identified with the circle $S^1$. The outcome for Rabinowitz loop homology is the following result, whose proof differs only superficially from Abouzaid's construction in~\cite[\S10]{Abouzaid-cotangent} and is therefore omitted. See also~\cite{Latschev-Oancea} for a definition and study of BV bialgebras in this context.

\begin{theorem} Rabinowitz loop homology $\wh{\H}_*\Lambda$ admits a twisted BV algebra structure with coefficients in any BV local system, consisting of the primary pair-of-pants product and the canonical BV operator. 
\qed
\end{theorem}

We will also use the dual BV structures on the cohomology groups
$$
\H^*\Lambda=H^{*+n}\Lambda,\qquad \H^*M=H^{*+n}M,\qquad
\wh{\H}^*\Lambda=\wh{H}^{*+n}\Lambda.
$$

\begin{theorem} Rabinowitz loop cohomology $\wh{\H}^{1-2n-*}\Lambda$ admits a twisted BV algebra structure with coefficients in any BV local system. This BV algebra structure is given by the secondary pair-of-pants product and the canonical BV operator. 
\end{theorem}

\begin{proof}
  The BV operator together with the primary pair-of-pants coproduct defines a twisted co-BV structure at chain level in homology. The proof relies on the study of the boundary degenerations of certain parametrized Floer moduli problems~\cite{Abouzaid-cotangent}. The parameter spaces that govern the twisted co-BV relations for the secondary pair-of-pants coproduct are obtained from the previous ones by crossing with an interval. Since the boundary of the interval does not introduce any non-vanishing terms in the resulting chain level relations, it follows that the twisted co-BV relations for the secondary pair-of-pants coproduct hold as well.
  Equivalently, we obtain a twisted BV structure on cohomology with respect to the secondary pair-of-pants product. 
\end{proof}

\begin{theorem}\label{thm:PD}
For any BV local system the Poincar\'e duality isomorphism 
$$
PD:\wh{\H}_*\Lambda\stackrel\simeq \longrightarrow \wh{\H}^{1-2n-*}\Lambda
$$
is an isomorphism of twisted BV algebras.
\end{theorem}

\begin{proof}
The Poincar\'e duality isomorphism is given at chain level by a chain map $f$ whose cone is naturally a Floer complex. The compatibility between the Floer differential and the BV operator on the cone implies that the map $f$ commutes with the BV operators on the factors. This compatibility between the Floer differential and the BV operator is a general feature of Hamiltonian Floer homology. 
\end{proof}

Recall from~\cite[\S2.1]{CHO-PD} 
the definition of the action filtered homology groups $\wh{\H}_*^{(a,b)}\Lambda$. The BV operator decreases the action, and therefore induces BV operators on all $\wh{\H}_*^{(a,b)}\Lambda$. In applications, a very useful feature of the BV structure is the following. 

\begin{proposition} \label{prop:BV level-0}
The induced BV operator on the $0$-level homology $\wh{\H}_*^{=0}\Lambda$ vanishes (where $\wh{\H}_*^{=0}\Lambda=\wh{\H}_*^{(-\eps,\eps)}\Lambda$ for $\eps>0$ small enough).  
\end{proposition}

\begin{proof}
  Consider a Morse perturbation of the Hamiltonian defining $\wh{\H}_*\Lambda$ near its zero level. When the size of the perturbation is small enough, all Floer trajectories involved in the definition of the induced BV operator at level $0$ are contained in a collar neighbourhood of the unit cotangent bundle, where the Hamiltonian is a $C^2$-small Morse function. One can apply the argument of Salamon-Zehnder~\cite{SZ92} to an $S^1$-family of almost complex structures which is time independent
  in order to show that these Floer trajectories are actually independent of time, and therefore the corresponding moduli spaces are empty. The induced BV operator is thus zero. 
\end{proof}

Recall from~\cite[Theorem 1.5]{CHO-PD} that, for orientable manifolds of dimension $n\ge 3$,  Rabinowitz loop homology and cohomology admit splittings
$$
  \wh{\H}_*\Lambda = \ol{\H}_*\Lambda \oplus \ol{\H}^{1-2n-*}\Lambda,\qquad
  \wh{\H}^{1-2n-*}\Lambda = \ol{\H}^{1-2n-*}\Lambda \oplus \ol{\H}_*\Lambda
$$
such that the Poincar\'e duality isomorphism 
has a simple form, exchanging the two factors up to a contribution coming from a secondary continuation map, which vanishes if $H_1M=0$. 
Here $\ol{\H}_*\Lambda$ and $\ol{\H}^*\Lambda$ denote the {\em reduced loop (co)homology} defined in~\cite{CHO-reducedSH},  which differ from $\H_*\Lambda$ and $\H^*\Lambda$ by a contribution coming from a multiple $\chi(M)[\mathrm{pt}]$ of the point class.

\begin{theorem}
In the above splittings the reduced loop (co)homology groups $\ol{\H}_*\Lambda$ and $\ol{\H}^{1-2n-*}\Lambda$ are BV subalgebras.  
\end{theorem}

\begin{proof}
It is proved in~\cite{CHO-reducedSH} that the summands $\ol{\H}_*\Lambda$ and $\ol{\H}^{1-2n-*}\Lambda$ are stable under products. That thay are also stable under the BV operators follows from the existence of the BV operators on $\H_*\Lambda$ and $\H^{1-2n-*}\Lambda$ and Proposition~\ref{prop:BV level-0}.
\end{proof}

\begin{remark}
The dimensional assumption $n\ge 3$ is imposed in~\cite[Theorem 1.5]{CHO-PD} in order for the extended Goresky-Hingston product on $\ol{\H}^{1-2n-*}\Lambda$ to be coassociative~\cite[Theorem~6.4, Remark~6.5]{CHO-reducedSH}, and for the splitting to be a coalgebra map. The existence of the splitting, and its compatibility with the Goresky-Hingston product on $\H^{1-2n-*}(\Lambda,\Lambda_0)$, hold regardless of the dimension~\cite[Corollary~7.7]{CHO-reducedSH}. We will use this fact in order to compute the Rabinowitz loop homology of the 2-sphere.
\end{remark}

\section{Homology of manifolds all of whose geodesics are closed}  \label{sec:SC-manifolds}

We focus now on the case of manifolds $M$ that admit a metric with the following property: there exists $L>0$ such that all geodesics of length $L$ are simple and closed. In this case we call $M$ an \emph{SC-manifold}~\cite{Besse}, or a \emph{Zoll manifold}.\footnote{The initials ``SC'' stand for Simple Closed.} 
In this section we recall important results due to Bott-Samelson and Besse, we infer the level algebra structure for Rabinowitz loop homology, and we propose a conjecture that generalizes the Bott-Samelson theorem. 

\subsection{The Bott-Samelson theorem} \label{sec:Bott-Samelson} The literature on SC-manifolds is vast, and we refer to the book~\cite{Besse} by Arthur Besse for a comprehensive treatment. The fundamental examples are the compact rank one symmetric spaces (CROSS) $S^n$, $\R P^n$, $\C P^d$, $\H P^d$, and $\Ca P^2$,
endowed with their standard homogeneous metric. The following theorem due to Bott-Samelson is foundational for the field and shows that there is no homological distinction between SC-manifolds and CROSS (Bott~{\cite{Bott54}}, Samelson~{\cite{Samelson63}}, Besse~{\cite[Theorem~7.23 and Historical note 7.74]{Besse}}).

\begin{theorem}[Bott-Samelson] \label{thm:Bott-Samelson}
Any SC-manifold has the integral cohomology ring of a CROSS (which is uniquely determined and is called its \emph{model}).  
\end{theorem}

The existence of Zoll metrics on $S^2$ shows that the homogeneous metrics are far from being the only ones with the SC-property. 
However, there are no known examples of SC-metrics whose underlying manifold is not a CROSS, hence the question: 

\begin{question}[{\cite[Question~7.72]{Besse}}] Is every SC-manifold diffeomorphic to a CROSS?
\end{question}

\begin{remark}
The literature seems to focus on Blaschke manifolds, which are special cases of SC-manifolds. The Blaschke conjecture states that any Blaschke manifold is isometric to a CROSS, see~\cite[Chapter~5 and Appendix~D]{Besse}. As such, it is a statement that belongs to the realm of Riemannian geometry. In contrast, the above question by Arthur Besse belongs to the realm of differential topology. 
\end{remark}

Going back to the Bott-Samelson theorem, it is important to note that it actually proves much more using much less. The hypothesis of the theorem is that there exists a point $p$ in $M$ and a positive real $L>0$ such that all the geodesics starting at $p$ come back to $p$ at time $L$ and not before. There is no periodicity assumption for these geodesics, they could come back to $p$ at an angle. The conclusion of the theorem is that, denoting $\lambda\ge 0$ the index of one of the geodesics of length $L$ starting at $p$, all the geodesics starting at $p$ have the same index, and with $n=\dim M$ the index can take the following values~\cite[Theorem~7.23]{Besse}: 
\begin{enumerate}
\item $\lambda=0$, in which case $M$ has the homotopy type of $\R P^n$. If we require in addition that all geodesics of length $L$ starting from $p$ are periodic, then $M$ is diffeomorphic to $\R P^n$.   
\item $\lambda=1$, in which case $M$ has the homotopy type of $\C P^d$, $n=2d$.
\item $\lambda=3$, in which case $M$ has the integral cohomology ring of $\H P^d$, $n=4d$.
\item $\lambda=7$, in which case $M$ has the integral cohomology ring of $\Ca P^2$, $n=16$.
\item $\lambda=n-1$, $n\ge 2$,
in which case $M$ has the homotopy type of $S^n$. 
\end{enumerate}

In particular, if $\lambda>0$ the manifold $M$ is simply-connected (this is actually a preliminary step in the proof of the Bott-Samelson theorem). The proof presented in~\cite[7.29]{Besse} shows that the cohomology ring of $M$ is the same as the one of its model CROSS with arbitrary coefficients --- in particular, it has only one generator with arbitrary coefficients if the model CROSS is $S^n,\C P^d,\H P^d,\Ca P^2$, and with $\Z/2$-coefficients if the model CROSS is $\R P^n$.

The case $\lambda=0$, i.e. $M=\R P^n$, is the only non-simply connected case. The manifold $\R P^n$ is orientable if and only if $n$ is odd, if and only if its Euler characteristic is zero. In our discussion of real projective spaces we use $\Z/2$-coefficients, so the issue of orientability will be irrelevant.

Let $\K$ be a field of coefficients, assumed to be of characteristic $2$ if $\lambda=0$.
We have $H^*(M)=\K[h]/h^{d+1}$, where 
$d=n/i$, $n=\dim\, M$, $\deg(h)=i$, and 
$$
i=
\left\{\begin{array}{ll} 
n & \mbox{for } S^n,\\
1 & \mbox{for }\R P^n,\\
2 & \mbox{for }\C P^d,\\
4 & \mbox{for }\H P^d,\\
8 & \mbox{for }\Ca P^2.
\end{array}\right. 
$$
The relation between $\lambda$ and $i$ is 
$$
\lambda=i-1.
$$
Borrowing from the projective terminology, we call the generator $h$ the \emph{hyperplane class}. We denote $1=h^0$ the unit and $\omega=h^d$ the \emph{volume class}. (For $S^n$ we have $h=\omega$.) Let $PD_M:H^*(M)\to H_{n-*}(M)$ be the Poincar\'e duality isomorphism and set $H=PD_M(h)$, so that $\deg(H)=n-i$ and
$$
H_*(M)={\rm span}\{H^d=[\mathrm{pt}],H^{d-1},\dots, H, H^0=[M]=E\}.
$$
Here the powers $H^j$ are understood with respect to the intersection product on homology, with unit the fundamental class $E=[M]$. The class $H$ is represented by a projective hyperplane in the projective case, and by a point in the sphere case.

\subsection{Cohomology of the the unit cosphere bundle}  \label{sec:unit_sphere_bundle}

As a direct consequence of the Bott-Samelson theorem we obtain:

\begin{theorem}\label{thm:Bott-Samelson-S*M} 
The unit cosphere bundle of any SC-manifold has the integral cohomology ring of the unit cosphere bundle of its model CROSS.
\end{theorem}

\begin{proof}
The ring structure on the cohomology of the unit cosphere bundle of a simply connected smooth manifold
is determined by the ring structure on the cohomology of that manifold: this follows from the multiplicativity of the cohomological Leray-Serre spectral sequence, and the fact that there are no extensions involved if the manifold is simply connected. This proves the theorem for SC-manifolds whose model CROSS is simply connected because the Bott-Samelson theorem determines their homotopy type, and in particular the structure of their cohomology ring. 

The only non-simply connected SC-manifolds are the ones that have $\R P^n$ as their model CROSS. In that case the Bott-Samelson theorem determines their diffeomorphism type, hence the diffeomorphism type of their unit cosphere bundles, and the theorem follows as well.
\end{proof}

The cohomology $H^{-*}(S^*M)$ is relevant in our case because it is the action zero part of Rabinowitz loop homology $\wh{\H}_*\Lambda$. See also the discussion around Conjecture~\ref{conjecture:RFH-SC} below.

We will need in our applications the following explicit computations of the cohomology of unit cosphere bundles $S^*M$. 

(a) The case $\lambda>0$, or $\lambda=0$ and $n$ odd, with $\K$ a field of characteristic $p$, with  $p$ prime dividing $\chi(M)$ if $\lambda>0$ and $p=2$ if $\lambda=0$.
The Leray-Serre spectral sequence for the unit cosphere bundle $S^*M$ in cohomology with coefficients in $\K$ collapses at the second page. We obtain an isomorphism of algebras
$$
H^*(S^*M)\simeq H^*(M)\otimes H^*(S^{n-1}), 
$$
or 
$$
H^{-*}(S^*M)\simeq \H_{*}M\otimes H^{-*}(S^{n-1})
$$
after regrading and applying Poincar\'e duality.
Consider $H\in \H_{-i}M$ the dual of the hyperplane class and $\alpha\in H^{-(1-n)}(S^{n-1})$ the generator determined by the orientation.
We identify $H$ with $H\otimes 1\in \H_{-i}M\otimes H^0(S^{n-1})$ and $\alpha$ with $E\otimes\alpha=[M]\otimes \alpha\in \H_0M\otimes H^{-(1-n)}(S^{n-1})$.
Denote by $|x|=\deg(x)-n$ the shifted degree for generators. Then
$$
H^{-*}(S^*M)\simeq \K[H,\alpha]/\la H^{d+1},\alpha^2\ra,\qquad |H|=-i,\ |\alpha|=1-n.
$$ 
The table below summarizes these facts. The leftmost column indicates the shifted degrees (recall $id=n$).

{\tiny
$$
\xymatrix
@C=3pt
@R=3pt
{
0       &&&&      E         & &  & & & & \\
    &&&&             & &       &     & &        & \\
-i  &&&&  H            & \\
       &&&&  \vdots &  & \\ 
-i(d-1)      &&&&  H^{d-1} & &  & \\
1-id      & & & \alpha &  & &  & \\
-id=-n      &  &&   &  H^d       & & &        & \\
1-i(d+1) &  & & H \alpha  &  \\
&  & & \vdots   &   \\
1 - i(2d-1) & & & H^{d-1} \alpha & \\
& & & & \\
1 - 2id = 1-2n & & & H^d \alpha & 
}
$$
}

(b) The case $\lambda=0$ with $n$ even, with $\K$ a field of coefficients of characteristic $2$. 

The Euler characteristic is equal to $1$ and $\R P^n$ is $\K$-orientable. The only nontrivial differential on the second page $d_2:E_2^{0,n-1}\to E_2^{n,0}$ is an isomorphism and the spectral sequence collapses at the third page. Since $E_2^{*,*}\simeq H^*(M)\otimes H^*(S^{n-1})$, the cohomology $H^*(S^*M)$ can be expressed in an obvious way as a quotient vector space. As an algebra, we obtain  
$$
H^*(S^*M)\simeq \K[h,x]/\la h^n,x^2\ra, \qquad \deg(h)=1, \ \deg(x)=n,
$$
with $h\otimes 1$ the class of the generator of $E_2^{1,0}=H^1(M)\otimes H^0(S^{n-1})$ and $x=h\otimes\alpha$ the class of the generator of $E_2^{1,n-1}=H^1(M)\otimes H^{n-1}(S^{n-1})$. After regrading and applying Poincar\'e duality we get
$$
H^{-*}(S^*M)\simeq \K[H,\xi] / \la H^n,\xi^2\ra,\qquad |H|=-1,\ |\xi|=-n.
$$
The figure below summarizes these facts, with shifted degrees indicated in the leftmost column.

{\tiny
$$
\xymatrix
@C=3pt
@R=3pt
{
0       &&&&      E         & &  & & & & \\
    &&&&             & &       &     & &        & \\
-1  &&&&  H            & \\
       &&&&  \vdots &  & \\
1-n      &&&  &  H^{n-1} & &  & \\
-n      &  && \xi   &        & & &        & \\
-n-1 &  & & H \xi &  \\
&  & & \vdots   &   \\
1 - 2n & & & H^{n-1} \xi & 
}
$$
}


\subsection{Level algebra structure on Rabinowitz loop homology}  \label{sec:RFH-level}
We now discuss the level algebra structure on Rabinowitz loop homology $\wh{\H}_*\Lambda$ for SC-manifolds. By \emph{level algebra structure} we mean the following. We endow the manifold $M$ with its SC-metric and normalize it so that $L=1$. With an appropriate choice of the Hamiltonians used in the definition of Rabinowitz loop homology $\wh{H}_*\Lambda$, the critical levels of all the nonzero homology classes are elements of $\Z$. We denote 
$$
\wh{H}_*^{=k}\Lambda = \wh{H}_*^{(k-\eps,k+\eps)}\Lambda,
$$
for $\eps>0$ small, and 
$$
\Gr \wh{H}_*\Lambda = \bigoplus_{k\in\Z}\wh{H}_*^{=k}\Lambda.
$$
Recall from~\cite{CHO-PD} that the {\em critical level} 
\begin{equation}\label{eq:Cr}
\Cr(X) = \inf\{a\in\R\mid X\in\im(\wh{\H}_*^{(-\infty,a)}\Lambda\to\wh{\H}_*\Lambda)\}
\end{equation}
of $X\in\wh{\H}_*\Lambda$ satisfies the inequality 
$$
\Cr(X\cdot Y)\le \Cr(X)+\Cr(Y)
$$ 
Therefore we obtain an induced algebra structure on $\Gr \wh{H}_*\Lambda$, which we call the \emph{level algebra structure} on $\wh{H}_*\Lambda$. In view of the inequality 
$$
\Cr(\Delta(X))\le \Cr(X),
$$
we also obtain an induced BV algebra structure on $\Gr \wh{H}_*\Lambda$, which we call the \emph{level BV algebra structure}.

\begin{proposition} \label{prop:zero-level-block}
The \emph{zero-level block} $\wh{\H}_*^{=0}\Lambda$ is isomorphic as a BV algebra to $H^{-*}(S^*M)$ with trivial BV-operator.
\end{proposition} 

\begin{proof}
The vanishing of the BV operator was proved in Proposition~\ref{prop:BV level-0}. That the ring structure is given by that of $H^{-*}(S^*M)$ is the content of~\cite[Proposition~2.17]{CO}.
\end{proof}

Consider now an SC-manifold that has dimension $\ge 3$. Denote $y\in\wh{H}_*\Lambda$ the \emph{homological Uebele class}~\cite{Uebele-products}. 
The degree of $y$ is 
$$
\deg(y)=2n-1+\lambda,
$$
where $2n-1$ is the dimension of the Bott manifold of simple closed geodesics and $\lambda$ their common index. Its shifted degree is 
$$
|y|=\deg(y) - n = n+i-2.
$$
As a consequence, $\deg(y^k)=n+k(n+i-2)$ for all $k\in\Z$. 

The following is a reformulation of Uebele's theorem~\cite[Theorem~1.2]{Uebele-products} for the case in point. We work with coefficients in a field $\K$ of characteristic $2$ in the case $\lambda=0$ and $n$ even, and with arbitrary field coefficients otherwise. This is to ensure $\K$-orientability of the SC-manifold $M$.

\begin{theorem}[Level algebra structure, Uebele~{\cite[Theorem~1.2]{Uebele-products}}] \label{thm:level_algebra} Let $M$ be an SC-manifold of dimension $\ge 3$.
  

 (i) As a $\K$-vector space we have 
$$
\wh{\H}_*\Lambda \simeq \Gr \wh{\H}_*\Lambda.
$$

(ii) The $\K$-algebra $\Gr \wh{\H}_*\Lambda$ is isomorphic to 
$$
\K[y,y^{-1}]\otimes_\K H^{-*}(S^*M). 
$$
More precisely: 
\begin{itemize}
\item $\wh{\H}_*\Lambda$ and $\Gr \wh{\H}_*\Lambda$ are free $\K[y,y^{-1}]$-modules with respect to multiplication by the Uebele class $y$.
\item Any basis of the zero-level block $\wh{\H}_*^{=0}\Lambda$ provides a basis for this module and 
$$
\wh{\H}_*^{=k}\Lambda = y^k \wh{\H}_*^{=0}\Lambda. 
$$
\item The multiplication in $\Gr \wh{\H}_*\Lambda$ is given by 
$$
(y^k\alpha)\cdot (y^\ell\beta)=y^{k+\ell}(\alpha\beta)
$$
for $\alpha,\beta\in H^{-*}(S^*M)$. 
\end{itemize}
\end{theorem}

\begin{proof}
The assumption in Uebele's theorem is that the underlying manifold admits an SC-metric for which the Morse indices of contractible closed geodesics are $>3-n$. We need to check that this condition holds as soon as $n\ge 3$. The minimal Morse index for an SC-metric is $\lambda$. If $\lambda>0$ then the condition obviously holds. If $\lambda=0$, the minimal Morse index of a contractible closed geodesic is $n-1$, the same as for the sphere of the same dimension, and the condition holds again. 
\end{proof}
 
\begin{remark} Theorem~\ref{thm:level_algebra} also holds in dimension $2$. 
Although one cannot apply Uebele's theorem anymore,
this can be proved as a consequence of, on the one hand, structural results for loop homology and loop cohomology due to Goresky-Hingston, and, on the other hand, the fact that Rabinowitz loop homology splits as a direct sum of reduced loop homology and reduced loop cohomology. See Theorems~\ref{thm:Rabinowitz_BV_S2_char_does_not_divide} and~\ref{thm:Rabinowitz_BV_S2_char_2} below. 
\end{remark}

\subsection{A conjecture}
We would like to propose the following conjecture in relation with the Bott-Samelson theorem. 
\begin{conjecture} \label{conjecture:RFH-SC}
For any field of coefficients, an SC-manifold has the same Rabinowitz loop homology BV algebra as its model CROSS. 
\end{conjecture} 

By the discussion in~\S\ref{sec:Bott-Samelson}, the conjecture obviously holds for $\R P^n$ because there is only one diffeomorphism type of SC-manifold with this model CROSS. Also by~\S\ref{sec:Bott-Samelson}, we expect the conjecture to hold for complex projective spaces and spheres as a consequence of the conjectural homotopy invariance of the BV algebra structure on Rabinowitz loop homology of simply connected manifolds.
In the case of quaternionic projective spaces and in the case of the Cayley plane the conjecture supports the idea that any SC-manifold is homotopy equivalent to its model CROSS. 

Conjecture~\ref{conjecture:RFH-SC} is also supported by the following result. 

\begin{proposition}
For any field of coefficients, an SC-manifold
has the same level Rabinowitz loop homology algebra structure as its model CROSS. 
\end{proposition}

\begin{proof}
In dimension $\ge 3$, this is a direct consequence of Theorems~\ref{thm:Bott-Samelson-S*M} and~\ref{thm:level_algebra}: the algebra structure at energy zero is that of the cohomology of $S^*M$, and this determines the algebra structure on $\Gr \wh{\H}_*\Lambda$ through the isomorphism 
$\Gr \wh{\H}_*\Lambda\simeq 
\K[y,y^{-1}]\otimes_\K H^{-*}(S^*M)$. 

In dimension $2$ there are two cases to consider. If the model CROSS is $\R P^2$, the SC-manifold is diffeomorphic to it and there is nothing to prove. If the model CROSS is $S^2$, this follows from Theorem~\ref{thm:Bott-Samelson-S*M} and the proofs of Theorems~\ref{thm:Rabinowitz_BV_S2_char_does_not_divide} and~\ref{thm:Rabinowitz_BV_S2_char_2}.
\end{proof}

\section{Loop BV algebras for $\C P^d$, $\H P^d$, $\Ca P^2$}  \label{sec:BV-CHCa}

In this section we compute the loop homology BV algebra, Rabinowitz loop homology BV algebra, and loop cohomology BV algebra for these CROSS, with arbitrary field coefficients. We rely on the computation of the loop homology BV algebra with $\Z$-coefficients by Menichi~\cite{Menichi} for $\C P^1=S^2$ and $\H P^1=S^4$, by Hepworth~\cite{Hepworth} and Chataur-Le Borgne~\cite{Chataur-Le_Borgne} for $\C P^d$, and by \v{C}adek-Moravec~\cite{Cadek-Moravec} for $\H P^d$ and $\Ca P^2$.  

Ultimately, the reason to discuss Rabinowitz loop homology is to infer the BV structure on loop cohomology from Uebele's theorem and the BV structure in loop homology. 

\subsection{Loop homology BV algebra with field coefficients}


Let $M$ be one of the CROSS $\C P^d$ or $\H P^d$ with $d\ge 1$, or $\Ca P^2$. We set by convention $i=2,4,8$ for $\C P^d$, $\H P^d$, and $\Ca P^2$ respectively, and also $d=2$ for $\Ca P^2$. The dimension of $M$ is $n=id$.  

\subsubsection{Integral coefficients} The structure of the loop 
homology BV algebra with integral coefficients has been 
computed by Menichi, Hepworth, Chataur-Le Borgne, \v{C}adek-Moravec.

\begin{theorem}[Menichi~{\cite[Theorems~17 and~25]{Menichi}}, Hepworth~{\cite[Theorem~A]{Hepworth}}, Chataur-Le Borgne~{\cite[Theorem~0.1]{Chataur-Le_Borgne}}, \v{C}adek-Moravec~{\cite[Theorem~1.1]{Cadek-Moravec}}] 
\label{thm:loop_BV_CROSS_over_Z}
Let $M$ be one of the CROSS $\C P^d$, $\H P^d$, or $\Ca P^2$. We have an isomorphism of BV algebras 
$$
\H_*(\Lambda;\Z)\cong \Z[a,x,y]/\la a^{d+1},x^2,xa^d, (d+1) ya^d\ra,
$$
where $|a|=-i$, $|x|=-1$, $|y|=i(d+1)-2$, and the BV operator acts by
$$
\Delta(y^k a^\ell)=0
$$
for $k\ge 0$ and $0\le \ell\le d$, 
\begin{equation} \label{eq:DeltaykaxellHPdOP2}
\Delta(y^k x a^\ell) = (-\ell + k + (k+1)d)y^k a^\ell
\end{equation}
for $k\ge 0$ and $0\le \ell\le d$ in the case of $\H P^d$ and $\Ca P^2$, respectively for $k\ge 0$ and $1\le \ell\le d$ in the case of $\C P^d$, and
\begin{equation}
  \Delta(y^k x)=(k+(k+1)d)y^k + \frac{(d+1)d}{2} y^{k+1}a^d
\end{equation}
for $\C P^d$. (Note that the term $\frac{(d+1)d}{2} y^{k+1}a^d$ vanishes if $d$ is even.) \qed
\end{theorem}

\begin{remark} These formulas directly specialize to those in~\cite{Menichi} for $\C P^1$ and $\H P^1$, and to those in~\cite{Cadek-Moravec} for $\H P^d$ and $\Ca P^2$. They also specialize directly to those in~\cite{Hepworth} in the case $d$ even. In the case $d$ odd they are equivalent to those in~\cite{Hepworth} up to replacing the generator $y$ by $y+\frac{d+1}2y^2a^d$. The formulas in~\cite{Hepworth} and~\cite{Chataur-Le_Borgne} are equivalent up to changing the sign of the degree $-1$ generator. 
\end{remark}

\begin{remark}[Geometric generators] \label{rmk:geometric-generators} Goresky-Hingston~\cite{Goresky-Hingston} have given a presentation of the ring structure using generators that are described geometrically from Morse theory. When the CROSS are endowed with the homogeneous metric, all the geodesics are closed and the energy functional on their free loop space is perfect~\cite{Ziller1977}. The critical manifolds other than the constant loops are diffeomorphic to the unit sphere bundle $SM$, since each geodesic is uniquely determined by its speed and starting point. Suitable cycles in $SM$ give rise to cycles in the homology of the free loop space via their descending manifolds.  Specifically, we consider the following homology classes. 
\begin{itemize}
\item $H$ is the class of a projective hyperplane, represented by constant loops.
  It has shifted degree $|H|=-i$. 
\item $\Theta$ is the class represented by the descending manifold of the entire unit sphere bundle, seen as the first critical Morse-Bott level with nonzero energy for the homogeneous metric.
  It has shifted degree $|\Theta|=id-1+\lambda=n+i-2$. 
\item  $X$ is the class represented by the descending manifold of a section of this unit sphere bundle restricted to a projective hyperplane.
  It has shifted degree $|X|=-1$.
\end{itemize}
Then $\H_*(\Lambda;\Z)\cong \Z[H,X,\Theta]/\la H^{d+1},X^2,XH^d, (d+1) \Theta H^d\ra$. 

In the case of $\H P^d$ and $\Ca P^2$ the loop homology is, in each degree, either free of rank 1, or $(d+1)$-torsion isomorphic to $\Z/(d+1)\Z$. The generators $H,X,\Theta$ live in torsion-free degrees, just like the generators $a,x,y$ from Theorem~\ref{thm:loop_BV_CROSS_over_Z}, and therefore coincide with the latter up to sign. As a consequence, Theorem~\ref{thm:loop_BV_CROSS_over_Z} gives the action of the BV operator on the geometric generators up to sign. 

In the case of $\C P^d$ the homology is free of rank $1$ in all degrees other than $k(n+i-2)=|y^k|=|\Theta^k|$, $k\ge 1$, and in these special degrees it has a free part generated by $y^k$, and a $(d+1)$-torsion part generated by $y^{k+1}a^d$. As a consequence, the geometric generators $H,X$, which live in torsion-free degrees, coincide up to sign with $a,x$, whereas the geometric generator $\Theta$ coincides up to sign with $y+\beta y^2a^d$ for some $\beta\in\Z/(d+1)\Z$. We do not know how to calculate the constant $\beta$. As a consequence, we do not know how to express the BV operator in terms of the geometric generators (specifically, we cannot calculate it on classes $\Theta X$).  
\end{remark}

\begin{remark}[Non-uniqueness of the generators for $\C P^d$]
Let $a,x,y$ be the generators from Theorem~\ref{thm:loop_BV_CROSS_over_Z} for $\C P^d$. The map $\Psi:\H_*(\Lambda;\Z)\to \H_*(\Lambda;\Z)$, $\Psi(a)=a$, $\Psi(x)=x$, $\Psi(y)=y+y^2a^d$ is an algebra automorphism: we have $\Psi^s(y)=y+sy^2a^d$ for any $s\in\Z$ and $\Psi^{d+1}=\mathrm{Id}$. The generator $y$ can therefore be traded for any other element of the form $y^{(s)}=\Psi^s(y)=y+sy^2a^d$ for $s\in\Z/(d+1)\Z$. Note that $\Delta(x)=-d+\frac{d+1}2y^{(s)}a^d$ for any $s$, but the coefficients in the formula for $\Delta((y^{(s)})^kx)$ depend on $s$ for $k\ge 1$.   
\end{remark}

\begin{remark}[Coefficients in a field $\K$ whose characteristic divides $d+1$] \label{rmk:geometric-generators-U}
In this case there is an additional geometric generator described as follows: 
\begin{itemize}
\item $U$ is the class represented by the descending manifold of a singular section of the unit sphere bundle. Such a section, and also the corresponding descending manifold, define cycles with $\K$-coefficients. The shifted degree of $U$ is $|U|=i-1$. 
\end{itemize} 
By restricting the section to a hyperplane we find the relation $X=UA$. 
\end{remark}

\vfill \pagebreak 

\begin{center}
TABLE FOR $\H_*(\Lambda;\Z)$ IN THE CASE OF $\H P^d$ AND $\Ca P^2$ \\
$(*)$ stands for torsion, $\times$ represents a homology group that is zero with $\Z$-coefficients, but nonzero with torsion coefficients, see below
\end{center}
{\tiny
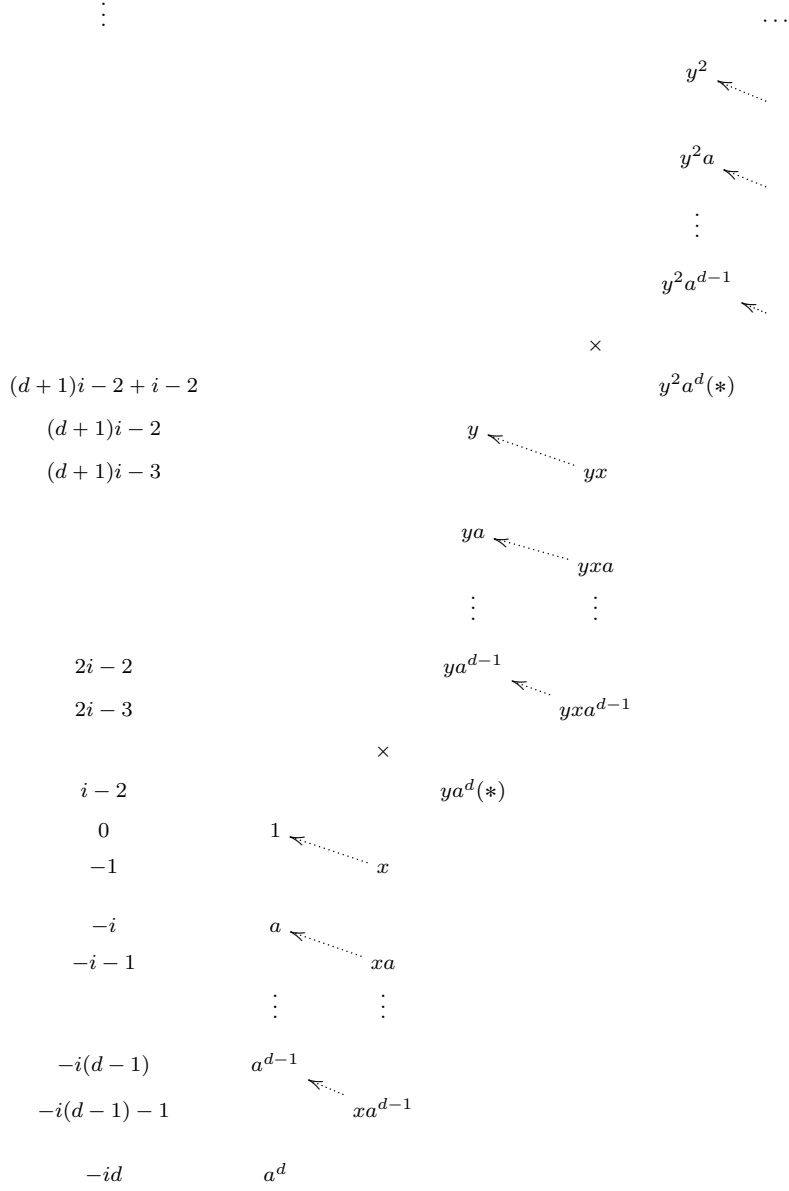
$$
\xymatrix
@C=3pt
@R=1pt
{
\vdots                &&                          & &           &                            & &                  &                   & \cdots \\
                         &&                          & &           &                             & &                  &   y^2          &          \\
                         &&                          & &           &                             & &                  &                   & \ar@{.>}[ul] \\
                         &&                          & &           &                             & &                  &                   & \\
                         &&                          & &           &                             & &                  & y^2 a         & \\
                         &&                          & &           &                             & &                  &                    &  \ar@{.>}[ul] \\
                         &&                          & &           &                             & &                  & \vdots          & \\
                         &&                          & &           &                             & &                  & y^2 a^{d-1} &  \\
                         &&                          & &           &                             & &                  &                     &      \ar@{.>}[ul]   \\
                         &&                          & &           &                             & & \times       &                      &        \\
(d+1)i-2+i-2       &&                          & &           &                             & &                  & y^2 a^d (*)  & \\
(d+1)i-2             &&                          & &           & y                          & &                  &                      &        \\
(d+1)i-3             &&                          & &           &                             & & yx \ar@{.>}[ull] &    \\
                         &&                          & &           &                             & &  &    \\
                         &&                          & &           & ya                        & &  & \\
                         &&                          & &           &                             & & yxa \ar@{.>}[ull] & \\
                         &&                          & &           & \vdots                   & & \vdots &  \\
2i-2                   &&                          & &           & ya^{d-1}                & &  & \\
2i-3                   &&                          & &           &                               & & yxa^{d-1}\ar@{.>}[ull] & \\
                         &&                          & & \times &                                & & & \\
i-2                     &&                          & &            & ya^d (*)                  & & & \\
0                       &&      1                  & &            &                              & & & \\
-1                      &&                          & &  x \ar@{.>}[ull]   & & & & \\
                         &&                          & &                            & & & & \\
-i                       &&      a                  & &                              & & & & \\
-i-1                    &&                          & &  xa \ar@{.>}[ull]    & & & & \\
                         &&  \vdots               & &  \vdots                   & & & & \\
-i(d-1)                &&  a^{d-1}            & &  &                             & & & \\
-i(d-1)-1            &&                           & & xa^{d-1} \ar@{.>}[ull] & & & & \\
                         &&                          & & &                          & & & & \\
 -id                    &&  a^d                   & & &                          & & & & 
}
$$
}

\vfill \pagebreak 

\begin{center}
TABLE FOR $\H_*(\Lambda;\Z)$ IN THE CASE OF $\C P^d$ \\
$(*)$ stands for torsion, $\times$ represents a homology group that is zero with $\Z$-coefficients, but nonzero with torsion coefficients, see below
\end{center}

{\tiny
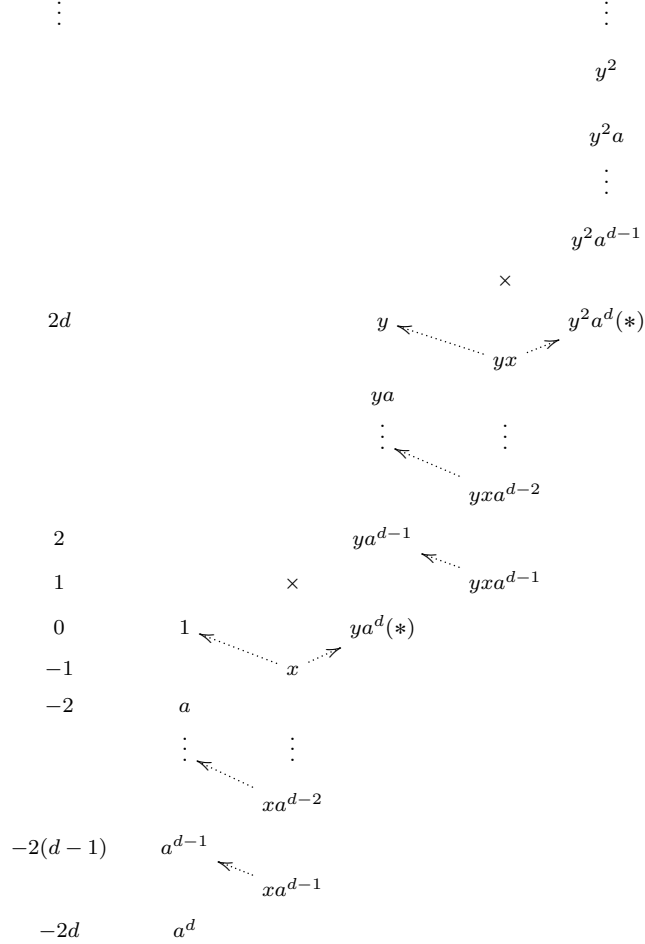
$$
\xymatrix
@C=3pt
@R=1pt
{
\vdots                &&                          & & &                            		& &                  & \vdots \\
                         &&                          & & &                             		& &                  &   y^2 \\
                         &&                          & & &                             		& &                  &    \\
                         &&                          & & &                             		& &                  & y^2a   \\
                         &&                          & & &                             		& &                  & \vdots   \\
                         &&                          & & &                             		& &                  & y^2a^{d-1}   \\
                         &&                          & & &                             		& & \times       &                              \\
2d                     &&                          & & & y                          		& &                  & y^2a^d (*)\\
                         &&                          & & &                             		& & yx \ar@{.>}[ull] \ar@{.>}[ur] &    \\
                         &&                          & & & y a                       		& &  & \\
                         &&                          & & &  \vdots                      		& & \vdots & \\                         
                         &&                          & & & 		                  		& & yx a^{d-2} \ar@{.>}[ull] &  \\
2                        &&                          & & & y a^{d-1}             		& &  & \\
1                        &&                          & & \times &                 		& & yx a^{d-1}\ar@{.>}[ull] & \\
0                       &&      1                   & &  & y a^d (*)              		& & & \\
-1                      &&                           & &  x \ar@{.>}[ull] \ar@{.>}[ur] &  & & & \\
-2                      &&      a                  & &  &          				& & & \\
                        &&       \vdots           & &  \vdots      &          			& & & \\
                         &&                          & &  x a^{d-2} \ar@{.>}[ull]&          & & & \\
-2(d-1)                &&  a^{d-1}             & & &                          			& & & \\
                         &&                           & &  xa^{d-1}\ar@{.>}[ull] & 	& & & \\
 -2d                    &&  a^d                   & & &                          			& & & \\
}
$$
}

We now compute the loop homology BV algebra with field coefficients.

For any field $\K$ we have an inclusion $\H_*(\Lambda;\Z)\otimes_\Z \K\hookrightarrow \H_*(\Lambda;\K)$, which is a morphism of BV algebras~\cite[Proposition~9]{Menichi}. Here the term $\H_*(\Lambda;\Z)\otimes_\Z \K$ inherits its BV-structure from $\H_*(\Lambda;\Z)$. 

\subsubsection{Coefficients in a field $\K$ whose characteristic does not divide the Euler characteristic $\chi(M)$} 
Let $\K$ be a field whose characteristic does not divide $d+1$ (in particular, characteristic $0$ is allowed). Then $\mathrm{Tor}^\Z(\H_*(\Lambda;\Z),\K)=0$ and, by the universal coefficient theorem, the inclusion $\H_*(\Lambda;\Z)\otimes_\Z \K\hookrightarrow \H_*(\Lambda;\K)$ is an isomorphism. We therefore obtain from Theorem~\ref{thm:loop_BV_CROSS_over_Z}:

\begin{corollary} 
\label{cor:loop_BV_CROSS_over_K_char_does_not_divide_d+1}
Let $M$ be one of the CROSS $\C P^d$, $\H P^d$, or $\Ca P^2$, and let $\K$ be a field whose characteristic does not divide $d+1$. We have an isomorphism of BV algebras 
$$
\H_*(\Lambda;\K)\cong \K[a,x,y]/\la a^{d+1},x^2,xa^d, ya^d\ra,
$$
where $|a|=-i$, $|x|=-1$, $|y|=i(d+1)-2$, and the BV operator acts by
$$
\Delta(y^k a^\ell)=0,\qquad \Delta(y^k x a^\ell) = (-\ell + k + (k+1)d)y^k a^\ell,
$$
for $k\ge 0$ and $0\le \ell\le d$. \qed
\end{corollary}

\subsubsection{Coefficients in a field $\K$ whose characteristic divides the Euler characteristic $\chi(M)$}
The case $\C P^1=S^2$ is special and was computed by Menichi with $\Z/2$-coefficients. 

\begin{theorem}[{Menichi~\cite[Theorem~24]{Menichi}}] 
\label{thm:Menichi_CP1}
Let $M=\C P^1=S^2$ and let $\K$ be a field of characteristic $2$. We have an isomorphism of BV algebras 
$$
\H_*(\Lambda;\K)=\K[a,u]/\langle a^2\rangle,
$$
where $|a|=-2$, $|u|=1$, and the BV operator acts by 
$$
\Delta (u^k)=0
$$
for $k\ge 0$, and 
$$
\Delta (au^k)=\left\{\begin{array}{rl} 0, & k \mbox{ even}, \\ u^{k-1}+au^{k+1}, & k \mbox{ odd}. \end{array} \right.
$$\qed
\end{theorem}

We now compute the remaining cases. 
\begin{theorem}
\label{thm:loop_BV_CROSS_over_K_char_divides_d+1}
Let $M$ be one of the CROSS $\C P^d$ with $d\ge 2$, $\H P^d$ with $d\ge 1$, or $\Ca P^2$, and let $\K$ be a field whose characteristic divides $d+1$. 
We have an isomorphism of BV algebras 
$$
\H_*(\Lambda;\K)\cong \K[a,u,y]/\la a^{d+1},u^2\ra,
$$
where $|a|=-i$, $|u|=i-1$, $|y|=i(d+1)-2$, and the BV operator acts by
$$
\Delta(y^k a^\ell)=0, \qquad \Delta(y^ku)=0,
$$
for $k\ge 0$ and $0\le \ell\le d$, 
\begin{equation}
\Delta(y^k u a^\ell) = \ell y^k a^{\ell-1}
\end{equation}
for $k\ge 0$ and $1\le \ell\le d$ in the case of $\H P^d$ and $\Ca P^2$, respectively for $k\ge 0$ and $2\le \ell\le d$ in the case of $\C P^d$, and
\begin{equation}
  \Delta(y^k ua) = y^k + \frac{(d+1)d}{2} y^{k+1}a^d = 
  \left\{\begin{array}{cl} y^k+\frac{d+1}2y^{k+1}a^d, & d \mbox{ odd,}\\
  y^k,& d \mbox{ even}
  \end{array}\right.
\end{equation}
for $\C P^d$.
\end{theorem}

\begin{proof}
Let $a',x',y'$ be the generators from Theorem~\ref{thm:loop_BV_CROSS_over_Z}, and let $\bar a, \bar x, \bar y$ be their images in $\H_*(\Lambda;\Z)\otimes_\Z\K\hookrightarrow \H_*(\Lambda;\K)$. Note that $\H_*(\Lambda;\Z)\otimes_\Z\K=\K[\bar a, \bar x,\bar y]/\langle {\bar a}^{d+1}, {\bar x}^2, \bar x{\bar a}^{d}\rangle$ (the table of multiplication can be visualized from the one of $\H_*(\Lambda;\Z)$, with the only difference that $\bar y{\bar a}^d$ is not torsion anymore).

Additively, the homology $\H_*(\Lambda;\K)$ differs from $\H_*(\Lambda;\Z)\otimes_\Z\K$ by one extra generator in each degree $k(i-1)$, $k\ge 1$. This is a consequence of Ziller's classical argument that the energy functional for $\Lambda$ is perfect for any choice of coefficients~\cite{Ziller1977}. (In the formulation of~\cite[Definition~6.1]{Oancea-Morse}, this is implied by the fact that all the critical Morse-Bott manifolds possess strong completing manifolds. See also~\cite[Definition~3.1]{Hingston-Oancea}.) The Morse-Bott critical manifolds that correspond to nonconstant geodesics are diffeomorphic to $STM$, the unit sphere tangent bundle of $M$, which has Euler class $d+1$. As a consequence, if one uses coefficients in a field $\K$ of characteristic dividing $d+1$, $H_*(STM;\K)$ is the same as for a trivial bundle and features one extra generator compared to $H_*(STM;\Z)\otimes_\Z\K$.  

The multiplicative structure on $\H_*(\Lambda;\K)$ can be computed using the Cohen-Jones-Yan spectral sequence~\cite{Cohen-Jones-Yan}. This is the Leray-Serre spectral sequence $E^2_{*,*}=\H_*(M;H_*(\Omega M))\Rightarrow \H_*(\Lambda)$ for the loop-loop fibration, and is multiplicative with respect to the intersection product on $\H_*(M)$, the Pontryagin product on $H_*(\Omega M)$, and the loop product on $\H_*(\Lambda)$. We have $H_*(\Omega M;\K)\simeq H_*(S^{i-1};\K)\otimes_\K H_*(\Omega S^{id+i-1};\K)\simeq \K[u,y]/\langle u^2\rangle$, $|u|=i-1$, $|y|=id+i-2$, see~\cite[Proof of Theorem~3]{Cohen-Jones-Yan} and~\cite[Lemmas~2.1 and~2.2]{Cadek-Moravec}. (The Hopf fibration $S^{i-1}\to S^{id+i-1}\to M$ determines the homotopy cofiber sequence $S^{id+i-1}\to M\to \Sigma S^{i-1}$, and further $\Omega S^{id+i-1}\to \Omega M\to \Omega \Sigma S^{i-1}\simeq S^{i-1}$.) Taking into account the additive structure of $\H_*(\Lambda;\K)$ given by Ziller's theorem, one infers that the Cohen-Jones-Yan spectral sequence collapses at the second page and $\H_*(\Lambda;\K)\simeq \H_*(M)\otimes_\K \K[u,y]/\langle u^2\rangle = \K[a,u,y]/\langle a^{d+1},u^2\rangle$.

We claim that we can choose the generators $a$, $u$, $y$ such that $a=\bar a$, $y=\bar y$, and $u\bar a=\bar x$. 
For $a$, this holds because the homology is $1$-dimensional in degree $-id=|a|=|\bar a|$, therefore $a=\bar a$ up to multiplication by a nonzero scalar. For $u$, this holds because the homology is $1$-dimensional in degree $-1=|\bar x|=|u\bar a|$, and we can achieve $\bar x=u\bar a$ by multiplying $u$ with a nonzero scalar. We now discuss $y$: in the case of $\H P^d$ or $\Ca P^2$ the homology is $1$-dimensional in degree $|y|=|\bar y|$, hence $y=\bar y$ up to multiplication by a nonzero scalar; in the case of $\C P^d$ the homology is $2$-dimensional in degree $2d=|y|=|\bar y|$, generated by $\bar y$ and ${\bar y}^2a^d$, and also by $y$ and $y^2a^d$. The generator $y$ cannot be a multiple of ${\bar y}^2a^d$, otherwise $y^2a^d=0$. Therefore $y=\alpha \bar y+\beta {\bar y}^2a^d$ with $\alpha\in\K^*$, $\beta\in\K$, and we can replace $y$ by $\bar y$ without altering the relations. 

Since the map $\H_*(\Lambda;\Z)\otimes_\Z\K\hookrightarrow \H_*(\Lambda;\K)$ is a morphism of BV algebras~\cite[Proposition~9]{Menichi}, we directly infer by reduction of coefficients from Theorem~\ref{thm:loop_BV_CROSS_over_Z}, with $a=\bar a$, $y=\bar y$, and $\bar x=u\bar a$, that
$$
\Delta(y^ka^\ell)=0
$$
for $k\ge 0$ and $0\le \ell\le d$, 
$$
\Delta(y^kua^\ell)=-\ell y^k a^{\ell-1}
$$
for $k\ge 0$ and $1\le \ell \le d$ in the case of $\H P^d$ and $\Ca P^2$, respectively for $k\ge 0$ and $2\le \ell\le d$ in the case of $\C P^d$ (this follows from~\eqref{eq:DeltaykaxellHPdOP2} using that $d=-1$ in $\K$), and 
$$
\Delta(y^kua)=-y^k+ \frac{(d+1)d}{2}y^{k+1}a^d
$$
for $\C P^d$. 

We are now left to compute $\Delta (y^ku)$ for $k\ge 0$. By the 7-term relation and knowing that $u^2=0$, it is enough to compute $\Delta(u)$ and $\Delta (yu)$. 

{\it The case of $\H P^d$ with $d\ge 1$, and $\Ca P^2$}. The homology $\H_*(\Lambda;\K)$ vanishes in degrees $|\Delta(u)|=i$ and $|\Delta(yu)|$, therefore $\Delta(u)=0$ and $\Delta(yu)=0$. 

{\it The case of $\C P^d$ with $d\ge 2$}. We claim that it is possible to choose the generator $u$ so that $\Delta(u)=0$. Indeed, the homology $\H_*(\Lambda;\K)$ is $1$-dimensional in degree $2=|\Delta(u)|$, generated by $ya^{d-1}$, so that $\Delta(u)=\beta ya^{d-1}$ for some $\beta \in\K$. Since $\Delta(yua^d)=-dya^{d-1}=ya^{d-1}$, up to replacing $u$ by $u-\beta yua^d$ we find that $\Delta(u)=0$. (This argument fails for $\C P^1$, since the homology is $2$-dimensional in degree $|\Delta(u)|$.)

To compute $\Delta(yu)$ we expand $\Delta(yau)$ using the 7-term relation~\eqref{eq:7-term}:
\begin{align*}
-y+\frac{(d+1)d}{2}y^2a^d= \Delta(yau) & = \cancel{\Delta(ya)}u + y\Delta(au)+ a\Delta(yu) \\
& \quad -\cancel{\Delta(y)}au -y\cancel{\Delta(a)}u-ya\cancel{\Delta(u)}\\
& = y(-1+\frac{(d+1)d}{2} ya^d) + a\Delta(yu). 
\end{align*}
This implies $a\Delta(yu)=0$.
On the other hand, because $\H_*(\Lambda;\K)$ is $1$-dimensional in degree $|\Delta(yu)|$ and generated by $y^2a^{d-1}$, we must have $\Delta(yu)=\beta y^2a^{d-1}$ for some $\beta\in\K$. We then find $\beta y^2a^d=0$, which implies $\beta=0$ since $y^2a^d\neq 0$, and therefore $\Delta(yu)=0$. 

The relations $\Delta(u)=0$ and $\Delta(yu)=0$, together with $\Delta(y)=0$ and $\Delta(y^2)=0$, imply by induction using the 7-term relation~\eqref{eq:7-term} that $\Delta(y^ku)=0$ for all $k\ge 0$.
\end{proof}

\vfill\pagebreak

\begin{center}
TABLE FOR $\H_*(\Lambda;\K)$, $\mathrm{char}(\K)\mid d+1$ IN THE CASE OF $\H P^d$ AND $\Ca P^2$ 
\end{center}

{\tiny
$$
\xymatrix
@C=3pt
@R=1pt
{
\vdots                &&                          & &           &                            & &                  &                   & \cdots \\
                         &&                          & &           &                             & &                  &   y^2          &          \\
                         &&                          & &           &                             & &                  &                   & \ar@{.>}[ul] \\
                         &&                          & &           &                             & &                  &                   & \\
                         &&                          & &           &                             & &                  & y^2 a         & \\
                         &&                          & &           &                             & &                  &                    &  \ar@{.>}[ul] \\
                         &&                          & &           &                             & &                  & \vdots          & \\
                         &&                          & &           &                             & &                  & y^2 a^{d-1} &  \\
                         &&                          & &           &                             & &                  &                     &      \ar@{.>}[ul]   \\
                         &&                          & &           &                             & & yu             &                      &        \\
(d+1)i-2+i-2       &&                          & &           &                             & &                  & y^2 a^d        & \\
(d+1)i-2             &&                          & &           & y                          & &                  &                      &        \\
(d+1)i-3             &&                          & &           &                             & & yua \ar@{.>}[ull] &    \\
                         &&                          & &           &                             & &  &    \\
                         &&                          & &           & ya                        & &  & \\
                         &&                          & &           &                             & & yua^2 \ar@{.>}[ull] & \\
                         &&                          & &           & \vdots                   & & \vdots &  \\
2i-2                   &&                          & &           & ya^{d-1}                & &  & \\
2i-3                   &&                          & &           &                               & & yua^d\ar@{.>}[ull] & \\
                         &&                          & & u	     &                                & & & \\
i-2                     &&                          & &            & ya^d                      & & & \\
0                       &&      1                  & &            &                              & & & \\
-1                      &&                          & &  ua \ar@{.>}[ull]   & & & & \\
                         &&                          & &                            & & & & \\
-i                       &&      a                  & &                              & & & & \\
-i-1                    &&                          & &  ua^2 \ar@{.>}[ull]    & & & & \\
                         &&  \vdots               & &  \vdots                   & & & & \\
-i(d-1)                &&  a^{d-1}            & &  &                             & & & \\
-i(d-1)-1            &&                           & & ua^d \ar@{.>}[ull] & & & & \\
                         &&                          & & &                          & & & & \\
 -id                    &&  a^d                   & & &                          & & & & 
}
$$
}

\vfill \pagebreak

\begin{center}
TABLE FOR $\H_*(\Lambda;\K)$, $\mathrm{char}(\K)\mid d+1$ IN THE CASE OF $\C P^d$ 
\end{center}

{\tiny
$$
\xymatrix
@C=3pt
@R=1pt
{
\vdots                &&                          & & &                            		& &                  & \vdots \\
                         &&                          & & &                             		& &                  &   y^2 \\
                         &&                          & & &                             		& &                  &    \\
                         &&                          & & &                             		& &                  & y^2a   \\
                         &&                          & & &                             		& &                  & \vdots   \\
                         &&                          & & &                             		& &                  & y^2a^{d-1}   \\
                         &&                          & & &                             		& & yu	       &                              \\
2d                     &&                          & & & y                          		& &                  & y^2a^d        \\
                         &&                          & & &                             		& & yua \ar@{.>}[ull] \ar@{.>}[ur] &    \\
                         &&                          & & & y a                       		& &  & \\
                         &&                          & & &  \vdots                      		& & \vdots & \\                         
                         &&                          & & & 		                  		& & yu a^{d-1} \ar@{.>}[ull] &  \\
2                        &&                          & & & y a^{d-1}             		& &  & \\
1                        &&                          & & u  &                 			& & yu a^d\ar@{.>}[ull] & \\
0                       &&      1                   & &  & y a^d                 		& & & \\
-1                      &&                           & &  ua \ar@{.>}[ull] \ar@{.>}[ur] &  & & & \\
-2                      &&      a                  & &  &          				& & & \\
                        &&       \vdots           & &  \vdots      &          			& & & \\
                         &&                          & &  u a^{d-1} \ar@{.>}[ull]&          & & & \\
-2(d-1)                &&  a^{d-1}             & & &                          			& & & \\
                         &&                           & &  ua^d\ar@{.>}[ull] & 	& & & \\
 -2d                    &&  a^d                   & & &                          			& & & \\
}
$$
}

\subsection{Rabinowitz loop homology BV algebra}\label{sec:Rab-BV}

\subsubsection{BV structure of Rabinowitz loop homology for the simply connected CROSS that are projective spaces, except $\C P^1$} 

\begin{theorem} \label{thm:Rabinowitz_BV_projective_spaces_char_divides}
  Let $M$ be one of the CROSS $\C P^d$ with $d\ge 2$, $\H P^d$ with $d\ge 1$, or $\Ca P^2$ (we set $d=2$ in this last case), and let $\K$ be a field whose characteristic divides $d+1$. 
We have an isomorphism of BV algebras 
$$
\wh{\H}_*(\Lambda;\K)\cong \K[a,u,y,y^{-1}]/\la a^{d+1},u^2\ra,
$$
where $|a|=-i$, $|u|=i-1$, $|y|=i(d+1)-2$, and the BV operator acts by
$$
\Delta(y^k a^\ell)=0, \qquad \Delta(y^ku)=0,
$$
for $k\in\Z$ and $0\le \ell\le d$, 
\begin{equation}
\Delta(y^k u a^\ell) = \ell y^k a^{\ell-1}
\end{equation}
for $k\in\Z$ and $1\le \ell\le d$ in the case of $\H P^d$ and $\Ca P^2$, respectively for $k\in\Z$ and $2\le \ell\le d$ in the case of $\C P^d$, and
\begin{equation}
  \Delta(y^k ua) = y^k + \frac{(d+1)d}{2} y^{k+1}a^d = 
  \left\{\begin{array}{cl} y^k+\frac{d+1}2y^{k+1}a^d, & d \mbox{ odd,}\\
  y^k,& d \mbox{ even}
  \end{array}\right.
\end{equation}
for $\C P^d$ and $k\in\Z$.
%
%
%
%
%
\end{theorem}

\begin{proof}
The multiplicative structure is a consequence of Uebele's theorem~\cite[Theorem~1.2]{Uebele-products}, which can be applied because $\dim M\ge 3$. In \emph{loc. cit.} the author proves that the algebra $\wh{\H}_*(\Lambda;\K)$ is a free module over $\K[\Theta,\Theta^{-1}]$, where $\Theta$ is the \emph{homological Uebele class} and lives in degree $|\Theta|=|y|$. We claim that $\wh{\H}_*(\Lambda;\K)$ is also a free module over $\K[y,y^{-1}]$. 
\begin{itemize}
\item In the case of $\H P^d$ and $\Ca P^2$, the homology is $1$-dimensional in that degree and therefore $\Theta=y$ up to multiplication by a nonzero scalar, the class $y$ is invertible in $\wh{\H}_*(\Lambda;\K)$, and the claim follows. 
\item In the case of $\C P^d$ the homology is 2-dimensional in that degree, generated by $y$ and $y^2a^d$. The class $\Theta$ cannot be a multiple of $y^2a^d$ because $(y^2a^d)^2=0$, therefore $\Theta =\beta y + \gamma y^2a^d$ with $\beta,\gamma\in\K$ and $\beta\neq 0$. Up to multiplying $\Theta$ with a nonzero scalar we can assume without loss of generality that $\Theta=y+\gamma y^2a^d$ for some $\gamma\in\K$. Then $\Theta^2=y^2+2\gamma y^3a^d$, hence $\Theta^2 a^d=y^2a^d$, and further $\Theta a^d=\Theta^{-1}y^2a^d$. We then claim that $y$ is invertible with inverse $y^{-1}=\Theta^{-1}+\gamma a^d$: indeed $(\Theta^{-1}+\gamma a^d)y=(\Theta^{-1}+\gamma a^d)(\Theta-\gamma y^2a^d)=1+\gamma\Theta a^d-\gamma\Theta^{-1}y^2a^d=1$. We also have $y=\Theta-\gamma\Theta^2a^d$, and this implies that $\wh{\H}_*(\Lambda;\K)$ is also a free module over $\K[y,y^{-1}]$. 
\end{itemize}

The formulas for $k\in\Z$ are  proved by decreasing induction on $k$ starting from the same formulas that are valid for $k\ge 0$, by applying the 7-term relation to $ya^\ell=(y^ka^\ell)y^{-k}y$, $yu=(y^ku)y^{-k}y$, and $yua^\ell=(y^kua^\ell)y^{-k}y$. For example, for $k<0$: 
\begin{align*}
\Delta(yua^\ell) & =\Delta((y^kua^\ell)y^{-k}y) \\ 
& = \Delta(ua^\ell)y-y^kua^\ell\cancel{\Delta(y^{-k}y)}+y^{-k}\Delta(y^{k+1}ua^\ell) \\
& \quad - \Delta(y^kua^\ell)y^{-k+1}+ y^kua^\ell\cancel{\Delta(y^{-k})}y + y^kua^\ell y^{-k}\cancel{\Delta(y)}, 
\end{align*}
so that $\Delta(y^kua^\ell)=(\Delta(ua^\ell)y+y^{-k}\Delta(y^{k+1}ua^\ell)-\Delta(yua^\ell))y^{k-1}$, and this relation allows to conclude by decreasing induction. 
\end{proof}

In a very similar way one computes the Rabinowitz loop homology BV algebra in the case where the characteristic of the ground field does not divide $d+1$. 

\begin{theorem} 
\label{thm:Rabinowitz_BV_projective_spaces_char_does_not_divide}
Let $M$ be one of the CROSS $\C P^d$ with $d\ge 2$, $\H P^d$ with $d\ge 1$, or $\Ca P^2$ (we set $d=2$ in this last case), and let $\K$ be a field whose characteristic does not divide $d+1$. 
We have an isomorphism of BV algebras 
$$
\wh{\H}_*(\Lambda;\K)\cong \K[a,x,y,y^{-1}]/\la a^d,x^2\ra,
$$
where $|a|=-i$, $|x|=-1$, $|y|=i(d+1)-2$, and the BV operator acts by
$$
\Delta(y^k a^\ell)=0,\qquad \Delta(y^k x a^\ell) = (-\ell + k + (k+1)d)y^k a^\ell,
$$
for $k\in\Z$ and $0\le \ell\le d$. \qed
\end{theorem}


\begin{center}
TABLE FOR $\wh{\H}_*(\Lambda;\K)$, $\mathrm{char}(\K)\mid d+1$ IN THE CASE OF $\H P^d$ AND $\Ca P^2$ 
\end{center}

{\tiny
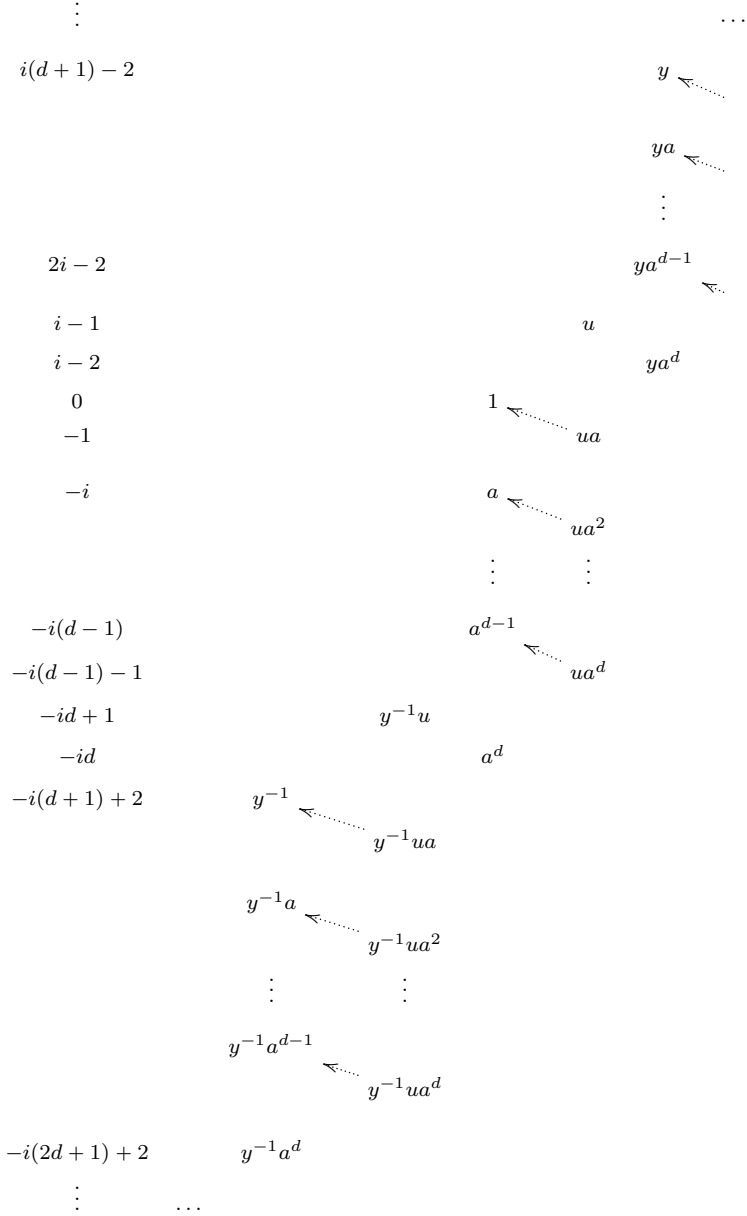
$$
\xymatrix
@C=3pt
@R=.5pt
{
\vdots                &&                          & &           &                            & &                  &                   & \cdots \\
i(d+1)-2             &&                          & &           &                             & &                  &   y              &          \\
                         &&                          & &           &                             & &                  &                   & \ar@{.>}[ul] \\
                         &&                          & &           &                             & &                  &                   & \\
                         &&                          & &           &                             & &                  & ya              & \\
                         &&                          & &           &                             & &                  &                    &  \ar@{.>}[ul] \\
                         &&                          & &           &                             & &                  & \vdots          & \\
2i-2                   &&                          & &           &                             & &                  & y a^{d-1}     &  \\
                         &&                          & &           &                             & &                  &                     &      \ar@{.>}[ul]   \\
i-1                     &&                          & &           &                             & &  u             &                      &        \\
i-2                     &&                          & &           &                             & &                  & y a^d           & \\
0                         &&                          & &           & 1                          & &                  &                      &        \\
-1                        &&                          & &           &                             & &  ua \ar@{.>}[ull] &    \\
                         &&                          & &           &                             & &  &    \\
-i                       &&                          & &           &  a                        & &  & \\
                         &&                          & &           &                             & &  ua^2 \ar@{.>}[ull] & \\
                         &&                          & &           & \vdots                   & & \vdots &  \\
-i(d-1)               &&                          & &           &  a^{d-1}                & &  & \\
-i(d-1)-1            &&                          & &           &                               & &  ua^d\ar@{.>}[ull] & \\
-id+1                         &&                          & & y^{-1}u&                                & & & \\
-id                     &&                          & &            &  a^d                      & & & \\
-i(d+1)+2           &&      y^{-1}           & &            &                              & & & \\
                        &&                          & &  y^{-1}ua \ar@{.>}[ull]   & & & & \\
                         &&                          & &                            & & & & \\
                         &&      y^{-1}a         & &                              & & & & \\
                         &&                          & &  y^{-1}ua^2 \ar@{.>}[ull]    & & & & \\
                         &&  \vdots               & &  \vdots                   & & & & \\
                         &&  y^{-1}a^{d-1}    & &  &                             & & & \\
                         &&                           & & y^{-1}ua^d \ar@{.>}[ull] & & & & \\
                         &&                          & & &                          & & & & \\
-i(2d+1)+2        &&  y^{-1}a^d          & & &                          & & & & \\
\vdots               &\dots&
}
$$
}


\begin{center}
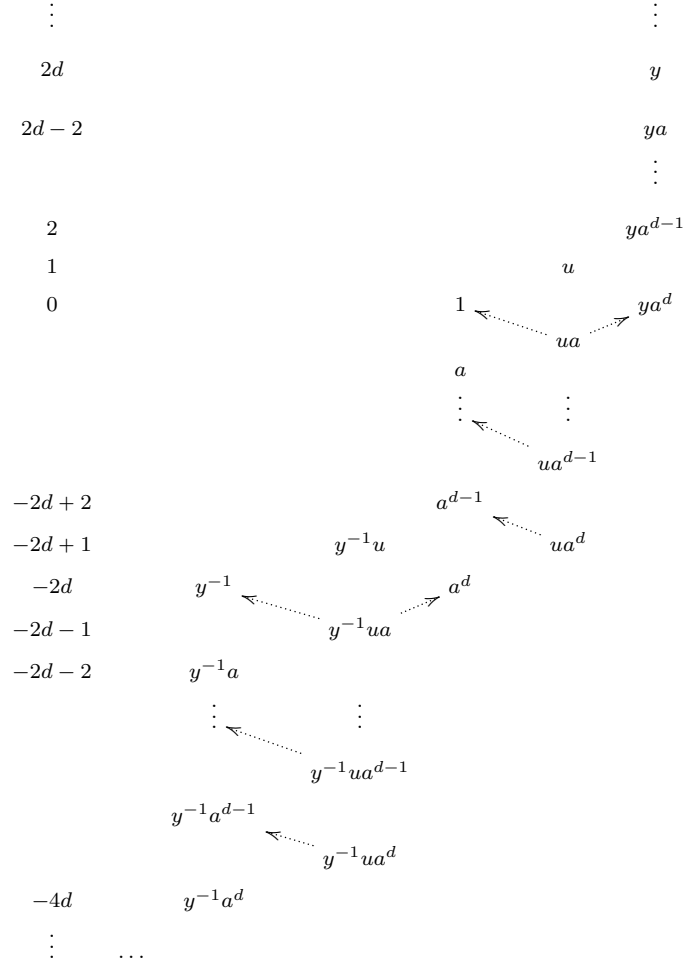
TABLE FOR $\wh{\H}_*(\Lambda;\K)$, $\mathrm{char}(\K)\mid d+1$ IN THE CASE OF $\C P^d$, $d\ge 2$ 
\end{center}

{\tiny
$$
\xymatrix
@C=3pt
@R=.5pt
{
\vdots                &&                          & & &                            		& &                  & \vdots \\
 2d                    &&                          & & &                             		& &                  &   y \\
                         &&                          & & &                             		& &                  &    \\
 2d-2                 &&                          & & &                             		& &                  & ya   \\
                         &&                          & & &                             		& &                  & \vdots   \\
 2                      &&                          & & &                             		& &                  & ya^{d-1}   \\
1                       &&                          & & &                             		& & u	       &                              \\
0                       &&                          & & & 1                          		& &                  & ya^d        \\
                         &&                          & & &                             		& & ua \ar@{.>}[ull] \ar@{.>}[ur] &    \\
                         &&                          & & &  a                         		& &  & \\
                         &&                          & & &  \vdots                      		& & \vdots & \\                         
                         &&                          & & & 		                  		& & u a^{d-1} \ar@{.>}[ull] &  \\
-2d+2               &&                          & & & a^{d-1}               		& &  & \\
-2d+1                 &&                          & & y^{-1}u  &                 			& & u a^d\ar@{.>}[ull] & \\
-2d                    &&      y^{-1}            & &  & a^d                   		& & & \\
-2d-1                 &&                           & &  y^{-1}ua \ar@{.>}[ull] \ar@{.>}[ur] &  & & & \\
-2d-2                &&      y^{-1}a          & &  &          				& & & \\
                        &&       \vdots           & &  \vdots      &          			& & & \\
                         &&                          & &  y^{-1}u a^{d-1} \ar@{.>}[ull]&          & & & \\
                        &&  y^{-1}a^{d-1}   & & &                          			& & & \\
                         &&                           & &  y^{-1}ua^d\ar@{.>}[ull] & 	& & & \\
-4d                    &&  y^{-1}a^d          & & &                          			& & & \\
\vdots               &\dots&
}
$$
}

\subsubsection{BV structure of Rabinowitz loop homology for $\C P^1$} 

\qquad


The case $M=\C P^1=S^2$ is special because we cannot directly apply Uebele's theorem in order to determine the multiplicative structure of $\wh\H_*\Lambda$ from that of $\H_*\Lambda$. We first discuss coefficients in a field of characteristic different from $2=\chi(S^2)$, and then coefficients in a field of characteristic $2$.

The structure claimed by the next theorem is the same as the one obtained by formally extending the formulas from Theorem~\ref{thm:Rabinowitz_BV_projective_spaces_char_does_not_divide} to the case of $\C P^1$, but it is of course convenient to state it separately. 
 
\begin{theorem} 
\label{thm:Rabinowitz_BV_S2_char_does_not_divide}
Let $M=\C P^1$, and let $\K$ be a field of characteristic different from $2$. 
We have an isomorphism of BV algebras 
$$
\wh{\H}_*(\Lambda;\K)\cong \K[x,y,y^{-1}]/\la x^2\ra,
$$
where $|x|=-1$, $|y|=2$, and the BV operator acts by
\begin{equation} \label{eq:BV_S2_char_does_not_divide}
\Delta(y^k)=0, \qquad \Delta(y^k x) = (2k + 1)y^k,
\end{equation}
for $k\in\Z$. 
\end{theorem}

\begin{samepage}
\begin{center}
TABLE FOR $\wh{\H}_*(\Lambda;\K)$, $\mathrm{char}(\K)\neq 2$ IN THE CASE OF $\C P^1$ 
\end{center}
{\tiny
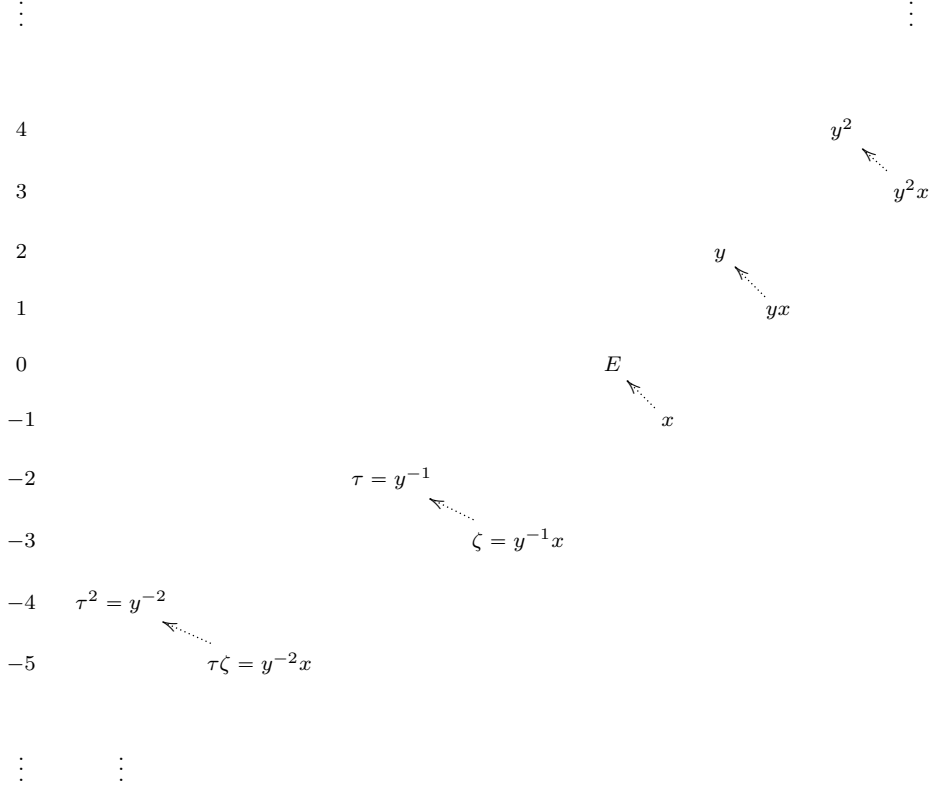
$$
\xymatrix
@C=8pt
@R=8pt
{
\vdots                && &                          &           			          &      & 		     &   &                 	  &	      & \vdots \\
                         && &                          &           			          &      & 		     &   &                  	  &	      &   	    \\
4                       && &                          &           			          &      & 		     &   &                   	  &   y^2 & 	    \\
3                       && &                          &           			          &      & 		     &   &                   	  &	      & y^2x \ar@{.>}[ul] \\
2                       && &                          &           			          &      & 		     & y&                  	  &	      &  	    \\
1                       && &                          &           			          &      & 		     &        & yx \ar@{.>}[ul]   &	      &      	    \\
0                       && &                          &           			          &  E  & 		     &        &                           &	      &       	    \\
-1                      && & 		         &           			          &      & x \ar@{.>}[ul] &                     && \\
-2                      && & \tau=y^{-1}     &          				  &       & 		     &                   &                     &&        \\
-3                      && &                          &\zeta=y^{-1}x \ar@{.>}[ul] &       & 		     &  & \\
-4  & \tau^2=y^{-2} & &                          &           &      & 		     &  & \\
-5  &			        &\tau\zeta=y^{-2}x \ar@{.>}[ul] &                          &           &      & 		     &  & \\
                         && &                          &           &      & 		     &  & \\
\vdots & \vdots  &&                          &           &      & 		     &  &  \\
}
$$
}
\end{samepage}

\begin{proof} We refer to the table for the notation used in the proof of the theorem. 

By the Splitting Theorem~\cite[Theorem~1.5]{CHO-PD}, which holds at an additive level regardless of dimension restrictions, we have a splitting 
$$
\wh H_*\Lambda= \ol H_*\Lambda \oplus \ol H^{1-*}\Lambda,
$$ 
with $\ol H_*\Lambda= H_*\Lambda/\langle [\mathrm{pt}]\rangle$ the reduced loop homology group, and $\ol H^*\Lambda = H^*\Lambda/\langle 1 \rangle$ the reduced loop cohomology group. This splitting is canonical because $H_1(\C P^1)=0$. 

The following maps intertwine the ring structures and the BV operators: 
\begin{itemize}
\item The inclusion $\ol H_*\Lambda \hookrightarrow \wh H_*\Lambda$. The Chas-Sullivan BV algebra structure on $\ol H_*\Lambda$ is induced from the one of $H_*\Lambda$ via the quotient map $H_*\Lambda\to \ol H_*\Lambda$. The unit is the fundamental class $[M]$, and we sometimes denote it $E$. 
\item The inclusion $H^{1-*}(\Lambda,\Lambda_0)\hookrightarrow \wh H_*\Lambda$, which can also be seen as the composition of the two inclusions $H^{1-*}(\Lambda,\Lambda_0)\hookrightarrow \ol H^{1-*}\Lambda\hookrightarrow \wh H_*\Lambda$. The product on $H^{1-*}(\Lambda,\Lambda_0)$ is the Goresky-Hingston product, and the cohomological BV operator $\Delta^*$ is the dual of the homological BV operator $\Delta$.\footnote{That $\Delta^*$ is dual to $\Delta$ means  
$\langle \Delta^*\alpha,A\rangle = (-1)^{|\alpha|}\langle \alpha,\Delta A\rangle$ or, equivalently, $\langle A,\Delta^*\alpha\rangle = (-1)^{|A|}\langle \Delta A,\alpha\rangle$, for $\alpha\in H^*(\Lambda,\Lambda_0)$ and $A\in H_*(\Lambda,\Lambda_0)$, with $|A|=\deg(A)-n$ and $|\alpha|=\deg(\alpha)+n-1$. 
}
\end{itemize}

The free loop space energy functional for the round metric on $\C P^1=S^2$ is perfect for both homology and cohomology. As a consequence, $H_*\Lambda$, $H^*\Lambda$ and $\wh H_*\Lambda$ have rank $1$ in every degree. This follows from the knowledge of the indices of the critical manifolds $SM$ and the fact that, in characteristic different from $2$, every such critical manifold contributes to the (co)homology in only two degrees, equal to its index and to its $\mathrm{index}+3$.  

We consider the following homology and cohomology classes. 
\begin{itemize}
\item Let $x\in H_1\Lambda$ and $y\in H_4\Lambda$ be generators as provided by Corollary~\ref{cor:loop_BV_CROSS_over_K_char_does_not_divide_d+1}. 
The generator $y$ is uniquely determined up to a constant, and we take $y$ to coincide with the geometric generator $\Theta$ of Goresky and Hingston (see Remark~\ref{rmk:geometric-generators}). The generator $x$ is uniquely determined by the relation $\Delta(x)=E$,
and it coincides up to sign with $X$, the geometric generator of Goresky and Hingston described in Remark~\ref{rmk:geometric-generators}, since both are defined over $\Z$. Let $\epsilon\in\{\pm1\}$ be this sign, so that $x=\epsilon X$.
\item Let $\tau\in H^1\Lambda$ be the algebraic dual of $x$, and let $\mu\in H^4\Lambda$ be the algebraic dual of $y$. Since $x$ coincides up to sign with the geometric generator $X$ of Goresky and Hingston, $\tau$ coincides up to sign with the geometric generator $\omega\in H^1\Lambda$ of Goresky and Hingston from~\cite[Theorem~14.2]{Goresky-Hingston}, i.e., $\tau=\epsilon \omega$.
\end{itemize} 

As a consequence of Corollary~\ref{cor:loop_BV_CROSS_over_K_char_does_not_divide_d+1}, we have an isomorphism of BV algebras 
$$
\ol H_*\Lambda = \K[x,y]/\langle x^2\rangle,
$$ 
with BV operator given by the formulas~\eqref{eq:BV_S2_char_does_not_divide}. 

As a consequence of~\cite[Theorem~14.2]{Goresky-Hingston}, we have an isomorphism of rings\footnote{Here $\K[\mu,\tau]_{\ge 1}$ means that we only consider polynomials of degree $\ge 1$. The relation $\mu^2=0$ also follows from graded commutativity of the Goresky-Hingston product.} 
$$
H^*(\Lambda,\Lambda_0)=\K[\mu,\tau]_{\ge 1}/\langle \mu^2\rangle.
$$
Indeed, the product with $\tau$ is injective on $H^*(\Lambda,\Lambda_0)$, and we conclude because $H^*(\Lambda,\Lambda_0)$ has rank $1$ in every degree. 

Multiplication by $\omega=\epsilon \tau$ and by $y$ are dual, see~\cite[Lemma~5.4]{Hingston-Rademacher}, in the sense that 
$$
\langle \alpha\omega^m,Ay^m\rangle = \langle \alpha,A\rangle, \quad \mbox{hence} \quad \langle \alpha\tau^m,Ay^m\rangle = \epsilon^m \langle \alpha,A\rangle,
$$
for the Kronecker pairing $\langle\cdot,\cdot\rangle$ and any (co)homology classes $\alpha$ and $A$. The following two particular cases are relevant for us: the relation $\langle \tau^m,xy^{m-1}\rangle =\epsilon^{m-1}\langle \tau,x\rangle=\epsilon^{m-1}$ implies that $\tau^m$ is algebraically dual to $\epsilon^{m-1}xy^{m-1}$, and the relation $\langle \mu\tau^m,y^{m+1}\rangle=\epsilon^m\langle\mu,y\rangle =\epsilon^m$ implies that $\mu\tau^m$ is algebraically dual to $\epsilon^m y^{m+1}$. 

\begin{remark} {\it These considerations allow us to directly compute the BV operator $\Delta^*:H^*(\Lambda,\Lambda_0)\to H^{*-1}(\Lambda,\Lambda_0)$ from the duality relation $\langle \Delta^*\alpha,A\rangle=(-1)^{|\alpha|}\langle \alpha,\Delta A\rangle$. (The expression of $\Delta^*$ is also a consequence of the full BV structure of $\wh \H_*\Lambda$ determined below.) 

We find 
\begin{align*}
\langle \Delta^*(\mu\tau^m),xy^{m+1}\rangle& =-\langle \mu\tau^m,\Delta(xy^{m+1})\rangle\\ 
& = -\langle \mu\tau^m,(2m+3)y^{m+1}\rangle = (-2m-3)\epsilon^m,
\end{align*} 
and $\langle \Delta^*(\tau^m),y^{m-2}\rangle =\langle \tau^m,\Delta(y^{m-2})\rangle = 0$. Since $\langle \tau^{m+2},xy^{m+1}\rangle =\epsilon^{m+1}$, we obtain 
$$
\Delta^*(\mu\tau^m)= \epsilon (-2m-3)\tau^{m+2},\qquad \Delta^*(\tau^m)=0.
$$
}
\end{remark}

To describe $\wh\H_*\Lambda$ we need one more cohomology class. 
\begin{itemize}
\item Let $\zeta\in H^2\Lambda_0=\ol H^2\Lambda=H^2\Lambda$ be the algebraic dual of the fundamental class $[M]$. 
\end{itemize} 

We then have additively\footnote{We also already know the BV operator on $\ol \H_*\Lambda = \langle [M]\rangle \oplus \K[y,x]_{\ge 1}/\langle x^2\rangle$, and on $H^{1-*}(\Lambda,\Lambda_0) = \K[\mu,\tau]_{\ge 1}/\langle \mu^2\rangle$. To calculate the BV operator on $\wh \H_*\Lambda$ we will only use knowledge of the BV operator on $\ol \H_*\Lambda$.}
\begin{align*}
\wh H_*\Lambda & = \ol H^{1-*}\Lambda \oplus \ol H_*\Lambda \\ 
& = \K[\mu,\tau]_{\ge 1}/\langle \mu^2\rangle \oplus \langle \zeta,[M]\rangle \oplus \K[y,x]_{\ge 1}/\langle x^2\rangle.
\end{align*}

We now compute the multiplicative structure on $\wh \H_*\Lambda$. We use the fact that $\wh \H_*\Lambda$ is a topological Frobenius algebra~\cite{CHO-PD}. The counit $\boldeps$ has degree $2n-1=3$ and therefore can only act nontrivially on $\wh \H_{-3}\Lambda = \wh H_{-1}\Lambda = H^2\Lambda =\langle \zeta\rangle$. 

Let $\boldmu$ be the product on $\wh \H_*\Lambda$, and $\boldp=-\boldeps\boldmu$ be the pairing on $\wh \H_*\Lambda$. The pairing $\boldp$ restricted to $\ol H^*\Lambda\oplus \ol H_*\Lambda$ coincides with the Kronecker product (Proposition~\ref{prop:Kronecker-Frobenius}). 
As a consequence, we find that 
$$
\boldeps(\zeta)=-1,
$$
because $\boldeps(\zeta)=\boldeps\boldmu(\zeta,[M])=-\boldp(\zeta,[M])=-\langle\zeta,[M]\rangle=-1$. 
This implies 
$$
\tau x=\zeta.
$$ 
Indeed, $\tau x$ lives in the degree of $\zeta$, where $\wh \H_*\Lambda$ has rank $1$, and $\boldeps(\tau x)=\boldeps\boldmu(\tau,x)=-\boldp(\tau,x)=-\langle\tau,x\rangle=-1$. 
This implies 
$$
\tau y=\epsilon [M].
$$
Indeed, $\tau y$ lives in the degree of $[M]$, where $\wh \H_*\Lambda$ has rank $1$, and $\langle \zeta,\tau y\rangle = -\boldeps\boldmu(\tau x,\tau y) = -\boldeps\boldmu(\tau^2,yx)=\boldp(\tau^2,yx)=\langle\tau^2,yx\rangle =\epsilon$. Here the second equality follows from associativity and graded commutativity of $\boldmu$. 

Similarly we find 
$$
\tau\zeta=\epsilon \mu.
$$ 
Indeed, $\tau\zeta$ lives in the degree of $\mu$, where $\wh \H_*\Lambda$ has rank $1$, and $\langle \tau\zeta,y\rangle =\langle \zeta,\tau y\rangle =\epsilon$. 

These calculations show that $y$ is invertible in $\wh \H_*\Lambda$ with 
$$
y^{-1}=\epsilon \tau=\omega, 
$$
and we finally obtain the isomorphism of rings
$$
\wh \H_*\Lambda =\K[x,y^{\pm 1}]/\langle x^2\rangle.
$$ 

To conclude, one still needs to establish the formulas~\eqref{eq:BV_S2_char_does_not_divide} for the BV operator on $\wh \H_*\Lambda$. This is done as in the proof of Theorem~\ref{thm:Rabinowitz_BV_projective_spaces_char_divides}, by decreasing induction on $k$, using the 7-term relation and the fact that these formulas are known for $k\ge 0$. (Note that the formulas~\eqref{eq:BV_S2_char_does_not_divide} are also compatible with our previous computation of $\Delta^*$, in view of $\epsilon\tau=y^{-1}$ and $\epsilon\mu=\tau\zeta=\tau^2x=y^{-2}x$.) 
\end{proof}

The next theorem describes the BV algebra structure on $\wh\H_*\Lambda$ with coefficients in a field of characteristic $2$. It should be compared to Theorem~\ref{thm:Menichi_CP1}. 

\begin{theorem} 
\label{thm:Rabinowitz_BV_S2_char_2}
Let $M=\C P^1=S^2$ and let $\K$ be a field of characteristic $2$. We have an isomorphism of BV algebras 
$$
\wh\H_*(\Lambda;\K)=\K[a,u^{\pm 1}]/\langle a^2\rangle,
$$
where $|a|=-2$, $|u|=1$, and the BV operator acts, for $k\in\Z$, by 
$$
\Delta (u^k)=0
$$
and
$$
\Delta (au^k)=\left\{\begin{array}{rl} 0, & k \mbox{ even}, \\ u^{k-1}+au^{k+1}, & k \mbox{ odd}. \end{array} \right.
$$
\end{theorem}

\begin{samepage}
\begin{center}
TABLE FOR $\wh{\H}_*(\Lambda;\K)$, $\mathrm{char}(\K)= 2$ IN THE CASE OF $\C P^1$ 
\end{center}
{\tiny
$$
\xymatrix
@C=6pt
@R=6pt
{
\vdots                && &                          &           			          &      & 		     &   		&                 	  	  &\vdots    & &  \\
                         && &                          &           			          &      & 		     &   		&                  	  	  &	          & & 	    \\
4                       && &                          &           			          &      & 		     &   		&              		     	  &y^2=\tilde u^4 & &	    \\
3                       && &                          &           			          &      & 		     &   		&         \tilde u^3.         &	      &  \ar@{.>}[ul] \\
2                       && &                          &           			          &      & 		     & y=\tilde u^2&                  	  	  & \tilde u^4a     &  	    \\
1                       && \tilde u^{-1}+\eps'\tilde u a \ar@{=}[ddrr]& 	&           			          &      & 	\tilde u   &        		& \tilde u^3a \ar@{.>}[ul] \ar@{.>}[ur]  &	      &      	    \\
0                       &\tilde u^{-2}+\eps' a \ar@{=}[ddrr] & &                          &           			          &  E  & 		     &  \tilde u^2a &                           &	      &       	    \\
-1                      && & 		         &           	1			  &       & \tilde u a \ar@{.>}[ul] \ar@{.>}[ur] &                     && \\
-2                      && & \tau	        &          				  &  a   & 		     &                   &                     &&        \\
-3                      &&\nu=\tilde u^{-3} &                          &\zeta=\tilde u^{-1}a \ar@{.>}[ul] \ar@{.>}[ur] &       & 		     &  & \\
-4  & \tau^2=\tilde u^{-4} & &     \tilde u^{-2}a    &             &      & 		     &  & \\
-5  &			        &\tilde u^{-3}a \ar@{.>}[ul]\ar@{.>}[ur] &                          &           &      & 		     &  & \\
                         && &                          &           &      & 		     &  & \\
\vdots & \vdots  &&                          &           &      & 		     &  &  \\
}
$$
}
\end{samepage}

\begin{proof}
It is enough to give the proof for coefficients in $\F_2$, since we have $\wh H_*(\Lambda;\K)=\wh H_*(\Lambda;\F_2)\otimes_{\F_2}\K$ when $\K$ has characteristic $2$. We refer to the table for the notation used in the proof.

In the proof we will use in an essential way the homological geometric generators of Goresky-Hingston for $H_*\Lambda$ described in Remarks~\ref{rmk:geometric-generators} and~\ref{rmk:geometric-generators-U}, given by $A=[\mathrm{pt}]\in H_0\Lambda$, $U\in H_3\Lambda$, $X=UA\in H_1\Lambda$ and $\Theta\in H_4\Lambda$. We will also use the Goresky-Hingston cohomological non-nilpotent class $\omega\in H^1\Lambda$ from~\cite[\S14]{Goresky-Hingston}, which is dual to $X$. 

Denote $a=A$ and $\tilde u=U$. 
Our first claim is that there exists $\eps\in\F_2$ such that 
$$
\tilde u=u+\eps u^3a. 
$$
To prove this, it is enough to show that $\tilde u$ is non-nilpotent: we have $\tilde u\in H_3\Lambda$, which is generated by $u$ and $u^3a$ with $u$ non-nilpotent and $u^3a$ nilpotent (Theorem~\ref{thm:Menichi_CP1}). Let $p\in S^2$ be a generic point and consider the inclusion $i:\Omega_p S^2\hookrightarrow \Lambda S^2$ and the induced map $i_!:H_*(\Lambda S^2)\to H_{*-2}(\Omega_p S^2)$. The key fact is that $i_!$ intertwines the loop product on $H_*(\Lambda S^2)$ with the Pontryagin product on $H_*(\Omega_p S^2)$ (\cite[\S9.3]{Goresky-Hingston}, \cite[Proposition~3.4]{CS}). Therefore, in order to prove that $\tilde u$ is non-nilpotent for the loop product, it is enough to prove that $i_!(\tilde u)$ is non-nilpotent for the Pontryagin product. By the definition of $\tilde u$ and $i_!$, the class $i_!(\tilde u)\in H_1\Omega_p S^2$ is represented by the descending manifold of one closed geodesic loop at $p$. Since the energy functional on $\Omega_p S^2$ for the round metric is perfect (see for example~\cite[\S6.1]{Oancea-Morse} for a proof in the case of the free loop space, which adapts directly to the case of the based loop space), this is a generator of $H_1(\Omega_p S^2)$. On the other hand, it is a classical fact that $H_*(\Omega_p S^2)$ is a polynomial algebra with one generator in degree $1$ (\cite[III.1.B]{Bott-Samelson-1953}), so that $i_!(\tilde u)$ is in non-nilpotent for the Pontryagin product. 

In particular, we also have 
$$
\H_*\Lambda = \K[a,\tilde u]/(a^2).
$$

Let $x=\tilde u a = UA=X\in H_1\Lambda$. 
Let $\tau\in H^1\Lambda$ be the algebraic dual of $x$, so that $\tau=\omega$. We thus have 
$$
\langle \tau, \tilde u a\rangle=1.
$$
Let $E=[M]\in H_2\Lambda_0 \hookrightarrow H_2\Lambda \hookrightarrow \wh H_2\Lambda = \wh \H_0\Lambda$ be the unit, represented by the fundamental class of $[M]$. 

Let $y=\tilde u^2\in H_4\Lambda$. Then $y$ is non-nilpotent and $H_4\Lambda$ is generated by $y$ and $y^2a$, which is nilpotent. Since the geometric class $\Theta\in\H_4\Lambda$ is also non-nilpotent~\cite{Goresky-Hingston}, there exists $\eps'\in\F_2$ such that 
$$
\Theta=y+\eps' y^2a.
$$
 Let $\nu\in H^2(\Lambda,\Lambda_0)\hookrightarrow H^2\Lambda$ be the cohomological geometric generator that corresponds to the unique degree $1$ class in $H^*(SM)$. Since $\tilde u^2a=ya=\Theta A\in H_2(\Lambda)$ is the homological geometric generator that corresponds to the unique degree $1$ class in $H_*(SM)$, we find that 
$$
\langle \nu,\tilde u^2a\rangle =1.
$$
Let $\zeta_0\in H^2 \Lambda_0$ be the volume form on $S^2$, and let $\zeta=\ev^*\zeta_0$ be its image in $H^2\Lambda$ through the injective map $\ev^*:H^2\Lambda_0\hookrightarrow H^2\Lambda$. Then 
$$
\langle \zeta,E\rangle =1.
$$

We use the Frobenius algebra structure on $\wh\H_*\Lambda$, with counit $\boldeps$, product $\boldmu$, and pairing $\boldp=-\boldeps\boldmu=\boldeps\boldmu$ (characteristic $2$). The counit $\boldeps$ has degree $2n-1=3$, and it acts nontrivially only on $\wh\H_{-3}\Lambda=\wh H_{-1}\Lambda=H^2\Lambda$. 

We have $\boldeps(\zeta)=\boldeps\boldmu(\zeta,E)=\boldp(\zeta,E)=\langle \zeta,E\rangle =1$, and $\boldeps(\nu)=\boldeps\boldmu(\nu,E)=\boldp(\nu,E)=\langle \nu,E\rangle=0$. The last equality holds by Lemma~\ref{lma: duality} because $\Cr (E)=0<2\pi = \ccr(\nu)$ (here we use the fact that $\nu$ is a geometric generator). Summarizing, 
$$
\boldeps(\zeta)=1,\qquad \boldeps(\nu)=0.
$$

We now show that $(\zeta,\nu)$ is the dual basis in $H^2\Lambda=(H_2\Lambda)^\vee$ for the basis $(E,\tilde u^2a=u^2a)$ of $H_2\Lambda$. We have seen already that $\langle \zeta,E\rangle=1$ and $\langle \nu,E\rangle=0$. We also have $\langle \zeta,\tilde u^2a\rangle = \langle \ev^*\zeta_0,\tilde u^2a\rangle = \langle \zeta_0,\ev_*(\tilde u^2a)\rangle=0$. The last equality holds because $\ev_*(\tilde u^2a)=0$, since this is a homology class of degree $2$ that is supported on a point. (We use again the fact that $\tilde u^2a=ya=\Theta A$ is geometric, described by the descending manifold of a fiber of $STS^2$.) Finally, we have already shown that $\langle \nu,\tilde u^2a\rangle=1$. 

We now prove that 
$$
\tau x =\zeta.
$$
We have $\boldeps(\tau x)=\boldeps\boldmu(\tau,x)=\boldp(\tau,x)=\langle \tau,x\rangle=1$. Therefore 
$\tau x=\zeta+\epsilon_1\nu$ for some $\epsilon_1\in\F_2$. The coefficient $\epsilon_1$ can be determined by evaluating $\tau x$ on $\tilde u^2a$: we find $\epsilon_1=\langle \tau x,\tilde u^2a\rangle =\boldeps\boldmu(\tau \tilde u a,\tilde u^2a) = \boldeps\boldmu(\tau \tilde u, \tilde u^2a^2)=0$. Here the third equality holds because $\boldmu$ is associative and commutative, and the last equality holds because $a^2=0$. 

We now prove that
$$
\tau y=E+\eps' \tilde u^2 a.
$$   
(We recall that $\Theta=y+\eps' y^2a$, hence also $y=\Theta+\eps'\Theta^2a$.)
Let us write $\tau y=\epsilon_1 E + \epsilon_2 \tilde u^2 a$ with $\epsilon_1,\epsilon_2\in\F_2$. We then find $\epsilon_1= \langle \zeta, \tau y\rangle = \langle \tau x,\tau y\rangle = \boldeps\boldmu(\tau x,\tau y)= \boldeps\boldmu(\tau^2,yx)=\boldp(\tau^2,yx)=\langle \tau^2,y\tilde u a\rangle = \break \langle \tau^2,\Theta \tilde u a \rangle = \langle \tau,\tilde u a \rangle =\langle \tau,x\rangle=1$. Here the fourth equality $\boldeps\boldmu(\tau x,\tau y)= \boldeps\boldmu(\tau^2,yx)$ holds by associativity and commutativity of $\boldmu$, and the eighth equality $\langle \tau^2,\Theta \tilde u a \rangle = \langle \tau,\tilde u a \rangle$ holds because multiplication by $\tau$ preserves the geometric generators in cohomology, and multiplication by $\Theta$ preserves the geometric generators in homology, hence both Kronecker products can be understood as the Kronecker product of the \emph{same} cohomology and homology classes on $SM$ (see also~\cite[\S13-14]{Goresky-Hingston} and~\cite[Lemma~5.4]{Hingston-Rademacher}). We also find $\epsilon_2=\langle \nu,\tau y\rangle = \boldeps\boldmu(\nu,\tau y)=\boldeps\boldmu (\tau\nu,y)=\boldp(\tau\nu,y)=\langle \tau\nu,\Theta+\eps'\Theta^2 a\rangle =\langle \nu,E\rangle + \eps'\langle \nu,\Theta a\rangle = 0+ \eps' \langle \nu, ya\rangle =\eps'\langle \nu, \tilde u ^2 a\rangle = \eps'$. Here the third equality $\boldeps\boldmu(\nu,\tau y)=\boldeps\boldmu (\tau\nu,y)$ holds by associativity and commutativity of $\boldmu$, and the sixth equality involving $\langle \tau\nu,\Theta\rangle =\langle \nu,E\rangle$ and $\langle \tau\nu,\Theta^2a\rangle=\langle \nu,\Theta a \rangle$ holds for the same reason as before: multiplication by $\tau$ preserves the geometric generators in cohomology, and multiplication by $\Theta$ preserves the geometric generators in homology. 

The equality $\tau y=E+\eps' \tilde u^2 a$ implies that
\begin{equation}\label{eq:thetamym_dual}
\langle \alpha\tau^m,B y^m\rangle =\left\{
\begin{array}{ll} 
\langle \alpha,B\rangle, & m \mbox{ even},\\
\langle \alpha,B\rangle + \eps'\langle \alpha,B\tilde u^2a\rangle, & m \mbox{ odd} 
\end{array}\right.
\end{equation}
for all $\alpha\in H^*\Lambda$, $B\in H_*\Lambda$. Indeed, we have $\langle \alpha\tau^m,By^m\rangle=\boldp(\alpha \tau^m,B y^m)\break =\boldeps\boldmu(\alpha\tau^m,B y^m)=\boldeps\boldmu(\alpha\tau^my^m,B)=\boldeps\boldmu(\alpha(E+\eps'\tilde u^2 a)^m,B)$. If $m$ is even, the last term is equal to $\boldeps\boldmu(\alpha,B)=\langle \alpha,B\rangle$ because $(E+\eps'\tilde u^2a)^m=E$. If $m$ is odd, the last term is equal to $\boldeps\boldmu(\alpha(E+\eps'\tilde u^2a),B)=\boldeps\boldmu(\alpha,B) + \eps'\boldeps\boldmu(\alpha,B\tilde u^2a) = \langle\alpha,B\rangle + \eps' \langle \alpha,B\tilde u^2a\rangle$.

We prove that 
$$
\tau \cdot 1 = \nu.
$$
We write $\tau\cdot 1 = \epsilon_1\nu + \epsilon_2\zeta$ with $\epsilon_1,\epsilon_2\in\F_2$. We find $\epsilon_1=\langle \tau\cdot 1,\tilde u^2 a\rangle = \langle \tau \cdot 1,ya\rangle= \langle \tau\cdot 1,\Theta a\rangle = \langle 1,a\rangle =1$, where the last equality follows from~\eqref{eq:thetamym_dual}. Also $\epsilon_2=\langle\tau \cdot 1,E\rangle = 0$ because $\Cr(\tau\cdot 1)\le \Cr(\tau)+\Cr(1)=-2\pi$ (here we consider $\Cr$ in $\wh{H}_*\Lambda$), hence $\ccr(\tau\cdot 1)\ge 2\pi>0=\Cr(E)$ (here we consider $\ccr$ and $\Cr$ in $H^*\Lambda$ and $H_*\Lambda$, and we use that, for cohomology classes, $\ccr$ is equal to minus $\Cr$ taken in $\wh{H}_*\Lambda$, cf.~\cite[\S2.2]{CHO-PD}).  

We now prove that 
$$
1\cdot a=\zeta.
$$
We write $1\cdot a=\epsilon_1 \nu+\epsilon_2\zeta$ with $\epsilon_1,\epsilon_2\in\F_2$. We find $\epsilon_1=\langle 1\cdot a,ya\rangle = \boldeps\boldmu(1\cdot a,ya)=\boldeps\boldmu(1,ya^2)=0$, because $a^2=0$. Also $\epsilon_2=\langle 1\cdot a,E\rangle = \boldeps\boldmu(1\cdot a,E)=\boldeps\boldmu(1,a\cdot E)=\boldp(1,a)=\langle 1,a\rangle =1$. 

We now prove that 
$$
\tau \tilde u=1.
$$
We write $\tau \tilde u=\epsilon_1 \cdot 1 + \epsilon_2 \tilde u a$ with $\epsilon_1,\epsilon_2\in\F_2$. We have $Cr(\tau u)\le \Cr(\tau)+\Cr(u) = -2\pi+2\pi=0$. Since $\Cr(1)=0$ and $\Cr(\tilde u a)=2\pi$, this implies $\epsilon_2=0$ (otherwise $\tilde u a=\tau\tilde u + \epsilon_1 \cdot 1$ would have critical level $\le 0$). Also $\epsilon_1=\langle \tau \tilde u,a\rangle = \boldeps\boldmu(\tau \tilde u,a)=\boldeps\boldmu(\tau, \tilde ua)=\boldp(\tau,x)=\langle \tau,x\rangle=1$. 

As consequences of these relations we find 
\begin{itemize}
\item $1\cdot \tilde u=\tau \tilde u^2=\tau y = E+\eps'\tilde u^2 a$, therefore $(1+\eps'\tilde u a)\tilde u = 1\cdot \tilde u + \eps'\tilde u^2 a = E$, hence $\tilde u$ is invertible and 
$$
\tilde u^{-1}=1+\eps'\tilde u a.
$$
This implies $1=\tilde u^{-1}+\eps'\tilde u a$. We also have $(\tau + \eps'a)\tilde u^2 = \tau\tilde u^2 + \eps'\tilde u^2 a =E$, hence $\tilde u^2$ is invertible and 
$$
\tilde u^{-2}= \tau +\eps' a = (1+\eps'\tilde u a)^2 = 1\cdot 1.
$$
\item 
$$
\zeta = 1\cdot a = \tilde u^{-1}a.
$$
\item 
$$
\nu = \tau\cdot 1 = (\tilde u^{-2}+\eps' a)(\tilde u^{-1}+\eps'\tilde u a) = \tilde u^{-3}.
$$
\end{itemize} 

We now infer invertibility of $u$. We have $u=\tilde u+\eps\tilde u^3 a$, and therefore $(1+\eps'\tilde u a+\eps\tilde u a)\cdot u = (\tilde u^{-1}+\eps\tilde u a)(\tilde u + \eps\tilde u^3a)= E$. Thus $u$ is invertible in $\H_*\Lambda$ and 
$$
u^{-1}=1+\eps'\tilde u a + \eps\tilde u a = 1 +\eps' u a + \eps u a.  
$$
Also $u^2=\tilde u^2$, and therefore $u^{-2}=\tilde u^{-2}=\tau + \eps' a$. 

The previous computations imply the multiplication table 
$$
\wh\H_*\Lambda = \K[a,u^{\pm 1}]/(a^2)=\K[a,\tilde u^{\pm 1}]/(a^2). 
$$ 
Indeed, $\wh{H}_{-3}\Lambda$ is generated by $1=\tilde u^{-1}+\eps'\tilde u a$ and $\tilde u a$, hence also by $\tilde u^{-1}$ and $\tilde u a$. (Similarly $\wh{H}_{-4}\Lambda$ is generated by $\tau=\tilde u^{-2}+\eps' a$ and $a$, hence also by $\tilde u^{-2}$ and $a$.) Since $\tilde u$ is invertible, multiplication by $\tilde u^{-1}$ is an isomorphism in each degree, and we conclude by taking into account that $\wh{H}_*\Lambda$ has rank $2$ in each degree. 
 
%
%

To establish the formulas for the BV operator from the statement of the theorem, one proceeds as in the proof of Theorem~\ref{thm:Rabinowitz_BV_projective_spaces_char_divides} by decreasing induction on $k$, using the 7-term relation and the fact that these formulas are known for $k\ge 0$. 
\end{proof}

The previous calculations allow us to determine precisely the relationship between the generators $u$, $\tilde u$ and $\Theta$. We will refer to $u$ as the \emph{Menichi class}, or \emph{Menichi generator}.  

\begin{proposition} Let $u$ be the Menichi generator from Theorem~\ref{thm:Rabinowitz_BV_S2_char_2}, and $\tilde u$, $\Theta$ the geometric generators used in the proof of Theorem~\ref{thm:Rabinowitz_BV_S2_char_2}.\footnote{I.e., $\tilde u$ is represented by the descending manifold of a singular section of $STS^2$, and $\Theta$ is represented by the descending manifold of $STS^2$.}  
We have 
\begin{equation} \label{eq:tildeu-u}
\tilde u=u+u^3a,
\end{equation}
\begin{equation} \label{eq:tildeu2-Theta}
\tilde u^2=\Theta+\Theta^2a,
\end{equation}
\begin{equation} \label{eq:Deltatildeu}
\Delta \tilde u =\Theta.
\end{equation}
\end{proposition}

\begin{proof}
As before, we can use without loss of generality coefficients in $\F_2$ rather than in a general field of characteristic $2$. 

We prove~\eqref{eq:tildeu-u}. We use the fact that $\wh{\H}_*\Lambda$ is a BV Frobenius algebra~\cite[Theorem~5.2]{Latschev-Oancea}, which means, besides the Frobenius relations, that $\boldp(\Delta\otimes 1)=\boldp(1\otimes\Delta)$. This relation is called the \emph{BV Frobenius relation}. Let $X=\alpha u ^{2k-1} + \beta u^{2k+1} a$ for some $k\ge 1$. Using notation from the proof of the previous theorem and the BV Frobenius relation, we compute   
\begin{align*}
\boldp (u^{-2k-2}+u^{-2k}a,X) & = \boldp (\Delta(u^{-2k-1}a),X) 
 = \boldp (u^{-2k-1}a,\Delta X)\\
& = \beta \boldp (u^{-2k-1}a,u^{2k}+u^{2k+2}a)\\
& = \beta\boldeps(u^{-1}a)
 = \beta \boldeps(\tilde u^{-1}a)
 = \beta \boldeps(\zeta) 
 = \beta.
\end{align*}
On the other hand, without appealing to the BV operator we can compute directly 
\begin{align*}
\boldp(u^{-2k-2}+u^{-2k}a,X) & = \boldp (u^{-2k-2}+u^{-2k}a,\alpha u^{2k-1} + \beta u^{2k+1}a) \\
& = \alpha\boldeps(u^{-3}+u^{-1}a) + \beta \boldeps(u^{-1}a)\\
& = \alpha \boldeps(u^{-3}) + \alpha+\beta.
\end{align*}
We claim that this equals $\alpha\eps+\alpha+\beta$, where $\tilde u=u+\eps u^3a$ for some $\eps\in\F_2$. Indeed, we then have $u=\tilde u + \eps\tilde u^3a$, $u^{-1}=\tilde u^{-1}+\eps \tilde u a$, and $u^{-3}=\tilde u^{-3}+\eps\tilde u^{-1}a=\nu+\eps\tilde u^{-1}a$, so that $\boldeps(u^{-3})=\boldeps(\nu)+\eps\boldeps(\tilde u^{-1}a) = \eps$. 

The relation 
$$
\beta=\alpha\eps+\alpha+\beta
$$
must be valid for all $\alpha,\beta\in\F_2$. This implies $\eps=1$ and proves~\eqref{eq:tildeu-u}.

We prove~\eqref{eq:tildeu2-Theta}. We already know that $\tilde u^2=\Theta+\eps'\Theta^2 a$ for some $\eps'\in\F_2$, and we need to exclude the case $\tilde u^2=\Theta$. Assuming this to be true, we calculate $\Delta (\tilde u)$ in two ways in order to obtain a contradiction. On the one hand, for degree reasons we must have $\Delta(\tilde u)=\beta\Theta+\beta'\Theta^2a$. But $\Cr \Delta(\tilde u)\le \Cr \tilde u=2\pi$ (for the round metric), whereas $\Cr \Theta=2\pi$ and $\Cr \Theta^2a=4\pi$. This implies $\beta'=0$, i.e., $\Delta (\tilde u) = \beta\Theta=\beta\tilde u^2$, where the last equality holds by the assumption $\tilde u^2=\Theta$. Secondly, using $\tilde u=u+u^3a$ we find $\Delta (\tilde u) = u^2+u^4a= \tilde u^2+ \tilde u^4a$. Therefore, $\beta \tilde u^2 = \tilde u^2+\tilde u^4a$ for some $\beta\in\F_2$, which is impossible because $\tilde u^2$ and $\tilde u^4a$ are linearly independent. This is the desired contradiction, which proves~\eqref{eq:tildeu2-Theta}. 

We prove~\eqref{eq:Deltatildeu}. We have computed in the previous paragraph $\Delta (\tilde u) = \tilde u^2 +\tilde u^4 a$, and the relation $\tilde u^2=\Theta+\Theta^2a$ implies $\Delta(\tilde u)=\Theta+\Theta^2a+ (\Theta+\Theta^2 a)^2a= \Theta$.  
\end{proof}

\begin{remark}
It is an interesting question to describe explicit homologies in $\Lambda S^2$ in order to prove the relations~(\ref{eq:tildeu-u}--\ref{eq:Deltatildeu}). 
\end{remark}

\begin{remark}
When expressed in terms of the geometric generator $\tilde u=u+u^3a$, the formulas for the BV operator change as follows: $\Delta(\tilde u^k)=\Delta(\tilde u^k a)=0$ if $k$ is even, $\Delta(\tilde u^k a)=\tilde u^{k-1}+\tilde u^{k+1}a$ if $k$ is odd, and $\Delta(\tilde u^k)=\tilde u^{k+1} + \tilde u^{k+3}a$ if $k$ is odd (in particular $\Delta (\tilde u)=\tilde u^2+\tilde u^4a\neq 0$). 
\end{remark}

The following corollary removes the indeterminacies involving $\eps$ and $\eps'$ in the proof of Theorem~\ref{thm:Rabinowitz_BV_S2_char_2}, since we now know that $\eps=\eps'=1$.
\begin{corollary}
With the notation from the proof of Theorem~\ref{thm:Rabinowitz_BV_S2_char_2}, the following relations hold:
\begin{align*}
\tau y & = E+\tilde u^2a = E+u^2a, \\
\tilde u^{-1} & = 1+\tilde ua, \\
u^{-1} & = 1,\\
\tilde u^{-2} =u^{-2} & = \tau + a.
\end{align*}
\end{corollary}

\subsection{Loop cohomology BV algebra for simply connected CROSS that are projective spaces}
\qquad 

In this section we infer the loop cohomology BV algebra structure from the one on $\wh\H_*\Lambda$, using the decomposition $\wh{\H}_*\Lambda\simeq \ol \H_*\Lambda\oplus \ol \H^{1-2n-*}\Lambda$.  

We first state the result for all simply connected CROSS that are projective spaces, except $\C P^1$. 

\begin{theorem} 
\label{thm:BV_projective_spaces_cohomology_char_does_not_divide}
Let $M$ be one of the CROSS $\C P^d$ with $d\ge 2$, $\H P^d$ with $d\ge 1$, or $\Ca P^2$ (we set $d=2$ in this last case), and let $\K$ be a field whose characteristic does not divide $d+1$. 
We have an isomorphism of BV algebras 
$$
\ol\H^{1-2n-*}(\Lambda;\K)\cong y^{-1}\K[a,x,y^{-1}]/\la a^{d+1},x^2,xa^d, y^{-1}a^d\ra,
$$
where $|a|=-i$, $|x|=-1$, $|y^{-1}|=-i(d+1)+2$, and the BV operator acts by
$$
\Delta(y^k a^\ell)=0,\qquad \Delta(y^k x a^\ell) = (-\ell + k + (k+1)d)y^k a^\ell,
$$
for $k<0$ and $0\le \ell\le d$. \qed
\end{theorem}

\begin{theorem} \label{thm:BV_projective_spaces_cohomology_char_divides} 
Let $M$ be one of the CROSS $\C P^d$ with $d\ge 2$, $\H P^d$ with $d\ge 1$, or $\Ca P^2$ (we set $d=2$ in this last case), and let $\K$ be a field whose characteristic divides $d+1$. 
We have an isomorphism of BV algebras 
$$
\H^{1-2n-*}(\Lambda;\K)\cong y^{-1}\K[a,u,y^{-1}]/\la a^{d+1},u^2\ra,
$$
where $|a|=-i$, $|u|=i-1$, $|y^{-1}|=-i(d+1)+2$, and the BV operator acts by
$$
\Delta(y^k a^\ell)=0, \qquad \Delta(y^ku)=0,
$$
for $k<0$ and $0\le \ell\le d$, 
\begin{equation}
\Delta(y^k u a^\ell) = \ell y^k a^{\ell-1}
\end{equation}
for $k<0$ and $1\le \ell\le d$ in the case of $\H P^d$ and $\Ca P^2$, respectively for $k<0$ and $2\le \ell\le d$ in the case of $\C P^d$, and
\begin{equation}
  \Delta(y^k ua) = y^k + \frac{(d+1)d}{2} y^{k+1}a^d =
  \left\{\begin{array}{cl} y^k+\frac{d+1}2y^{k+1}a^d, & d \mbox{ odd,}\\
  y^k,& d \mbox{ even}
  \end{array}\right.
\end{equation}
for $\C P^d$ and $k<-1$,
and 
$$
\Delta(y^{-1}ua)=y^{-1}.
$$ \qed
\end{theorem}

We now state the corresponding results for $\C P^1$.

\begin{theorem} 
\label{thm:BV_S2_cohomology_char_does_not_divide}
Let $M=\C P^1$, and let $\K$ be a field of characteristic different from $2$. 
We have an isomorphism of BV algebras 
$$
\ol\H^{-3-*}(\Lambda;\K)\cong y^{-1}\K[x,y^{-1}]/\la x^2\ra,
$$
where $|x|=-1$, $|y|=2$, and the BV operator acts by
\begin{equation*} 
\Delta(y^k)=0, \qquad \Delta(y^k x) = (2k + 1)y^k,
\end{equation*}
for $k<0$. \qed
\end{theorem}

\begin{theorem} 
\label{thm:BV_S2_cohomology_char_2}
Let $M=\C P^1$ and let $\K$ be a field of characteristic $2$. We have an isomorphism of BV algebras 
$$
\H^{-3-*}(\Lambda;\K)=u^{-1}\K[a,u^{-1}]/\langle a^2\rangle,
$$
where $|a|=-2$, $|u|=1$, and the BV operator acts, for $k<0$, by 
$$
\Delta (u^k)=0
$$
and 
$$
\Delta (au^k)=\left\{\begin{array}{rl} 0, & k \mbox{ even}, \\ u^{k-1}+au^{k+1}, & k \mbox{ odd}. \end{array} \right.
$$
\qed
\end{theorem}

\subsection{The homological and cohomological Uebele class}\label{ss:coh-Uebele}

If the characteristic of $\K$ does not divide $d+1$, let $y\in\wh{\H}_*\Lambda$ be the generator from 
Theorems~\ref{thm:Rabinowitz_BV_projective_spaces_char_does_not_divide} and~\ref{thm:Rabinowitz_BV_S2_char_does_not_divide}. If the characteristic of $\K$ divides $d+1$, let 
$$
y\in\wh{\H}_{n+i-2}\Lambda
$$ 
be the generator from Theorem~\ref{thm:Rabinowitz_BV_projective_spaces_char_divides}, or $y=\tilde u^2=u^2$ 
the generator from the proof of Theorem~\ref{thm:Rabinowitz_BV_S2_char_2}. (The degree of $y$ is indeed $|y|=i(d+1)-2=n+i-2$.) By abuse of language we will henceforth call $y$ the \emph{homological Uebele class}, although in the case of $\C P^d$, $d\ge 2$ it may differ by a multiple of $y^2a^d$ from the Uebele class $\Theta$ used in the proof of Theorem~\ref{thm:Rabinowitz_BV_projective_spaces_char_divides}, and in the case of $\C P^1$ it may differ by $y^2a$ from the non-nilpotent geometric class $\Theta$ mentioned in the proof of Theorem~\ref{thm:Rabinowitz_BV_S2_char_2}. 

Recall that $y$ is invertible and $|y^{-1}|=-n-i+2$. Let $PD:\wh\H_*\Lambda\to\wh\H^{1-2n-*}\Lambda$ be the Poincar\'e
duality isomorphism. We define
$$
   \theta :=
   PD(y^{-1})\in \wh{\H}^{i-1-n}\Lambda=\wh{H}^{i-1}\Lambda=H^{i-1}\Lambda\oplus H_{2-i}\Lambda. 
$$
We call $\theta$ the \emph{cohomological Uebele class}. Equivalently, and because the Poincaré duality isomorphism simply exchanges factors in the simply connected case, we can identify $\theta$ with $y^{-1}\in \wh{H}_{-i+2}\Lambda$. The key property of the cohomological Uebele class is that multiplication by $\theta$ is an isomorphism in each degree. 


\begin{remark} In the case of $\H P^d$, $d\ge 1$ or $\Ca P^2$, we have $i=4$ or $i=8$, so that $\wh{\H}^{i-1-n}\Lambda=H^{i-1}\Lambda$ and $\theta$ is a class in $H^{i-1}\Lambda$.  
  In the case of $\C P^d$, $d\ge 1$ we have $i=2$ and therefore $\theta$ may differ from its projection to $H^{i-1}\Lambda$ by a multiple of the point class
$a^d\in H_0\Lambda$. Nevertheless, all higher powers $\theta^k$, $k\ge 2$ are classes in $H^*\Lambda$. 
In the case of $\C P^1$, the class denoted $\tau$ in the proofs of Theorems~\ref{thm:Rabinowitz_BV_S2_char_does_not_divide} and~\ref{thm:Rabinowitz_BV_S2_char_2} is the projection onto $H^1\Lambda$ of the cohomological Uebele class, and differs from it by $a$.
\end{remark}

Recall that the degrees of an arbitrary homology class $X\in\wh{H}_*\Lambda$ and of an arbitrary cohomology class $\xi\in \wh{H}^*\Lambda$ are related to their shifted degrees in $\wh\H_*\Lambda$ and $\wh\H^{1-2n-*}\Lambda$ by 
\begin{equation}\label{eq:deg}
  \deg(X) = |X|+n,\qquad \deg(\xi) = -|\xi|-n+1.  
\end{equation}
In particular, given the class $\theta\in \wh{H}^{i-1}\Lambda$, we view $\theta^k$, $k\ge 2$ as an element of $H^*(\Lambda,\Lambda_0)$ of degree   
\begin{equation}\label{eq:deg-theta}
   \deg(\theta^{k}) = k(n+i-2)-n+1.
\end{equation}

The following corollary summarizes the structure of loop homology and cohomology with respect to the Uebele classes. 

\begin{corollary}\label{cor:CROSS-structure}
Let $M$ be a CROSS that is a projective space and $\K$ any coefficient field. 

(a) Each $X\in\H_*(\Lambda;\K)$ has the form $X=Zy^k$, with $k\geq 0$ and $Z$ in a fixed finite dimensional space $F$ of shifted degree $-n\leq |Z|\leq i-1$. 

(b) Each $\xi\in\H^{1-2n-*}(\Lambda;\K)$ has the form $\xi=\zeta\theta^\ell$, with $\ell\geq 1$ and $\zeta$ in a fixed finite dimensional space $F$ of shifted degree $-n\leq |\zeta|\leq i-1$. 

(c) If $X=Zy^k$ and $\xi=\zeta\theta^\ell$ as in (a) and (b) satisfy $\deg(X)=\deg(\xi)$, then $k\leq\ell\leq k+2$. 
\end{corollary}

\begin{proof}
For a CROSS that is not $\C P^1$, part (a) follows from Corollary~\ref{cor:loop_BV_CROSS_over_K_char_does_not_divide_d+1} with the vector space $F=\K[a,x]/\la a^{d+1},x^2,xa^d\ra$ if ${\rm char}\,\K$ does not divide $d+1$, and from Theorem~\ref{thm:loop_BV_CROSS_over_K_char_divides_d+1} with $F=\K[a,u]/\la a^{d+1},u^2\ra$ if ${\rm char}\,\K$ divides $d+1$.  
Part (b) follows from Theorem~\ref{thm:BV_projective_spaces_cohomology_char_does_not_divide} with $F=\K[a,x]/\la a^{d+1},x^2,xa^d\ra$ if ${\rm char}\,\K$ does not divide $d+1$, and from Theorem~\ref{thm:BV_projective_spaces_cohomology_char_divides} with $F=\K[a,u]/\la a^{d+1},u^2\ra$ if ${\rm char}\,\K$ divides $d+1$. For $\C P^1$, both parts follows in a similar way from Corollary~\ref{cor:loop_BV_CROSS_over_K_char_does_not_divide_d+1} and Theorems~\ref{thm:Menichi_CP1}, \ref{thm:BV_S2_cohomology_char_does_not_divide}, and~\ref{thm:BV_S2_cohomology_char_2}.

For (c), we abbreviate $\nu:=n+i-2$ and combine equations~\eqref{eq:deg} and~\eqref{eq:deg-theta} with parts (a) and (b) to compute the degrees:
\begin{align*}
  \deg(X) &= |X|+n = k\nu+n+|Z| \in [k\nu,(k+1)\nu+1], \cr
  \deg(\xi) &= -|\xi|-n+1 = \ell\nu-n+1+|\zeta| \in [(\ell-1)\nu,\ell\nu+1].
\end{align*}
Comparison of the degrees yields $k\leq\ell$ and $\ell-1\leq k+1$.
\end{proof}

The following lemma describes more precisely the homology classes that pair nontrivially with certain cohomology classes. It will be used in the discussion of resonances in~\S\ref{sec: res}.

\begin{lemma}\label{lem:coh-Uebele}
In the setting of Theorem~\ref{thm:BV_projective_spaces_cohomology_char_divides}, i.e., for a CROSS that is a simply connected projective space, 
and for a field $\K$ whose characteristic divides $d+1$,  
consider an element $X\in H_*\Lambda$ with $\deg(X) = \deg(\theta^{k+1})$, $k\geq 1$. Then $X$ is of the form  
\begin{equation} \label{eq:Xalphabeta}
   X=\alpha y^{k-1} u + \beta y^k a^d u,\qquad \alpha,\beta\in\K,
\end{equation}
where the first term is only present in the case $i=2$.
Moreover, the coefficient $\beta$ is nonzero in each of the following cases:

(i) $M$ is $\H P^d$ or $\Ca P^2$ and $\brat{\theta^{k+1},X} \neq 0$;

(ii) $M=\C P^d$ with $d\ge 2$ and 
$\langle\theta^{k+1} + \frac{(d+1)d}{2} \theta^k a^d,X\rangle\neq 0$. 

(iii) $M=\C P^1$ and $\langle \theta^{k+1}+\theta^k a,X\rangle \neq 0$.
\end{lemma}

\begin{proof}
Since $\deg(\theta^{k+1})$ is odd, $X$ must be a linear combination of elements of the form $X=y^r a^\ell u$ with $r\geq 0$ and $0\leq\ell\leq d$. A short computation shows that $\deg(X) = \deg(\theta^{k+1})$ spells out as
$$
   (r-k+1)n+(r-k-\ell)(i-2)-2\ell=0.
$$
For $r\leq k-1$ the left hand side is $\leq (-1-\ell)(i-2)-2\ell$, which is $<0$ unless $i=2$ and $\ell=0$, so in this case the only solution is $X=\alpha y^{k-1} u$ if $i=2$, with $\alpha\in\K$.
For $r\geq k$ the left hand side is $\geq n-\ell(i-2)-2\ell = i(d-\ell)$, which is $>0$ unless $d=\ell$, so in this case the only solution is $X=\beta y^k a^d u$ with $\beta \in\K$. 
This proves the first assertion. 
 

Consider now Case (i).
By Theorem~\ref{thm:BV_projective_spaces_cohomology_char_divides}
(where $\theta^{k+1}$ and the dual BV operator $\Delta^*$ correspond to $y^{-(k+1)}$ and $\Delta$), $\theta^{k+1}$ is in the image of the dual BV operator $\Delta^*$, so $\theta^{k+1} = \Delta^*b$ for some $b\in H^*\Lambda$.  
The assumption $\brat{\theta^{k+1},X} \neq 0$ implies 
$0\neq \brat{\theta^{k+1},X} = \brat{\Delta^*b,X} = \brat{b,\Delta X}$ and thus
$$
   0\neq \Delta X = \alpha\Delta(y^{k-1}u)+\beta\Delta(y^k a^d u) = \beta d y^k a^{d-1}.
$$
This implies $\beta\neq 0$.
Case (ii) is similar: 
Since $\theta^{k+1} + \frac{(d+1)d}{2} \theta^k a^d$ is in the image of the dual BV operator $\Delta^*$, it follows as in Case (i) that $\Delta X\neq 0$, and therefore $\beta\neq 0$. Case (iii) follows in the same way: $\theta^{k+1} + \theta^k a$ is in the image of the dual BV operator $\Delta^*$, and we obtain $\Delta X\neq 0$, hence $\beta\neq 0$.
\end{proof}

\section{String point invertibility of $\C P^d$, $\H P^d$, $\Ca P^2$}  \label{sec:string_point_invertibility}

One application of the previous computations is to \emph{string point invertibility}, a notion introduced by the third named author
in~\cite{Shelukhin19}. Denote $\{\cdot,\cdot\}$ the loop bracket on $H_*\Lambda$ defined by Chas and Sulivan~\cite{CS}, which is expressed in terms of the loop product and the BV operator by equation~\eqref{eq:Gerstenhaber-bracket-intro}. 
For a class $c\in H_*\Lambda$, consider the operator $P_c:H_*(M)\to H_{*+\deg(a) -n+1}(M)$ defined by 
$$
P_c=\ev_* \circ \{\cdot,c\}\circ i_*,
$$
with $i:M\hookrightarrow \Lambda$ the inclusion of constant loops and $\ev:\Lambda\to M$ the evaluation at the basepoint of a loop. 

\begin{definition}[{\cite{Shelukhin19}}] The manifold $M$ is called \emph{string point invertible}
over the coefficient field $\K$ if it is $\K$-orientable and there exists a collection of classes $c_1,\dots,c_N\in H_*\Lambda$ 
with $\K$-coefficients such that 
$$
[M]=P_{c_N}\circ\dots\circ P_{c_1}([\mathrm{pt}]).
$$
\end{definition}

 The terminology originates in the notion of point invertible symplectic manifolds \cite{BC-uniruling} whose Lagrangian submanifolds satisfy special properties; this condition has also appeared in the symplectic dynamics literature, see \cite[Theorem 1.4(ii)]{GG-generic}.
The interest of this definition comes from the following resolution of a conjecture of Viterbo for string point invertible manifolds.

\begin{theorem}[{\cite{Shelukhin19}}]\label{thm:Viterbo-conj}
Let $M$ be a closed manifold that is string point invertible over $\K$, and choose a Riemannian metric on $M$. Then the spectral norm over $\K$ of the pair consisting of the zero-section inside $T^*M$ and a closed exact Lagrangian contained in the unit disc bundle is uniformly bounded.
\end{theorem}

The paper~\cite{Shelukhin19} contains a discussion of examples of manifolds that are string point invertible.  Those examples most notably include all spheres with $\Z/2$-coefficients
and all compact connected Lie groups with characteristic zero coefficients,
from which one can construct further examples by taking products of manifolds that are string point invertible over the same field of coefficients.  
Our next result addresses string point invertibility for the other CROSS.

\begin{theorem}\label{thm:string-point} 
Let $M$ be one of the CROSS $\C P^d$, $\H P^d$, or $\Ca P^2$ (set $d=2$ in this last case). Then $M$ is string point invertible over a field of characteristic $p$ if and only $p$ equals the Euler characteristic $\chi(M)=d+1$ (in particular, $d+1$ must be prime). 
\end{theorem}

  Combined with Theorem~\ref{thm:Viterbo-conj}, this implies Viterbo's conjecture on the uniform bound of the spectral norm with coefficients in a field of characteristic $p=d+1$ for $\C P^d$ or $\H P^d$ with $d+1$ prime, and for $\Ca P^2$ with coefficients in a field of characteristic $p=3$. 
   
\begin{remark}
It was already observed in~\cite{Shelukhin19} that the CROSS from Theorem~\ref{thm:string-point} are not string point invertible in characteristic 0, and also not string point invertible with $\Z/2$-coefficients except for $\C P^1=S^2$ and $\H P^1=S^4$.  Using different methods, Viterbo's conjecture is proved in~\cite{Shelukhin22} for all CROSS except $\Ca P^2$ with coefficients in $\Z/2$, and in \cite{Viterbo2022-inverse, Guillermou2022} for all homogenenous spaces $X=G/H$ with arbitrary coefficients, at least if $H$ is connected.
\end{remark}

Before turning to the proof of Theorem~\ref{thm:string-point}, we note the following general fact about string point invertibility. 

\begin{proposition}\label{prop:string-point}
A closed orientable manifold $M$ with $\chi(M)\neq 0$ is not string point invertible over a coefficient field whose characteristic does not divide $\chi(M)$. 
\end{proposition}

\begin{proof}
Let $\K$ be a field of coefficients. Recall from~\cite{Tamanoi} (see also~\cite{CO}) that $\K\chi(M)[\mathrm{pt}]$ is an ideal in the ring $H_*\Lambda$. If the characteristic of $\K$ does not divide $\chi(M)$, then $\chi(M)$ is invertible in $\K$, and thus $\K[\mathrm{pt}]$ is an ideal for the loop product. Since $\Delta[\mathrm{pt}]=0$, it follows that $\K[\mathrm{pt}]$ is also an ideal for the loop bracket, i.e., $\{[\mathrm{pt}],c\}\in\K[\mathrm{pt}]$ for all $c\in H_*\Lambda$. This means that $P_c([\mathrm{pt}])\in\K[\mathrm{pt}]$ for all $c\in H_*\Lambda$, so $M$ is not string point invertible over $\K$. 
\end{proof}

\begin{proof}[Proof of Theorem~\ref{thm:string-point}] 
Let $\K$ be a ground field of characteristic $p$. By Proposition~\ref{prop:string-point}, $M$ is not string point invertible over $\K$ if $p$ does not divide $d+1$. So it remains to consider the case where $p$ divides $d+1$. 

We first discuss the case of $\H P^d$ with $d\ge 1$, and $\Ca P^2$. Using the definition of the Gerstenhaber bracket~\eqref{eq:Gerstenhaber-bracket-intro} and Theorem~\ref{thm:loop_BV_CROSS_over_K_char_divides_d+1} we compute
\begin{equation}\label{eq:iterated-bracket}
  \{ua^k,a^\ell\} = -\ell a^{k+\ell-1} 
\end{equation}
for all $k\ge 0$ and $\ell\ge 0$. 
This implies $P_u(a^\ell)=\{u,a^\ell\} = -\ell a^{\ell-1}$ for all $\ell\ge 0$, hence 
$$
P_u^d([\mathrm{pt}])=P_u^d(a^d)=(-1)^d d! a^0 = (-1)^d d! [M]. 
$$
If $p=d+1$ then $1,\dots, d$ are invertible in $\K$, so $M$ is string point invertible over $\K$. 
Assume now that $p\le d$. Theorem~\ref{thm:loop_BV_CROSS_over_K_char_divides_d+1} shows that the BV operator is $y$-linear (i.e., it commutes with left and right multiplication by $y$), and thus the loop bracket is $y$-bilinear. From this and~\eqref{eq:iterated-bracket} we see that the only way to reach $[M]=a^0$ starting from $[\mathrm{pt}]=a^d$ is by applying $P_u^d$. However, $p\le d$ and $P_u^d([\mathrm{pt}])=(-1)^dd![M]$ imply $P_U^d([\mathrm{pt}])=0$, so $M$ is not string point invertible over $\K$. 

We now discuss the case of $\C P^d$ with $d\ge 2$. The equality $\{ua^k,a^\ell\} = -\ell a^{k+\ell-1}$ holds now for all $k\ge 0$, $\ell\ge 0$, except $k=0$, $\ell=1$ in which case we have 
$$
\{u,a\}=-\Delta(ua)=-1-\frac{(d+1)d}{2}ya^d. 
$$
Then $P_u(a^\ell)=\{u,a^\ell\} = -\ell a^{\ell-1}$ for all $\ell\ge 0$, except $\ell=1$ in which case we have $P_u(a)=\{u,a\}=-1-\frac{(d+1)d}2ya^d$. Therefore 
$$
P_u^d([\mathrm{pt}])=P_u^d(a^d)=(-1)^d d! (1+\frac{(d+1)d}{2}ya^d). 
$$
If the prime $p$ equals $d+1$, and since $d\ge 2$, we find that $p=d+1$ is odd, and therefore $p$ divides $\frac{(d+1)d}{2}$, so that $P_u^d([\mathrm{pt}])=(-1)^dd![M]$ and $M$ is string point invertible over $\K$. If $p\le d$ we find that $P_u^d([\mathrm{pt}])=0$. On the other hand, the same argument as above based on $y$-linearity of the BV operator shows that the only way to reach $[M]$ from $[\mathrm{pt}]$ would be to apply $P_u^d$. We conclude that $M$ is not string point invertible if $p$ divides strictly $d+1$. 

We finally discuss the case of $\C P^1=S^2$ with coefficients in a field of characteristic $2$. Using Theorem~\ref{thm:Menichi_CP1} we directly find $P_{u+u^3a}([\mathrm{pt}])=P_{u+u^3a}(a)=\Delta(ua)+\Delta(u+u^3a)a = 1+au^2+(u^2+u^3a)a=1=[M]$, so that $\C P^1$ is string point invertible. 
\end{proof}

\section{Resonances for $\C P^d$, $\H P^d$, $\Ca P^2$}\label{sec: res}

In this section, $M$ denotes a closed manifold of dimension $n$ and $\Lambda$ its free loop space. All homology groups are taken with coefficients in a field $\bK$. Given a Finsler metric $g$ on $M$, we denote by $\Lambda_\alpha\subset\Lambda$ the subset of loops of length $\leq\alpha$. Recall from~\cite{CHO-PD} the definition of the critical levels of nonzero classes $X\in H_*\Lambda$ and $x\in H^*(\Lambda,\Lambda_0)$,
\begin{align*}
\Cr(X) &= \inf\{\alpha\in\R\mid X\in\im(H_*\Lambda_\alpha\to H_*\Lambda)\}, \\
\ccr(x) &= \sup\bigl\{\alpha\in\R\mid x\in\im\bigl(H^*(\Lambda,\Lambda_\alpha)\to H^*(\Lambda,\Lambda_0)\bigr)\bigr\}.
\end{align*}
Gromov~\cite[Theorem 7.3]{gromov-metric} proved that for every closed simply connected Finsler manifold $(M,g)$ and every coefficient field $\K$ there exist constants $\lambda^*, \lambda_*, C$ such that for all $X \in H_*(\Lambda;\bK)\setminus\{0\}$\footnote{When speaking of $\deg(X)$ we always assume that $X$ has homogeneous degree.}
\[\lambda^* \deg(X)-C \leq \Cr(X) \leq \lambda_* \deg(X)+C.\]
We provide a new proof of this result for CROSS in Proposition~\ref{prop: mini Gromov} below. Sharpening Gromov's result in a special case, Hingston and Rademacher proved in~\cite{Hingston-Rademacher} that, for every Finsler metric $g$ on the sphere $S^n$, $n>2$ and every field $\K$, there exist positive constants $\lambda,C$ depending only on $g$ and $\K$ such that for all $X \in H_*(\Lambda;\bK)\setminus\{0\}$ and all $x\in H^*(\Lambda;\K)\setminus\{0\}$,
\[|\lambda \deg(X) - \Cr(X)| \leq C\]
and 
\[|\lambda \deg(x) - \ccr(x)| \leq C.\]
The following theorem extends this result to other CROSS with coefficients in a field over which they are string point invertible.
This makes partial progress on a question of Hingston.

\begin{theorem}\label{thm: resonance}
Let $M$ be one of the CROSS $\C P^d$ with $d\geq 2$, $\H P^d$ with $d\geq 1$, or $\Ca P^2$ (where we set $d=2$), 
such that $d+1$ is prime, and let $\K$ be a coefficient field of characteristic ${\rm char}(\K) = d+1$.
Then, for every Finsler metric $g$
on $M$, there exist positive constants $\lambda,C$ depending only on $g$ and $\K$ 
such that for all homogeneous elements $X \in H_*(\Lambda;\K) \setminus \{0\}$ and $x\in H^*(\Lambda;\K)\setminus \{0\}$ we have
\[ |\lambda \cdot \deg(X) - \Cr(X)| \leq C\]
and 
\[ |\lambda \cdot \deg(x) - \ccr(x)| \leq C.\]
\end{theorem}


The next result concerns $\C P^1=S^2$, which is the only sphere that is not covered in Hingston-Rademacher~\cite{Hingston-Rademacher}. 
 
\begin{theorem} \label{thm: resonance CP1}
Let $M=\C P^1$, and $\K$ a coefficient field of characteristic $\neq 2$. Then, for every Finsler metric $g$ on $M$, there exist positive constants $\lambda, C$ depending only on $g$ and $\K$ such that for all homogeneous elements $X \in H_*(\Lambda;\K) \setminus \{0\}$ and $x\in H^*(\Lambda;\K)\setminus \{0\}$ we have
\[ |\lambda \cdot \deg(X) - \Cr(X)| \leq C\]
and 
\[ |\lambda \cdot \deg(x) - \ccr(x)| \leq C.\]
\end{theorem}

\begin{remark}\label{rmk: actual condition}
In the case of $\H P^d$ and $\Ca P^2$, the algebraic phenomenon underlying resonance is that a nonzero multiple of $y^k$ can be obtained from $y^k a^{d-1}$
by repeated application of the string bracket $P_u=\{-,u\}$. In the case of $\C P^d$, $d\ge 2$, the same is true for a nonzero class living in the same degree as $y^k$. This phenomenon is straightforward in the case of spheres of dimension $>2$ with arbitrary coefficients, and for $M$ as in the theorem it is essentially equivalent to string point invertibility. This explains why the arguments of~\cite{Hingston-Rademacher} work with arbitrary coefficients for spheres of dimension $>2$, while the arguments in this paper work only with coefficients of a well-chosen finite characteristic for other CROSS.
\end{remark}

We will deduce Theorem~\ref{thm: resonance} from the more general Theorem~\ref{thm: resonance-Liouville} below, and Theorem~\ref{thm: resonance CP1} from the more general Theorem~\ref{thm: resonance CP1-Liouville}, both about Reeb flows on the unit cotangent bundles of $\C P^d$, $\H P^d$ or $\Ca P^2$.
The proof generally follows the arguments of~\cite{Hingston-Rademacher}, with two new ingredients that we will list separately.

{\bf Critical values for Liouville domains. }
We begin with the general setup of a $2n$-dimensional Liouville domain $V$. It is understood, whenever we consider degrees of homology or cohomology classes, that the canonical bundle of $V$ is trivializable and trivialized, so that the symplectic homology and the symplectic cohomology of $V$ are $\Z$-graded. We denote by $SH_*(V)$ the symplectic homology with coefficients in $\K$.  We denote by $SH^{\leq\alpha}_*(V)$ the filtered symplectic homology in action $\leq\alpha$ and by $SH_{>\alpha}^*(V)$ the filtered symplectic cohomology in action $>\alpha$, and recall from~\cite{CHO-PD} the definition of the critical levels of nonzero classes $X\in SH_*(V)$ and $x\in SH^*_{>0}(V)$,
\begin{align*}
\Cr(X) &= \inf\{\alpha\in\R\mid X\in\im\bigl(SH_*^{\leq\alpha}(V)\to SH_*(V)\bigr)\}, \\
\ccr(x) &= \sup\bigl\{\alpha\in\R\mid x\in\im\bigl(SH^*_{>\alpha}(V)\to SH^*_{>0}(V)\bigr)\bigr\}.
\end{align*}
The homological critical levels satisfy {\em subadditivity} with respect to 
the pair-of-pants product (see~\cite{CHO-PD}),
$$
\Cr(XY) \leq \Cr(X) + \Cr(Y),
$$
and the {\em non-Archimedean property}
$$
\Cr(X+Y) \leq \max\{\Cr(X), \Cr(Y)\}.
$$
Since the BV operator $\Delta:SH_*(V)\to SH_{*+1}(V)$ does not increase the action, we also have {\em BV monotonicity}
$$
\Cr(\Delta X)\leq \Cr(X). 
$$
In view of subadditivity and the non-Archimedean property, this implies that the associated bracket satisfies
\begin{equation}\label{eq: bracket subadd}
\Cr(\{X,Y\}) \leq \Cr(X) + \Cr(Y).
\end{equation}
The cohomological critical levels satisfy {\em superadditivity} with respect to 
the secondary pair-of-pants product (see~\cite{CHO-PD}),
$$
\ccr(xy)\ge \ccr(x)+\ccr(y),
$$
and the {\em non-Archimedean property} 
$$
\ccr(x+y)\ge \min\{\ccr(x),\ccr(y)\}.
$$
The BV operator $\Delta^*:SH^*_{>0}(V)\to SH^{*-1}_{>0}(V)$ does not decrease the action and we obtain {\em cohomological BV monotonicity} in the form 
$$
\ccr(\Delta^*x)\ge \ccr(x).
$$
As a consequence, the associated cohomological bracket satisfies
\begin{equation}\label{eq: bracket superadd}
\ccr(\{x,y\}) \geq \ccr(x) + \ccr(y).
\end{equation}

The first new ingredient in the proof of Theorem~\ref{thm: resonance} is the following duality between homological and cohomological critical levels.
Given a $\K$-vector space $V$, we denote $V^\vee=\Hom_\K(V,\K)$ its linear dual. 

\begin{lemma}[Duality for critical levels]\label{lma: duality}
Let $k>n+1$. We then have $SH_k(V)^{\vee} \cong SH^k_{>0}(V)$,
and for all nonzero classes $A \in SH_k(V)$ and $a \in SH^k_{>0}(V)$ we have
\begin{align*}
  \Cr(A) &= \sup \{ \ccr(b) \mid \deg(b)=k,\, \left< A,b\right> \neq 0 \},\\
  \ccr(a) &= \inf \{\Cr(B) \mid \deg(B)=k,\, \left<B,a\right> \neq 0 \}. 
\end{align*}
\end{lemma}


\begin{proof}
The proof follows an argument of Leclercq-Zapolsky \cite[Theorem 35, Duality]{Leclercq-Zapolsky}. For $\alpha \geq  0$ consider the long exact sequences
\begin{gather*}
   \cdots SH_k^{\leq\alpha}(V) \stackrel{I_\alpha}\longrightarrow SH_k(V) \stackrel{R_\alpha}\longrightarrow SH_k^{>\alpha}(V)\cdots, \\
   \cdots SH^k_{>\alpha}(V) \stackrel{i_\alpha}\longrightarrow SH^k(V) \stackrel{r_\alpha}\longrightarrow SH^k_{\leq\alpha}(V)\cdots
\end{gather*}
Note that $r_\alpha$ and $I_\alpha$ are adjoints for the usual pairing between cohomology and homology, and similarly for $R_\alpha$ and $i_\alpha$.
Since
$SH^*_{\leq 0}(V)\cong H_{n-*}(V)$ is supported in degrees $1-n,\dots,n$, for $k>n+1$ we have a canonical isomorphism $SH^k_{>0}(V)\cong SH^k(V)$ and we can rewrite the critical levels as	
\begin{align*}
\Cr(A) &= \inf\{\alpha\in\R\mid A\in\im I_\alpha = \ker R_\alpha\}, \\
\ccr(a) &= \sup\{\alpha\in\R\mid a\in\im i_\alpha = \ker r_\alpha\}.
\end{align*}
Let us now prove $\Cr(A) = \sup \{ \ccr(b) \mid \deg(b)=k,\, \left< A,b\right> \neq 0 \}$. First of all, if $\brat{A,b} \neq 0,$ then
\[\ccr(b) \leq \Cr(A).\]
Indeed, otherwise for $\ccr(b) > \alpha > \Cr(A)$ we have $r_\alpha(b) = 0$, while there exists $A_{\al} \in SH_k^{\leq\alpha}(V)$ with $A = I_\alpha(A_{\al})$. Then $\brat{A,b} = \brat{I_\alpha(A_{\al}),b} =  \brat{A_{\al}, r_\alpha(b)} = 0$, a contradiction. This yields
\[\Cr(A) \geq \sup \{ \ccr(b) \mid \deg(b)=k,\, \left< A,b\right> \neq 0 \}.\]
For the reverse inequality we use the fact that the coefficients are in a field, so that $SH^k_{>\alpha}(V)\cong SH_k^{>\alpha}(V)^\vee$ for all $\alpha$. Consider $\alpha < \Cr(A).$ Then $R_\alpha(A) \neq 0,$ and as the coefficients are in a field, there exists $b_{\alpha}$ with $\brat{A, i_\alpha(b_{\alpha})} = \brat{R_\alpha(A), b_{\alpha}} \neq 0$. Hence $b = i_\alpha(b_{\alpha})$ is a cohomology class with $\brat{A,b} \neq 0$, and $\ccr(b) \geq \alpha$ because $r_\alpha(b) = r_{\al} \circ i_{\al}(b_{\al}) = 0$. Since $\alpha < \Cr(A)$ was arbitrary, we get
\[\Cr(A) \leq \sup \{\ccr(b) \mid \deg(b)=k,\, \left< A,b\right> \neq 0 \}.\]
The second equality $\ccr(a) = \inf \{\Cr(B) \mid \deg (B)=k,\, \left<B,a\right> \neq 0 \}$ is dual to the first one. As above, $\brat{B,a} \neq 0$ implies $\ccr(a) \leq \Cr(B),$ yielding
\[\ccr(a) \leq \inf \{\Cr(B) \mid \deg(B)=k,\, \left<B,a\right> \neq 0 \}.\]
In the reverse direction, if $\alpha > \ccr(a),$ then $r_{\alpha}(a) \neq 0$ and therefore, since the coefficients are in a field, there exists $B_{\alpha}$ with $\brat{I_{\alpha}(B_{\alpha}), a} = \brat{B_{\alpha},r_{\al}(a)} \neq 0$. Then $B = I_{\alpha}(B_{\alpha})$ is the desired homology class with $\brat{B,a} \neq 0$ and $\Cr(B) \leq \alpha$. This yields
\[\ccr(a) \geq \inf \{\Cr(B) \mid \deg(B)=k,\, \left<B,a\right> \neq 0 \},\]
finishing the proof.  
\end{proof}

{\bf Critical values associated to CROSS.}
Consider now the cotangent bundle $T^*M$ of a closed manifold $M$ with its canonical Liouville form $\lambda_{\rm st}=p\,dq$. 
Let $\pi:T^*M\to M$ be the projection. We assume that the manifold $M$ is orientable.\footnote{In this case the canonical bundle is isomorphic to the (dual of) the complexification of $\det(TM)$, and as such admits a canonical trivalization up to homotopy, induced by any of the two trivializations up to homotopy of $\det(TM)$. This determines a canonical grading on symplectic homology and cohomology. In the nonorientable case, although only the square of the canonical bundle may be trivializable, a canonical grading still exists by a construction of Abouzaid~\cite[\S9.4.5]{Abouzaid-cotangent}}

Any fiberwise starshaped domain (compact with smooth boundary) $V\subset T^*M$ becomes a Liouville domain with the restriction of $\lambda_{\rm st}$, so we can speak of the critical values on its symplectic (co)homology. The characteristics on the boundaries of such $V$ correspond to the Reeb orbits of arbitrary contact forms defining the standard contact structure on the unit sphere cotangent bundle $S^*M$.
Recall Viterbo's theorem (see~\cite{Abouzaid-cotangent})
$$
  SH_*(V;\eta)\cong H_*\Lambda,\qquad SH^*(V;\eta)\cong H^*\Lambda. 
$$
Here $\eta$ is the local system on $\Lambda T^*M$ induced from $\Lambda M$ by transgressing the second Stiefel Whitney class $w_2(M)$: the fiber of $\eta$ over a coefficient field $\K$ is a vector space of dimension 1, and the holonomy along a loop $\gamma:S^1\to \Lambda T^*M$ is equal to $(-1)^{\langle w_2(M),T_\gamma\rangle}$, where 
$T_\gamma:S^1\times S^1\to M$, $T_\gamma(\theta,t)=\pi(\gamma(\theta)(t))$ 
is the 2-torus in $M$ determined by $\gamma$ seen as a loop of loops. The discussion of critical values from the previous paragraph applies verbatim to symplectic homology with coefficients in $\eta$ because, in the terminology of~\cite[Appendix~A]{CHO-MorseFloerGH}, the local system $\eta$ is compatible with products.  

We will repeatedly use the identification $SH^*(V;\eta)\simeq SH_*(V;\eta)^\vee$, which induces an evaluation map 
$$
SH^*(V;\eta)\otimes SH_*(V;\eta)\to\K  
$$
that coincides with the Kronecker product $H^*\Lambda\otimes H_*\Lambda\to\K$ through the Viterbo isomorphism.\footnote{In general, given a local system $\cL$ on $\Lambda T^*M$ we have a canonical identification $SH^*(V;\cL^\vee)\simeq SH_*(V;\cL)^\vee$, where $\cL^\vee$ is the \emph{dual local system} with fiber the dual vector space of $\cL$ and holonomy defined by adjunction from the holonomy of $\cL$. In our case, because the holonomy of $\eta$ has image in $\pm 1$, the dual local system $\eta^\vee$ is isomorphic to $\eta$.}     

Note that $V$ is fiberwise convex iff it is the unit disk cotangent bundle of a Finsler metric, in which case Viterbo's isomorphism respects the filtrations
$$
  SH_*^{\leq\alpha}(V;\eta)\cong H_*\Lambda_\alpha,\qquad SH^*_{>\alpha}(V;\eta)\cong H^*(\Lambda,\Lambda_\alpha). 
$$
Thus the following proposition is more special than Gromov's result~\cite[Theorem 7.3]{gromov-metric} in that it applies only to very special manifolds $M$, but more general in that it extends it from Finsler geodesic flows  to Reeb flows. 

\begin{proposition}\label{prop: mini Gromov}
Let $M$ be a simply connected CROSS and $V\subset T^*M$ a fiberwise starshaped domain. For every coefficient field $\K$ there exist positive constants $\lambda^*, \lambda_*, C$ (depending only on $V$ and $\K$) such that, for every $X \in SH_*(V;\eta)\setminus\{0\}$ and every $\xi\in SH^*(V;\eta)\setminus \{0\}$, we have
\[\lambda^* \deg(X)-C \leq \Cr(X) \leq \lambda_* \deg(X)+C\]
and
\[\lambda^* \deg(\xi)-C \leq \ccr(\xi) \leq \lambda_* \deg(\xi)+C.\]
\end{proposition}

In the sequel, $\mrm{const}$ denotes a constant depending only on
$V$ and $\K$, but not on the (co)homology classes involved in the estimates.

\begin{proof}
Recall from~\S\ref{ss:coh-Uebele} the Uebele classes $y\in SH_{2n+i-2}(V;\eta)\cong \break H_{2n+i-2}\Lambda$ (homological) and $\theta\in SH^{i-1}_{>0}(V;\eta)\cong H^{i-1}(\Lambda,\Lambda_0)$ (cohomological). 
By the subadditivity property of homological critical levels, the sequence $a_k=\Cr(y^k)$ is subadditive, i.e.~$a_{k+\ell}\leq a_k+a_\ell$ for all $k,\ell\geq 1$. 
By Fekete's lemma~\cite[Proposition~6.2.1]{Coornaert-DTSD-SMF}, 
for a subadditive sequence $(a_k)$ the limit $\lim_{k\to\infty}\frac{a_k}{k}$ exists and is equal to $\inf_k\frac{a_k}{k}$. Applying this to the sequence $\Cr(y^k)$, and its dual version to the superadditive sequence $\ccr(\theta^k)$, we get limits
$$
\mu_* = \lim_{k \to \infty} \frac{1}{k} \Cr(y^k) = \inf_k\frac{1}{k} \Cr(y^k),\quad
\mu^* = \lim_{k \to \infty} \frac{1}{k} \ccr(\theta^k) = \sup_k\frac{1}{k} \ccr(\theta^k).
$$
Fix $\varepsilon > 0$ and let $k_0 \in \mathbb{N}$ be such that
\begin{equation}\label{eq:mustar-eps}
  \frac{\Cr(y^k)}{k} \leq \mu_*+\varepsilon \quad\text{and}\quad \frac{\ccr(\theta^k)}{k} \geq \mu^*-\varepsilon \quad\text{for all }k \geq k_0.
\end{equation}
Recall from Corollary~\ref{cor:CROSS-structure} that all $X \in SH_*(V;\eta)$ and $\xi \in SH^*_{>0}(V;\eta)$
can be written as $X = Z y^k$ and $\xi = \zeta \theta^\ell$, respectively, with $Z,\zeta$ in a fixed finite dimensional space $F$.
Suppose now that $X=Zy^k$ is nonzero. 
By Lemma~\ref{lma: duality}, $\Cr(X) \geq \ccr(\xi)$ for some (in fact any) cohomology class $\xi$ with $\la X,\xi\ra\neq 0$. Since $\deg(X)=\deg(\xi)$, Corollary~\ref{cor:CROSS-structure}(c) implies that $\xi$ must have the form $\xi = \zeta_0\theta^k + \zeta_1\theta^{k+1} + \zeta_2\theta^{k+2}$ with $\zeta_i\in F$ not all zero. For $k\geq k_0$ we conclude
\begin{equation}\label{eq:Cr-lower-est}
\Cr(X) \geq \ccr(\xi) \geq \ccr(\theta^{k}) - \const \geq (\mu^*-\varepsilon) k - \const.
\end{equation}
Here the second inequality follows from the reverse non-Archimedean property
and superadditivity of the cohomological critical value,
and the third one from~\eqref{eq:mustar-eps}. 
On the other hand, for $k\geq k_0$ we get from subadditivity 
\[
\Cr(X) \leq \Cr(y^k) + \const \leq (\mu_*+\varepsilon) k + \const.
\]
By making the constant $\const$ sufficiently large we can ensure that the inequalities hold also for $k < k_0$. 
Since $|\deg(X)-k|y|\,|\leq\const$, setting $\lambda_* = (\mu_*+\varepsilon)/|y|$ and $\lambda^* = (\mu^*-\varepsilon)/|y|$
we get
\[
\lambda^* \deg(X) - \const \leq  \Cr(X) \leq \lambda_* \deg(X) + \const
\]
for all $X$. This finishes the proof in the homological case. 

Suppose now $\xi=\zeta\theta^\ell$ is nonzero. Then, again by Lemma~\ref{lma: duality}, $\ccr(\xi)\leq \Cr(X)$ for some (in fact any) homology class $X$ with $\la \xi,X\ra\neq 0$. Since $\deg(X)=\deg(\xi)$, Corollary~\ref{cor:CROSS-structure}(c) implies that $X$ must have the form $X = Z_{-2}y^{\ell-2} + Z_{-1}y^{\ell-1} + Z_0y^\ell$ with $Z_i\in F$ not all zero. We find as above 
\begin{equation} \label{eq:ccr-upper-est}
\ccr(\theta^\ell)-\const \le \ccr(\xi)\le \Cr(X)\le \Cr(y^{\ell-2})+\const,
\end{equation}
which implies 
$$
(\mu^*-\eps)\ell -\const \le \ccr(\xi)\le (\mu_*+\eps)\ell +\const 
$$
for $\ell\ge k_0+2$. By making the constant $\const$ sufficiently large we can ensure that the inequalities hold also for $\ell<k_0+2$. To conclude we compute $\deg(\xi)=\deg(\zeta)+\ell(\deg(\theta)+n-1)$ and use the fact that 
$$
\deg(\theta)+n-1 = \deg(y)-n=|y|.
$$
We find $|\deg(\xi)-\ell |y||\le \const$, and by setting $\lambda_*=(\mu_*+\eps)/|y|$ and $\lambda^*=(\mu^*-\eps)/|y|$ as above, we get 
\[
\lambda^* \deg(\xi) - \const \leq  \ccr(\xi) \leq \lambda_* \deg(\xi) + \const
\]
for all $\xi$. This concludes the proof in the cohomological case. 
\end{proof}

\begin{remark}
As seen in the proof, the constants $\lambda^*$, $\lambda_*$ can be chosen arbitrarily close to $\mu^*=\lim_k\frac{1}{k}\ccr(\theta^k)$, respectively $\mu_*=\lim_k \frac{1}{k} \Cr(y^k)$.
\end{remark}

{\bf Resonance theorem.} 
We will prove the following generalizations of Theorems~\ref{thm: resonance} and~\ref{thm: resonance CP1} from Finsler metrics to Reeb flows. 

\begin{theorem}\label{thm: resonance-Liouville}
Let $M$ be one of the CROSS $\C P^d$ with $d\geq 2$, $\H P^d$ with $d\geq 1$, or $\Ca P^2$ (where we set $d=2$). Assume that $d+1$ is prime and let $\K$ be a coefficient field of characteristic ${\rm char}(\K) = d+1$.
Then, for every fiberwise starshaped domain $V\subset T^*M$,
there exist positive constants $\lambda,C$ depending only on $V$ and $\K$ such that for all homogeneous elements $X \in SH_*(V;\eta) \setminus \{0\}$ and $\xi\in SH^*(V;\eta)\setminus \{0\}$ we have
\[ |\lambda \cdot \deg(X) - \Cr(X)| \leq C\]
and
\[ |\lambda \cdot \deg(\xi)- \ccr(\xi)|\le C.\]
\end{theorem}

\begin{theorem} \label{thm: resonance CP1-Liouville}
Let $M=\C P^1$, and $\K$ a coefficient field of characteristic $\neq 2$. Then, for every fiberwise starshaped domain $V\subset T^*M$,
there exist positive constants $\lambda,C$ depending only on $V$ and $\K$ such that for all homogeneous elements $X \in SH_*(V;\eta) \setminus \{0\}$ and $\xi\in SH^*(V;\eta)\setminus \{0\}$ we have
\[ |\lambda \cdot \deg(X) - \Cr(X)| \leq C\]
and
\[ |\lambda \cdot \deg(\xi)- \ccr(\xi)|\le C.\]
\end{theorem}

\begin{remark}
We find it interesting that the assumptions for Theorem~\ref{thm: resonance-Liouville} are the same as for string point invertibility in Theorem~\ref{thm:string-point}, although there is no direct implication between the two theorems. The necessity of these assumptions rather comes from the method of proof. 
\end{remark}

Consider the setup of Theorem~\ref{thm: resonance-Liouville}. 
Recall from Theorem~\ref{thm:loop_BV_CROSS_over_K_char_divides_d+1} the generators $y,a,u$ of $SH_*(V;\eta)\cong H_*\Lambda$ and their shifted degrees
$$
|y|=i(d+1)-2=n+i-2,\qquad |a|=-i,\qquad |u|=i-1,
$$
where $i=2,4,8$ for $\C P^d$, $\H P^d$, $\Ca P^2$, respectively, and
$n=\dim M=id$. 
The second new ingredient in the proof of Theorems~\ref{thm: resonance-Liouville} and~\ref{thm: resonance} is the following

\begin{lemma}[Effect of string point invertibility]\label{lma: spi}
In the setup of Theorem~\ref{thm: resonance-Liouville}, 
for all $k\ge 1$ and $\alpha, \beta \in \K$ with $\beta \neq 0$ we have
\[\Cr(\alpha y^{k-1} u + \beta y^k u a^d) \geq \Cr(y^{k}) - \mrm{{const}}.\]
\end{lemma}

\begin{proof}
Recall from Theorem~\ref{thm:loop_BV_CROSS_over_K_char_divides_d+1} the expression of the BV operator $\Delta$. 

Let $X = \alpha y^{k-1} u + \beta y^k u a^d$ with $\alpha,\beta\in\K$ and $\beta \neq 0$. Then $\Delta X = d \beta y^{k} a^{d-1}$ and, using BV monotonicity and the invertibility of $\beta d$ in $\bK$, we find
$$
\Cr(X) \geq \Cr(\Delta X) = \Cr(\beta d y^{k} a^{d-1}) =\Cr(y^{k} a^{d-1}).
$$
Consider first the case of $\H P^d$ and $\Ca P^2$. The result follows by repeated application of $\{\cdot, u\}$: for $1\leq\ell\leq d-1$ we compute
$$
\{y^{k} a^\ell, u\}
= \Delta(y^k a^\ell u)-\Delta(y^k a^\ell)u - y^k a^\ell\Delta(u) 
=  \ell y^{k} a^{\ell-1}.
$$ 
By~\eqref{eq: bracket subadd} and invertibility of $\ell$ in $\K$ this implies
$$
\Cr(y^k a^\ell) \geq \Cr(\ell y^k a^{\ell-1}) - \Cr(u) = \Cr(y^k a^{\ell-1})-\Cr(u).
$$
Applying this $d-1$ times we conclude
$$
\Cr(X) \geq \Cr(y^k a^{d-1}) \geq \Cr(y^{k}) - (d-1) \Cr(u),
$$
which yields the desired inequality with constant $(d-1) \Cr(u)$.

Consider now the case of $\C P^d$ with $d\ge 2$. The previous argument gives 
$$
\Cr(X) \geq \Cr(y^k a^{d-1}) \geq \Cr(y^{k}a) - (d-2) \Cr(u). 
$$
The final step of the computation is now $\{y^ka,u\}=\Delta(y^kau)=y^k+\frac{(d+1)d}{2} y^{k+1}a^d$. This gives 
\begin{align*}
\Cr(y^k) & \le \max(\Cr\{y^ka,u\},\Cr (y^{k+1}a^d))\\
& \le \max(\Cr (y^k a) +\Cr u,\Cr (y^ka^{d-1}) +\Cr y+\Cr a)\\
& \le \max(\Cr (y^ka^{d-1})+(d-2)\Cr u,\Cr (y^ka^{d-1})+\Cr y+\Cr a)\\
& = \Cr(y^ka^{d-1})+\max((d-2)\Cr u,\Cr y+\Cr a).
\end{align*} 
This yields the desired inequality with a constant equal to $\max((d-2)\Cr u ,\Cr y +\Cr a )$.
\end{proof}


Recall the cohomological Uebele class $\theta$ defined in~\S\ref{ss:coh-Uebele}.
Lemmas \ref{lma: duality} and \ref{lma: spi} now imply

\begin{corollary}\label{cor: theta and Theta}
In the setup of Theorem~\ref{thm: resonance-Liouville}, the homological Uebele class $y$ and the cohomological Uebele class $\theta$ satisfy
$$
\Cr(y^{k}) - \mrm{const} \leq \ccr(\theta^{k}) \leq \Cr(y^{k}) + \mrm{const}\qquad\text{for all }k\geq 0.
$$
\end{corollary}

\begin{proof}
We first treat the case of $\H P^d$ or $\Ca P^2$. By Lemma~\ref{lem:coh-Uebele} and the Duality Lemma~\ref{lma: duality} we get
\begin{equation*}
\ccr(\theta^{k+1}) = \inf_{\beta \in \K^{\ast} } \Cr(\beta y^k u a^d) = \Cr(y^k u a^d).
\end{equation*}
By Lemma~\ref{lma: spi} and subadditivity of $\Cr$ this gives the lower bound
$$
   \ccr(\theta^{k+1})\geq \Cr(y^k)-\const \geq \Cr(y^{k+1})-\const.
$$
Using the subadditivity of $\Cr$ and the superadditivity of $\ccr$ we obtain the upper bound
$$
   \ccr(\theta^{k-1})\leq \ccr(\theta^{k+1}) = \Cr(y^k u a^d) \leq \Cr(y^{k-1})+\const.
$$
We now treat the case of $\C P^d$ with $d\ge 2$. By the non-Archimedean property of $\ccr$ we find 
\begin{align*}
\ccr(\theta^{k+1}) & \ge \min \Bigl\{\ccr\Bigl(\theta^{k+1}+\frac{(d+1)d}{2} \theta^{k+2} a^d\Bigr), \ccr \Bigl(\frac{(d+1)d}{2} \theta^{k+2} a^d\Bigr)\Bigr\}\\
& \geq \min \Bigl\{\ccr\Bigl(\theta^{k+1}+\frac{(d+1)d}{2} \theta^{k+2} a^d\Bigr),\ccr (\theta^{k+2} a^d)\Bigr\}. 
\end{align*}
We now use the fact that $\ccr (\theta a^d)>0$ (since the only homology or cohomology classes supported on level 0 are the constants). Then
\begin{equation}\label{eq:alex}
  \ccr (\theta^{k+2}a^d) \ge \ccr (\theta^{k+1}) + \ccr (\theta a^d) > \ccr (\theta^{k+1}),
\end{equation}
and therefore\footnote{It is necessary to bring into play the class $\theta^{k+1}+\frac{(d+1)d}{2} \theta^{k+2} a^d$ because it lies in the image of the cohomological BV operator. It is this property that makes it the subject of Lemma~\ref{lem:coh-Uebele}.} 
$$
\ccr(\theta^{k+1}) \geq \ccr\Bigl(\theta^{k+1}+\frac{(d+1)d}{2} \theta^{k+2} a^d\Bigr).
$$
By Lemma~\ref{lem:coh-Uebele}, the Duality Lemma~\ref{lma: duality}, Lemma~\ref{lma: spi} and subadditivity of $\Cr$ we find  
\begin{align*}
  \ccr\Bigl(\theta^{k+1}+& \frac{(d+1)d}{2} \theta^{k+2} a^d\Bigr)\\
  & = \inf \{ \Cr(X)\, : \, \langle \theta^{k+1}+\frac{(d+1)d}{2} \theta^{k+2} a^d,X \rangle\neq 0\} \\
  & = \inf_{(\alpha, \beta) \in \K \times \K^{\ast} } \Cr(\alpha y^{k-1} u + \beta y^k u a^d)\\
& \ge \Cr(y^k)-\const \\
& \ge \Cr(y^{k+1})-\const,
\end{align*}
so that 
$$
\ccr(\theta^{k+1}) \ge  \Cr(y^{k+1})-\const.
$$
The reverse inequality follows via 
\begin{align*}
  \ccr(\theta^{k-1})
  & \le \ccr(\theta^{k+1}) \\
  &= \min\{\ccr(\theta^{k+1}),\ccr(\theta^{k+2}a^d)\} \cr
  &\leq \ccr\Bigl(\theta^{k+1}+\frac{(d+1)d}{2} \theta^{k+2} a^d\Bigr) \\
  & = \inf \Bigl\{\Cr(X)\, : \, \langle \theta^{k+1}+\frac{(d+1)d}{2} \theta^{k+2} a^d,X\rangle\neq 0\Bigr\} \\
  &= \inf_{(\alpha, \beta) \in \K \times \K^{\ast} } \Cr(\alpha y^{k-1} u + \beta y^k u a^d) \\
  & \le \max\{\Cr(y^{k-1} u),\Cr(y^k u a^d)\} \\
& \le \Cr(y^{k-1})+\const. 
\end{align*}
Here we have used in successive order
superadditivity of $\ccr$, 
inequality~\eqref{eq:alex},
the non-Archimedean property of $\ccr$, 
the Duality Lemma~\ref{lma: duality},
Lemma~\ref{lem:coh-Uebele},
the non-Archimedean property of $\Cr$, 
and subadditivity of $\Cr$.
\end{proof}

\begin{proof}[Proof of Theorem~\ref{thm: resonance-Liouville}] 
We use the notation and results from the proof of Proposition~\ref{prop: mini Gromov}, in particular the definition of the limits $\mu^*$ and $\mu_*$. 
By Corollary~\ref{cor: theta and Theta} the limits are equal, $\mu^* = \mu_* = : \mu$, and therefore, again by Fekete's lemma~\cite[Proposition~6.2.1]{Coornaert-DTSD-SMF}, 
\begin{equation}\label{eq:mu}
\frac{1}{k} \ccr(\theta^k) \leq \mu \leq \frac{1}{k} \Cr(y^k) \quad\text{for all }k.
\end{equation}
Recall from Corollary~\ref{cor:CROSS-structure} that each nonzero $X \in SH_*(V;\eta)$
can be written as $X = Z y^k$ with $Z$ in a fixed finite dimensional space.
Hence, on the one hand, Corollary~\ref{cor: theta and Theta} and~\eqref{eq:mu} imply
\[
\Cr(X) = \Cr(Zy^k) \leq \Cr(y^k) + \const \leq \ccr(\theta^{k}) + \const \leq k \mu + \const.
\]
On the other hand,
the first two inequalities in~\eqref{eq:Cr-lower-est}, 
Corollary~\ref{cor: theta and Theta} and~\eqref{eq:mu} yield
\[
\Cr(X)
\geq \ccr(\theta^{k}) -\const \geq \Cr(y^{k}) - \const \geq k \mu - \const.
\]
In summary, we obtain that
$$
  |\Cr(X)-k\mu|\leq\const\quad\text{for all }X=Zy^k.
$$
For the degrees, $\deg(y^k)-n=|y^k|=k\nu_*$ with $\nu_*:=|y|$ implies
$$
  |\deg(X)-k\nu_*|\leq\const\quad\text{for all }X=Zy^k.
$$
With $\lambda_*:=\mu/\nu_*$ the last two displayed equations combine to
$$
|\Cr(X)-\lambda_*\deg(X)| \leq |\Cr(X)-k\mu| + \lambda_*|k\nu_*-\deg(X)| \leq \const.
$$

Consider now a nonzero $\xi\in SH^*_{>0}(V;\eta)$. According to Corollary~\ref{cor:CROSS-structure} we can write $\xi=\zeta\theta^\ell$ with $\zeta$ in a fixed finite dimensional space. Thus, on the one hand, using superadditivity of $\ccr$ and equation~\eqref{eq:mu} we obtain 
$$
\ccr(\xi)\ge \ccr(\theta^\ell)+\const \ge \ell\mu+\const, 
$$
and on the other hand, using equation~\eqref{eq:ccr-upper-est} and Corollary~\ref{cor: theta and Theta} we find 
$$
\ccr(\xi)\le \ccr(y^{\ell-2})+\const \le \ccr(\theta^{\ell-2})+\const \le \ell\mu+\const.
$$
In summary 
$$
|\ccr(\xi)-\ell\mu|\le \const \quad \mbox{ for all } \xi=\zeta\theta^\ell.
$$
We now express the degrees $\deg(\xi)=\deg(\zeta)+\ell\deg(\theta)+\ell(n-1) = \deg(\zeta)+\ell\nu^*$, with $\nu^*=\deg(\theta)+n-1$. This implies 
$$
|\deg(\xi)-\ell\nu^*|\le\const \quad \mbox{ for all } \xi=\zeta\theta^\ell.   
$$
Setting $\lambda^*=\mu/\nu^*$ we find 
$$
|\ccr(\xi)-\lambda^*\deg(\xi)| \leq |\ccr(\xi)-\ell\mu|+|\ell\mu-\lambda^*\deg(\xi)|\le \const. 
$$

This proves the resonance inequalities with possibly two different slopes $\lambda_*$ (homological) and $\lambda^*$ (cohomological). But $\lambda_*=\lambda^*$ because 
$$
\nu_* = \deg(y)-n = \deg(\theta)+n-1=\nu^*.
$$ 

This finishes the proof of Theorem~\ref{thm: resonance-Liouville}, and thus of Theorem~\ref{thm: resonance}. 
\end{proof}

\begin{proof}[Proof of Theorem~\ref{thm: resonance CP1-Liouville}] As seen in the proof of Theorem~\ref{thm: resonance-Liouville}, we need to establish an estimate of the form
\begin{equation} \label{eq:thetakyk-CP1}
\Cr(y^k)-\const \le \ccr(\theta^k)\le \Cr(y^k)+\const,
\end{equation}
as in Corollary~\ref{cor: theta and Theta}. We use the fact that, with coefficients in a field $\K$ of characateristic $\neq 2$, loop homology $H_*\Lambda$ has rank $1$ in each degree, and the same holds, of course, for loop cohomology. 

We first find 
$$
\ccr(\theta^k)\le \ccr(\theta^{k+1})=\Cr(y^k x)\le \Cr(y^k)+\const,
$$
by superadditivity, subadditivity, and duality $\langle\theta^{k+1},y^k x\rangle\neq 0$. 

For the converse inequality we use the BV operator 
\begin{align*}
\ccr(\theta^k)=\Cr(y^{k-1}x) \ge \Cr(\Delta(y^{k-1}x)) =\Cr((2k-1)y^{k-1}).
\end{align*}
We distinguish two cases: 
\begin{enumerate}
\item Either $2k-1\neq 0$ in $\K$, in which case $\Cr((2k-1)y^{k-1})=\Cr(y^{k-1})\ge \Cr(y^k)-\const$, by subadditivity.
\item Or $2k-1=0$ in $\K$, in which case $2k-3\neq 0$ in $\K$ because the characteristic of $\K$ is $\neq 2$. In this case we find using superadditivity, duality, BV operator, and subadditivity:
\begin{align*}
\ccr(\theta^k)& \ge \ccr(\theta^{k-1})=\Cr(y^{k-2}x)\\
&\ge \Cr(\Delta(y^{k-2}x)) = \Cr((2k-3)y^{k-2})=\Cr(y^{k-2})\\ 
&\ge \Cr(y^k)-\const. 
\end{align*} 
\end{enumerate}
In both cases we obtain $\ccr(\theta^k)\ge \Cr(y^k)-\const$, and this proves~\eqref{eq:thetakyk-CP1}. 

We conclude as in the proof of Theorem~\ref{thm: resonance-Liouville}.
\end{proof}

\begin{remark}
We were not able to prove resonance for $S^2$ with coefficients in a field $\K$ of characteristic $2$. 
\end{remark}

\section{Density theorem}\label{sec:density}

In this section we outline an application of Theorem~\ref{thm: resonance-Liouville} to a ``density theorem" similar to~\cite[Theorem 1.2]{Hingston-Rademacher}. 
To be consistent with~\cite[Theorem 1.2]{Hingston-Rademacher}, we rewrite the inequalities in Theorem~\ref{thm: resonance-Liouville} in terms of the {\em global mean frequency} $\alpha=\lambda^{-1}$ as\footnote{Under the assumptions of Theorem~\ref{thm: resonance-Liouville} we have $\lambda=\frac1{|y|}\lim_k\frac{\ccr(\theta^k)}{k}\ge \frac{\ccr(\theta)}{|y|}>0$.}
\begin{equation}\label{eq:global-mean-freq}
  |\deg(X) - \alpha\Cr(X)| \leq C,\qquad |\deg(\xi)-\alpha\ccr(\xi)|\leq C.
\end{equation}
For a closed Reeb orbit $\gamma$ of a contact form $\beta$ we will consider the following quantities: its length (or action) $\ell(\gamma)=\int_\gamma\beta$; 
its Bott-Long index $i(\gamma)$, defined as the maximal lower continuous extension of the Conley-Zehnder index to arbitrary paths~\cite[Definition~6.1.10 with $\omega=1$]{Long-book} (see also~\cite[Remark~1.8]{Hein-Hryniewicz-Macarini});
its {\em average index}
\[\hat{i}(\gamma) := \lim_{k \to \infty} i(\gamma^k)/k;\]
and its {\em mean frequency}
\[\alpha(\gamma) := \hat{i}(\gamma)/\ell(\gamma).\]
These definitions specialize to closed geodesics of a Finsler metric.\footnote{
The intuition behind the term ``mean frequency'' is that it measures the number of conjugate points per unit length along a closed geodesic.}
\footnote{The index $i(\gamma)$ and mean index $\hat{i}(\gamma)$ are defined from the Bott-Long index for paths of symplectic matrices by linearizing the Reeb flow along $\gamma$ with a given starting point. However, they do not depend on the choice of starting point.
Since we did not find this fact in the literature, we provide the simple proof in Appendix~\ref{sec:basepoint}. 
}


In the next statement we recall that $i=2,4,8$ for $M=\C P^d$, $\H P^d$, and respectively $\Ca P^2$.

\begin{theorem}[Density theorem]\label{thm: dense}
  Consider a simply connected CROSS $M$ and coefficient field $\bK$ to which Theorems~\ref{thm: resonance-Liouville} and~\ref{thm: resonance CP1-Liouville} apply, and let $\alpha = \alpha_{V,\bK}$ be the global mean frequency in~\eqref{eq:global-mean-freq} it provides for a fiberwise starshaped domain $V\subset T^*M$. For $\eps > 0$ denote by $\cS_\eps$ the set of geometrically distinct simple closed Reeb orbits $\gamma$ with mean frequency
$\alpha(\gamma) \in (\alpha-\eps, \alpha+\eps)$. Then
\[\sum_{\gamma\in\cS_\eps} \frac{1}{\hat{i}(\gamma)} \geq \frac{1}{n+i-2}.\]
\end{theorem}

The proof is based on the following lemma, which is implicit in~\cite{Hingston-Rademacher} for the Finsler case. For the reader's convenience we recall its proof in our more general setting.

Let $N_*(m)$ be the number of distinct critical values of non-zero homology classes in 
$SH_*(V;\eta)\cong H_*(\Lambda; \K)$ of degree at most $m$, and $N^*(m)$ the number of distinct critical values of non-zero cohomology classes in $SH^*(V;\eta)\cong H^*(\Lambda; \K)$ of degree at most $m$. By Lemma~\ref{lem:crx=CrX} we see that, for any Liouville domain and any field of coefficients $\K$, we have 
\begin{equation} \label{eq:N*m}
N_*(m)=N^*(m).
\end{equation}
We denote this number $N(m)$. Of course, it depends on $V$ and $\K$.  

\begin{lemma}\label{lem:density-estimate}
In the setting above, for each $\eps>0$ we have
\begin{equation}\label{ineq: frequency}
  \sum_{\gamma\in\cS_\eps} \frac{1}{\hat{i}(\gamma)} \geq \limsup_{m\to \infty} \frac{N(m)}{m}.
\end{equation}
\end{lemma}

This result does not use anything on the multiplicative structure. It holds automatically whenever resonance is established. We will prove the result in a cohomological setting, but the proof could equivalently be phrased homologically. 

\begin{proof}
Suppose that an iteration $\delta = \gamma^k$ of a simple closed Reeb orbit $\gamma$ carries the critical level 
$\ccr(\xi)$ and $m_0 \leq s = \deg(\xi) \leq m,$ where $m_0>n$ is a fixed large constant depending on $\eps,$ to be chosen later. Then \[\ell(\delta) = \ccr(\xi),\]
 \begin{equation}\label{eq: mean index} \hat{i}(\delta) = s + b(\xi)\end{equation} for $|b(\xi)| \leq n$.
The inequality $|b(\xi)| \leq n$ holds because $s\in [i(\delta),i(\delta)+\nu(\delta)]\subset [\hat{i}(\delta)-n,\hat{i}(\delta)+n]$, with $\nu(\delta)$ the nullity of $\delta$.
This is a classical fact, see for example~\cite[(LF5), p.~335]{Ginzburg_Gurel_localFH}, ~\cite[Remark~1.8]{Hein-Hryniewicz-Macarini}, or~\cite[Equation (7)]{AS-torsion}, for which we briefly recall the proof. Let $I=[i(\delta),i(\delta)+\nu(\delta)]$. To prove that $s\in I$ we reason as follows. By~\cite[Theorem~5.4.1 and Corollary~6.1.9]{Long-book}, any path of symplectic matrices that starts at the identity and is sufficiently close to the linearization of the Hamiltonian flow along $\delta$ starting at any point on $\delta$ must have index contained in $I$. As a consequence, given any $C^2$-small Hamiltonian perturbation of the flow such that the $S^1$-family of closed orbits corresponding to $\delta$ splits into nondegenerate orbits, the indices of the latter must be contained in $I$. Since the class $\xi$ induces a nonzero class of the same degree in local Floer cohomology, there must exist at least one such perturbed orbit whose index is equal to $s=\deg(\xi)$, and therefore $s\in I$. The inclusion $I\subset [\hat{i}(\delta)-n,\hat{i}(\delta)+n]$ is a reformulation of the main theorem in~\cite{Liu-Long1998} (see also~\cite{Liu-Long2000} and~\cite[Theorem~10.1.2]{Long-book}). 

By Theorem~\ref{thm: resonance-Liouville}, \[\ccr(\xi) = \lambda s + c(\xi)\] for $|c(\xi)| < \const.$ Hence 
 \[ \alpha(\gamma) = \alpha(\delta) = \frac{s + b(\xi)}{\lambda s + c(\xi)} \in \left(\frac{1}{\lambda + \eps}, \frac{1}{\lambda-\eps}\right)\] for $m_0 = m_0(\eps)$ sufficiently large.\footnote{Here we implicitly use the fact that $SH^*(V;\eta)$ is supported in infinitely many positive degrees.} 
Note also that $m_0>n$ implies that $\hat{i}(\gamma) = \hat{i}(\delta)/k > 0.$ Now, given a 
simple closed Reeb orbit $\gamma$ with $\hat{i}(\gamma) > 0,$ there are at most \[\frac{m-m_0+2n}{\hat{i}(\gamma)}\] iterations $\delta = \gamma^k$ of $\gamma$ satisfying \eqref{eq: mean index} with $m_0 \leq s \leq m$, because $\hat{i}(\delta)=k\hat{i}(\gamma)$. Hence, for a given $\eps < \lambda,$ we get \[\frac{N(m)}{m} \leq \left(1-\frac{m_0-2n}{m}\right) \sum \frac{1}{\hat{i}(\gamma)}\] the sum running over the set of geometrically distinct 
simple closed Reeb orbits 
with mean frequency $\alpha(\gamma) \in (\frac{1}{\lambda + \eps}, \frac{1}{\lambda-\eps}).$ This yields \[ \sum \frac{1}{\hat{i}(\gamma)} \geq \limsup_{m\to \infty} \frac{N(m)}{m}\] as desired.
\end{proof}
  
\begin{proof}[Proof of Theorem~\ref{thm: dense}]
Let $\theta \in SH^{i-1}_{>0}(V;\eta)\simeq H^{i-1}(\Lambda,\Lambda_0; \K)$ be the cohomological Uebele class as in Section \ref{sec: res}. By super-additivity of cohomological critical levels and the fact that $\ccr(\theta)>0,$ we obtain that $c_k = \ccr(\theta^k)$ is a strictly increasing sequence. Since $\deg(\theta^k) = k(n+i-2) - (n-1)\to\infty$ as $k\to\infty$, and $SH^*_{>0}(V;\eta)=SH^*(V;\eta)$ in degrees $>n+1$, we obtain \begin{equation}\label{ineq: lower bound} \liminf_{m\to \infty} \frac{N(m)}{m} \geq \frac{1}{n+i-2}.\end{equation} Combining the estimates \eqref{ineq: frequency} and \eqref{ineq: lower bound} now finishes the proof.
 \end{proof}

As explained in \cite[Section 7]{Hingston-Rademacher}, Theorem \ref{thm: dense} has the following consequence.
Here the {\em reversibility} of a Finsler metric $F:TM\to\R_{\geq 0}$ is $\rho := \max\{F(x,-y) \mid F(x,y)=1\}\geq 1$. 

\begin{corollary}\label{cor:two-or-infinity}
Consider a simply connected CROSS $M$ and coefficient field $\bK$ to which Theorems~\ref{thm: resonance} and~\ref{thm: resonance CP1} apply, and let $\alpha = \alpha_{M, F,\bK}$ be the global mean frequency in~\eqref{eq:global-mean-freq} it provides for  a Finsler metric $F$ on $M.$ 
Suppose that $F$ has reversibility {$\rho \geq 1$} and flag curvature $K$ satisfying the inequality \[ \delta<K \leq 1 \] for
\[ \delta = \left(\frac{n-1}{n+i-2} \frac{\rho+1}{\rho}\right)^{-2}.\]
Then either $F$ possesses at least two geometrically distinct prime closed geodesics $\gamma_1, \gamma_2$ with $\alpha(\gamma_1) = \alpha(\gamma_2) = \alpha$, or $F$ possesses infinitely many geometrically distinct prime closed geodesics. 
\end{corollary}

\begin{proof}[Proof of Corollary \ref{cor:two-or-infinity}]
By \cite[Lemma 2]{Rademacher07}, the hypothesis implies that the average index of every closed geodesic $\gamma$ of $F$ satisfies \begin{equation}\label{mean-index-bound} \hat{i}(\gamma) > \sqrt{\delta} \frac{\rho+1}{\rho} (n-1) = n+i-2.\end{equation}
If $F$ possesses only finitely many geometrically distinct prime closed geodesics, then the sum in Theorem \ref{thm: dense} is finite for every $\eps>0.$ Therefore, for $\eps>0$ sufficiently small it reduces to the sum over the collection of prime closed geodesics $\gamma$ with $\alpha(\gamma) = \alpha$, and the value $S$ of this sum satisfies \[S \geq \frac{1}{n+i-2}.\] This sum cannot run over a single element $\gamma$, as in this case \eqref{mean-index-bound} would imply that \[S < \frac{1}{n+i-2}.\] 
\end{proof}

\begin{remark}
For a given $i$, recalling that $n=di$, the minimal value 
of $\delta$ is attained when $\frac{\rho+1}{\rho}$ is maximal, which occurs for $\rho=1$. Thus, the minimal value of $\delta$ is $\frac{n+i-2}{2n-2}$, which is $<1$ iff $i<n=di$, i.e., $d\ge 2$. 
This leaves out the case of $\C P^1\simeq S^2$ and $\mathbb{H}P^1 \simeq S^4,$ to which~\cite[Theorem 1.2]{Hingston-Rademacher} does not apply either.

Hingston and Rademacher~\cite{Hingston-Rademacher} proved a density theorem for spheres $S^n$ of dimension $n>2$ with lower bound $1/(2n-2)$ in even dimensions and $1/(n-1)$ in odd dimensions.\footnote{Note that $2n-2=n+i-2$ with $i=n$.} In even dimensions the bound $1/(2n-2)$ gives $\delta =((\rho+1)/2\rho)^{-2}\ge 1$ and is therefore not good enough to infer any conclusion in the spirit of Corollary~\ref{cor:two-or-infinity}. In odd dimensions the bound $1/(n-1)$ provides $\delta=((\rho+1)/\rho)^{-2}<1$, hence any reversibility allows for a suitable range of pinching.  
\end{remark}

\section{Resonance and density for finite quotients of spheres} \label{sec:finite_quotients_of_spheres}

In this section we discuss how the resonance and density theorems pass to finite quotients. As a consequence, we obtain resonance and density theorems for finite quotients of spheres, in particular for real projective spaces.
Throughout this section, all spaces and manifolds are assumed to be path connected and we fix a field $\K$. We refer to local systems of $\K$-vector spaces of dimension $k$ as \emph{$\K$-local systems of rank $k$}. 

{\bf Topological coverings. }
We begin with some topological preparations. Consider a finite cover $p:X\to\bar X$ of (path connected) topological spaces. Recall that the cover is {\em normal} if $p_*\pi_1(X)\subset\pi_1(\bar X)$ is a normal subgroup. This is equivalent to $\bar X=X/G$ for a finite group $G$ acting on $X$; the group of deck transformations $G$ is then isomorphic to $\pi_1(\bar X)/p_*\pi_1(X)$. Every cover with $\pi_1(X)=\{e\}$ is normal.

Given a $k$-to-$1$ cover $p:X\to\bar X$, let 
$\underline{\ell}$ be a $\K$-local system of rank $1$ on $X$ and define the pushforward local system $p_*\underline{\ell}$ to be the $\K$-local system of rank $k$ on $\bar X$ that, viewed as a flat bundle of $\K$-vector spaces, has stalk at $\bar x\in\bar X$ given by 
$H_0(p^{-1}(\bar x);\underline{\ell})=\oplus_{x\in p^{-1}(\bar x)} \underline{\ell}_x$.
The map $p$ induces an isomorphism
\begin{equation}\label{eq:local-coeff-iso-top}
  p_*:H_*(X;\underline{\ell})\stackrel{\simeq}{\longrightarrow} H_*(\bar X;p_*\underline{\ell}).
\end{equation}
This follows from the Leray-Serre spectral sequence for the cover $X\to \bar X$, which degenerates at the second page for degree reasons. 

Consider a normal $k$-to-$1$ cover $p:X\to\bar X=X/G$ and denote by $\Lambda p:\Lambda X\to\Lambda\bar X$ the induced map on loop spaces. Let us call a loop $\bar\gamma\in\Lambda\bar X$ {\em liftable} if some lift of $\bar\gamma$ is closed. Then every lift of $\bar\gamma$ is closed and there are $k$ lifts. Denoting $\Lambda_0\bar X\subset\Lambda\bar X$ the subspace of liftable loops, we thus have a normal $k$-to-$1$ cover
$$
  \Lambda p:\Lambda X\to\Lambda_0\bar X.
$$
Assume now for simplicity that $X$ is simply connected, so that liftability is equivalent to contractibility.  
To discuss non-liftable loops, denote by $C$ the set of conjugacy classes of $G$, which is in canonical bijection with the set of free homotopy classes of loops in $\bar X$. For $c\in C$, let $\Lambda_c\bar X\subset\Lambda\bar X$ be the component consisting of loops in the free homotopy class $c$, so that $\Lambda\bar X=\bigsqcup_{c\in C}\Lambda_c\bar X$. For $g\in G$ define
$$
  \cP_gX := \{\gamma:[0,1]\to X\mid \gamma(1)=g\gamma(0)\}.
$$
Then $p$ induces a map $\cP_gX\to\Lambda_c\bar X$ with $c$ the conjugacy class of $g$. If $\gamma\in\cP_gX$ and $h\in G$, then $h\gamma\in\cP_{hgh^{-1}}X$. Since $G$ acts transitively on $c$ by conjugation with stabilizer at $g$ the centralizer $Z(g)=\{h\in G\mid hgh^{-1}=g\}$, we see that
\begin{equation}\label{eq:G-covering}
  \bigsqcup_{g\in c}\cP_gX \to \Lambda_c\bar X
\end{equation}
is a $k$-to-$1$ map whose restriction to each component $\cP_gX \to \Lambda_c\bar X$ is a normal finite cover with deck transformation group $Z(g)$. 
  
{\bf Coverings of Liouville domains. }
Consider now a Liouville domain $(V,\lambda)$ with the free action of a finite group $G$ preserving the Liouville form $\lambda$.  Then $\lambda$ descends to a Liouville form $\bar\lambda$ on the quotient $\bar V=V/G$. Let $p:V\to\bar V$ be the projection and consider the normal finite cover $\Lambda p:\Lambda V\to \Lambda_0\bar V$. Let $\underline{\ell}$ be a $\K$-local system of rank 1 on $\Lambda V$, let 
\[
\cL=(\Lambda p)_*\underline{\ell},
\] 
and let $SH_*^0(\bar V;\cL)$ be the part of symplectic homology of $\bar V$ with local coefficients $\cL$ corresponding to liftable loops. We also assume that the canonical bundle of $V$ is trivial and is endowed with a trivialization up to homotopy. This determines a $\Z$-grading on $SH_*(V;\underline{\ell})$, and also on $SH_*^0(\bar V;\cL)$: while the canonical bundle of $\bar V$ may be nontrivial, the chain complex underlying $SH_*^0(\bar V;\cL)$ is in canonical bijection with the chain complex underlying $SH_*(V;\underline{\ell})$ for a suitable choice of Floer data, as explained below. 

\begin{theorem}[Resonance and density for quotients] \label{thm: resonance-Liouville-quotient}

(a) Assume that homological resonance holds on $V$ with $\underline{\ell}$-coefficients, i.e., there exist positive constants $\alpha,C$ such that for all homogeneous elements $X \in SH_*(V;\underline{\ell}) \setminus \{0\}$ we have
\[ |\deg(X) - \alpha\Cr(X)| \leq C.\]
Then homological resonance holds on $\bar V$ with $\cL$-coefficients and the same constants: for all homogeneous elements $\bar X \in SH_*^0(\bar V;\cL) \setminus \{0\}$ we have
\[ |\deg(\bar X) - \alpha\Cr(\bar X)| \leq C.\]
A similar statement holds in cohomology.

(b) Assume that for some constant $c>0$ we have
\begin{equation}\label{eq:lower-bound-Nm}
  \limsup_{m\to\infty} \frac{N(m)}{m}\ge c,
\end{equation}
where $N(m)$ is the number of distinct critical values $\Cr(X)$ of homology classes $X\in SH_*(V;\underline{\ell})$ of degree $\le m$. 
Let $\bar N(m)$ be the number of distinct critical values $\Cr(\bar X)$ of homology classes $\bar X\in SH^0_*(\bar V;\cL)$ of degree $\le m$. Then
$$
  \limsup_{m\to\infty} \frac{\bar N(m)}{m}\ge c
$$
with the same constant $c$.\footnote{In view of~\eqref{eq:N*m}, $N(m)$ is also the number of distinct critical values $\ccr(\xi)$ of cohomology classes $\xi\in SH^*(V;\underline{\ell})$ of degree $\le m$. Similarly for $\bar N(m)$.}

(c) If the assumptions of both (a) and (b) hold, then density holds for $V$ and $\bar V$ with the same constants: for each $\eps>0$ we have
$$
  \sum_{\gamma\in \cS_\eps} \frac{1}{\hat{i}(\gamma)} \ge c,
$$
where $\cS_\eps$ is the set of simple closed Reeb orbits $\gamma$ on $\p V$ with mean frequency $\alpha(\gamma)\in (\alpha-\eps,\alpha+\eps)$, and 
$$
  \sum_{\bar\gamma\in \cS_\eps^0} \frac{1}{\hat{i}(\bar\gamma)} \ge c,
$$
where $\cS_\eps^0$ is the set of liftable simple closed Reeb orbits $\bar\gamma$ on $\p\bar V$ with mean frequency $\alpha(\bar\gamma)\in (\alpha-\eps,\alpha+\eps)$.
\end{theorem}

\begin{proof}
Recall the definition of symplectic homology, see e.g.~\cite{Abouzaid-cotangent,CO}. Consider an asymptotically linear $1$-periodic Hamiltonian $\bar H$ and an SFT-like almost complex structure $\bar J$ on the completion $\wh{\bar V}$. Denote by $H,J$ their pullbacks to the completion $\wh{V}$. The Floer chain complex $FC_*^0(\bar H,\bar J;\cL)$ is the direct sum over all liftable $1$-periodic orbits $\bar\gamma$ of the stalk $\oplus_{\gamma\in p^{-1}(\bar \gamma)} \underline{\ell}_{\bar \gamma}$. 
Since the latter is the 
$\K$-vector space generated by the preimages of $\bar\gamma$ with coefficients in $\underline{\ell}$, 
the projection defines a canonical isomorphism of $\K$-vector spaces $FC_*(H,J;\underline{\ell})\cong FC^0_*(\bar H,\bar J;\cL)$. This isomorphism is a chain map because both differentials count Floer cylinders with respect to $(H,J)$ in $\wh{V}$, so it induces an isomorphism on homology $FH_*(H,J;\underline{\ell})\cong FH_*^0(\bar H,\bar J;\cL)$. Passing to the limit over Hamiltonians with slopes going to infinity, we obtain an isomorphism 
\begin{equation}\label{eq:local-coeff-iso-symp}
  p_*:SH_*(V;\underline{\ell})\stackrel{\simeq}{\longrightarrow} SH_*^0(\bar V;\cL).
\end{equation}
Since the projection preserves the Hamiltonian actions, this isomorphism preserves the critical values. Since it also preserves degrees, this proves parts (a) and (b). Part (c) now follows from Lemma~\ref{lem:density-estimate} applied on $V$ (with coefficients $\underline{\ell}$) and on $\bar V$ (with coefficients $\cL$).
\end{proof}

{\bf Coverings of loop spaces. }
Consider a normal finite cover $p:M\to\bar M=M/G$ of closed manifolds. This induces a cover of cotangent bundles $T^*p:T^*M\to T^*\bar M$ that intertwines the canonical Liouville forms. Consider a fiberwise starshaped domain $\bar V\subset T^*\bar M$ and its preimage $V\subset T^*M$. Let $\eta$ be the spin $\K$-local system of rank $1$ on $\Lambda T^*M$, let $\cL=(\Lambda T^*p)_*\eta$ for the resulting normal finite cover $\Lambda T^*p:\Lambda V\to\Lambda \bar V$, and let $SH_*^0(\bar V;\cL)$ be the part of symplectic homology of $\bar V$ with local coefficients $\cL$ corresponding to liftable loops. By Viterbo's theorem $SH_*^0(\bar V;\cL)\simeq H_*(\Lambda_0\bar M;(\Lambda p)_*\K)$ for the normal finite cover $\Lambda p:\Lambda M\to\Lambda_0\bar M$. 

We will now apply this in the case $M=S^n$, using the following resonance and density theorems for spheres. They have been proved by Hingston and Rademacher in~\cite{Hingston-Rademacher} for Finsler metrics, and they directly carry over to fiberwise starshaped domains combining their arguments with the ones in our~\S\ref{sec: res} and~\S\ref{sec:density}. 

\begin{theorem}[{Resonance theorem for spheres~\cite[Theorem~1.1]{Hingston-Rademacher}}] 
\label{thm:resonance-Sn}
Let $V\subset T^*S^n$ be a fiberwise starshaped domain, $n> 2$, and $\bK$ a coefficient field.
There exist positive constants $\alpha,C$ such that for all homogeneous elements $X \in SH_*(T^* S^n;\eta) \setminus \{0\}$ and $\xi\in SH^*(T^*S^n;\eta)\setminus \{0\}$ we have
$$
  |\deg(X) - \alpha \, \Cr(X)|\leq C,\qquad |\deg(\xi) - \alpha \, \ccr(\xi)|\leq C.
$$
\qed
\end{theorem}

\begin{theorem}[{Density theorem for spheres~\cite[Theorem~1.2]{Hingston-Rademacher}}] 
\label{thm:density-Sn}
Let $\alpha=\alpha_{V,\bK}$ be the global mean frequency from Theorem~\ref{thm:resonance-Sn} for a fiberwise starshaped domain $V\subset T^*S^n$, $n>2$, and a coefficient field $\bK$. For any $\eps>0$ we have the estimate 
$$
\sum_{\gamma\in \cS_\eps} \frac{1}{\hat{i}(\gamma)} \ge 
  \limsup_{m\to\infty} \frac{N(m)}{m} \ge
\left\{\begin{array}{ll} \frac{1}{n-1}, & n \mbox{ odd},\\ \frac{1}{2(n-1)}, & n \mbox{ even},
\end{array}\right.
$$
where $N(m)$ is the number of distinct critical values $\Cr(X)$ of homology classes $X\in SH_*(T^* S^n;\eta)$ of degree $\le m$, and 
$\cS_\eps$ is the set of simple closed Reeb orbits $\gamma$ on $\p V$ with mean frequency $\alpha(\gamma)\in(\alpha-\eps,\alpha+\eps)$.  \qed
\end{theorem}

Theorems~\ref{thm: resonance CP1-Liouville} and~\ref{thm: dense} provide counterparts of the above two theorems for $S^2$, with coefficients in a field $\K$ of characteristic $\neq 2$.

Combining these two theorems with Theorem~\ref{thm: resonance-Liouville-quotient} yields

\begin{corollary}[Resonance and density for finite quotients of spheres]\label{cor:density-quotients-spheres}
Let $\bar M$ be a finite quotient of a sphere $M=S^n$, $n\ge 2$. Let $\bar V\subset T^*\bar M$ be a fiberwise starshaped domain, and $\bK$ a coefficient field, of characteristic $\neq 2$ if $n=2$.

(a) Resonance holds for $\bar V$ with $\cL$-coefficients and with the same constants $\alpha,C$ as for its preimage $V\subset T^*S^n$:
for all homogeneous elements $\bar X \in SH_*^0(\bar V;\cL) \setminus \{0\}$ and $\bar\xi\in SH^*_0(\bar V;\cL)\setminus\{0\}$ we have
\[ 
|\deg(\bar X) - \alpha\Cr(\bar X)| \leq C,\qquad |\deg (\bar\xi) - \alpha\, \ccr(\bar \xi)|\le C.
\]
(b) Density holds for $\bar V$ with the same constants as for $V\subset T^*S^n$: for every $\eps>0$ we have
$$
\sum_{\bar\gamma\in \cS_\eps^0} \frac{1}{\hat{i}(\bar\gamma)} \geq
\left\{\begin{array}{ll} \frac{1}{n-1}, & n \mbox{ odd},\\ \frac{1}{2(n-1)}, & n \mbox{ even},
\end{array}\right.
$$
where $\cS_\eps^0$ is the set of contractible simple closed Reeb orbits $\bar\gamma$ on $\p\bar V$ with $\alpha(\bar\gamma)\in (\alpha-\eps,\alpha+\eps)$. \qed
\end{corollary}

In particular, this corollary applies to real projective spaces. In fact, applying Corollary~\ref{cor:density-quotients-spheres} for $n$ odd, and a theorem of Stegemeyer~\cite[Theorem~7.6]{Stegemeyer} for $n$ even, we get in this case a better lower bound in the density theorem if we also consider noncontractible Reeb orbits:

\begin{corollary}[Density theorem for real projective spaces]\label{cor:density-projective-spaces}
Let $\bar V\subset T^*\R P^n$ be a fiberwise starshaped domain, $n\ge 2$, and $\bK$ a coefficient field, of characteristic $\neq 2$ if $n=2$. Let $\alpha=\alpha_{\bar V,\K}$ be the global mean frequency from Corollary~\ref{cor:density-quotients-spheres}. 
Then for each $\eps>0$ we have 
$$
\sum_{\bar \gamma\in\cS_\eps} \frac{1}{\hat{i}(\bar \gamma)} \geq \frac{1}{n-1},  
$$
where $\cS_\eps$ is the set of {\em all} simple closed Reeb orbits $\bar\gamma$ on $\p\bar V$ with $\alpha(\bar\gamma)\in (\alpha-\eps,\alpha+\eps)$. \qed
\end{corollary}

\begin{remarks}
(a) Stegemeyer's proof for even dimensional projective spaces actually shows that the contractible and non-contractible Reeb orbits each contribute at least $\frac{1}{2(n-1)}$ to the sum in Corollary~\ref{cor:density-projective-spaces}. It uses the decomposition $\Lambda\R P^n = \Lambda_0\R P^n\sqcup \Lambda_1\R P^n$ into contractible and noncontractible loops with its $2$-to-$1$ cover $\Lambda S^n\sqcup  P_1S^n\to \Lambda_0\R P^n\sqcup \Lambda_1\R P^n$ (see the discussion at the beginning of this section), and the definition and calculation of a novel string topology algebra on $\Lambda S^n\sqcup P_1S^n$. It would be interesting to see whether Stegemeyer's proof carries over to the coverings~\eqref{eq:G-covering} for more general quotients of $S^n$ to give better lower bounds in the density theorem when considering all Reeb orbits. 

(b) For a $k$-fold covering $p:V\to\bar V$ of Liouville domains as in Theorem~\ref{thm: resonance-Liouville-quotient}, each liftable simple closed Reeb orbit $\bar\gamma$ on $\p\bar V$ has precisely $k$ lifts to $\p V$, which are all simple and have the same length, average index, and mean frequency as $\bar \gamma$, and conversely every simple closed Reeb orbit on $\p V$ arises in this way. Thus a density theorem for $V$ with lower bound $c$ immediately implies a density theorem for $\bar V$ with the weaker lower bound $c/k$. By contrast, the algebraic argument in the proof of Theorem~\ref{thm: resonance-Liouville-quotient} yields density for $\bar V$ with lower bound $c$.  

(c) Theorem~\ref{thm: resonance-Liouville-quotient} and its corollaries apply in particular to a Finsler metric $\bar F$ on $\bar M$ and its lift $F$ to a finite cover $M\to\bar M$, with symplectic homology replaced by loop space homology and Reeb orbits replaced by geodesics. In this case we also obtain an analogue to Corollary~\ref{cor:two-or-infinity}.

(d) We have already established resonance and lower bounds~\eqref{eq:lower-bound-Nm}, as well as density theorems, for complex, quaternionic and octonionic projective spaces. One can ask whether there are any finite quotients to consider. 
Let $G$ be a finite group acting freely on $\K P^d$ with $\K=\C,\H,\Ca$ (and $d=2$ in this last case). The following are true: 
\begin{itemize}
\item $|G|$ divides $d+1=\chi(\K P^d)$. This follows by comparing Euler characteristics.  
\item Any element in $G$ is $2$-torsion. Indeed, it is a general fact that the square of any diffeomorphism of $\K P^d$ acts by the identity on homology (for rank reasons), so the Lefschetz fixed point theorem ensures that it has at least one fixed point. However, the square of any element in $G$ acts freely unless it is the identity.
\end{itemize}
Thus, if $G$ is nontrivial then $2$ divides $d+1$. In particular, nontrivial finite groups cannot act freely on $\K P^d$ with $d$ even. Since our resonance and density theorems for $\K P^d$ hold when $d+1$ is prime, there are no finite quotients to consider for $d\ge 2$.    

(e) Finite quotients of $S^n$, $n\geq 2$, by {\em linear} actions are spherical space forms, which have been classified in~\cite[\S7.4]{Wolf}. For $n$ even, the only spherical space forms are $S^n$ and $\R P^n$. For $n$ odd, further examples include lens spaces and quotients associated to the symmetry groups of Platonic solids such as the Poincar\'e homology sphere. 
For $n=3$, as a consequence of Perelman's work (see~\cite{Morgan-Tian07}), all finite quotients of $S^3$ are quotients by linear actions; in higher dimensions this is unknown.
\end{remarks}

\begin{remark}[On string point invertibility]
String point invertibility is only defined for orientable manifolds. More generally, one could declare a nonorientable manifold $M$ to be string point invertible (SPI) if its orientation double cover $\tilde M$ is SPI. Note that, in the vein of the present chapter, this is a property that can be read on $H_*(\Lambda_0 M;\cL)$ for a suitable rank 2 local system $\cL$. Accordingly, the spectral norm from Theorem~\ref{thm:Viterbo-conj} can be interpreted as a spectral norm for Lagrangian Floer homology with coefficients in $\cL$. We find it an interesting question to pursue the implications that string topology with coefficients in more general local systems can have for symplectic topology. 
\end{remark}

\appendix

\section{Another proof of the Duality Lemma~\ref{lma: duality}}

We present a proof of a slightly different result, which in turn implies Lemma~\ref{lma: duality}. We work in the setup of a Liouville domain $V$ with filtered symplectic homology groups $SH_*(V)$, cohomology groups $SH^*(V)$, and their filtered versions
\begin{gather*}
   \cdots SH_k^{\leq\alpha}(V) \stackrel{I_\alpha}\longrightarrow SH_k(V) \stackrel{R_\alpha}\longrightarrow SH_k^{>\alpha}(V)\cdots, \\
   \cdots SH^k_{>\alpha}(V) \stackrel{i_\alpha}\longrightarrow SH^k(V) \stackrel{r_\alpha}\longrightarrow SH^k_{\leq\alpha}(V)\cdots
\end{gather*}
We work with coefficients in a field $\K$, so that $SH^*(V)\simeq SH_*(V)^\vee$, $SH^*_{>\alpha}(V)\simeq SH_*^{>\alpha}(V)^\vee$, and $SH^*_{\le \alpha}(V)\simeq SH_*^{\le \alpha}(V)^\vee$. 
Then $i_\alpha=R_\alpha^\vee$ and $r_\alpha=I_\alpha^\vee$ are the dual maps. Recall the critical values of nonzero classes $X\in SH_k(V)$ and $x\in SH^k(V)$, 
\begin{align*}
  Cr(X) &= \inf\{\alpha\mid X\in\im I_\alpha=\ker R_\alpha\}, \cr
  cr(x) &= \sup\{\alpha\mid x\in\im i_\alpha=\ker r_\alpha\}.
\end{align*}


\begin{lemma} \label{lem:crx=CrX} 
(a) Given a nonzero homology class $X\in SH_k(V)$, there exists $x\in SH^k(V)$ such that $\langle x,X\rangle=1$ and $cr(x)=Cr(X)$.

(b) Given a nonzero cohomology class $x\in SH^k(V)$, there exists $X\in SH_k(V)$ such that $\langle x,X\rangle=1$ and $cr(x)=Cr(X)$.
\end{lemma}

\begin{proof} 
We will prove part (a), the proof of (b) being analogous. 

Consider $0\neq X\in SH_k(V)$. By definition, $C:=Cr(X)$ is characterized by the conditions 
$$
  X\in\im I_\alpha \text{ for all }\alpha>C,\qquad X\notin\im I_\alpha \text{ for all }\alpha<C.
$$
In particular, $X$ does not belong to the image of $SH_k^{<C}(V)$ inside $SH_k(V)$. Therefore, we find $x\in SH^k(V)$ that vanishes on the image of $SH_k^{<C}(V)$ inside $SH_k(V)$ and satisfies $\la x,X\ra=1$. 
By definition, $c:=cr(x)$ is characterized by the conditions 
$$
  r_\alpha x=0 \text{ for all }\alpha<c,\qquad r_\alpha x\neq 0 \text{ for all }\alpha>c.
$$
Suppose first that $C<c$ and pick $C<\alpha<c$. Then $X\in\im I_\alpha$, so $X=I_\alpha X_\alpha$ for some $X_\alpha\in SH_k^{\leq\alpha}$. Together with $r_\alpha x=0$ this yields the contradiction 
$$
  1 = \la x,X\ra = \la x,I_\alpha X_\alpha\ra = \la r_\alpha x,X_\alpha\ra = 0.
$$
Suppose next that $c<C$ and pick $c<\alpha<C$. Then $r_\alpha x\neq 0$, so there exists $X_\alpha\in SH_k^{\leq\alpha}$ such that $0\neq \la r_\alpha x,X_\alpha\ra = \la x,I_\alpha X_\alpha\ra$. But this contradicts our choice of $x$ to vanish on the image of $SH_k^{<C}(V)$ inside $SH_k(V)$.
Therefore, we must have $c=C$ and the lemma is proved. 
\end{proof}

Let us now see how this implies Lemma~\ref{lma: duality}.

\begin{proof}[Proof of Lemma~\ref{lma: duality}] It follows from the definitions that, given a cohomology class $x$ and a homology class $X$ such that $\langle x,X\rangle \neq 0$, we have $cr(x)\le Cr(X)$. This shows that, given $X$ we have 
\[\sup\{ cr(x) \mid \langle x,X\rangle\neq 0\}\le Cr(X),\]
and given $x$ we have 
\[cr(x)\le \inf \{Cr(X) \mid \langle x,X\rangle \neq 0\}.\]
That the equality is attained in each of the above equalities is a consequence of Lemma~\ref{lem:crx=CrX}.
\end{proof}

\section{The Frobenius pairing determines the Kronecker pairing}

Let $M$ be a closed orientable manifold of dimension $n$, and let $\chi=\chi(M)$ be its Euler characteristic. Let $\K$ be a field of coefficients. Let $[\mathrm{pt}]\in H_0\Lambda_0\hookrightarrow H_*\Lambda$ be the class of a point, and $1\in H^0\Lambda_0\hookrightarrow H^0\Lambda$ be the unit for the cup-product in cohomology. 

We define \emph{reduced loop homology} as 
$$
\ol H_*\Lambda= \left\{\begin{array}{ll} H_*\Lambda/\langle [\mathrm{pt}]\rangle, & \mbox{if } \chi\neq 0 \mbox{ in } \K,\\
H_*\Lambda & \mbox{if } \chi=0 \mbox{ in } \K,
\end{array}\right.
$$
and we define \emph{reduced loop cohomology} as 
$$
\ol H^*\Lambda= \left\{\begin{array}{ll} H^*\Lambda/\langle 1\rangle, & \mbox{if } \chi\neq 0 \mbox{ in } \K,\\
H^*\Lambda & \mbox{if } \chi=0 \mbox{ in } \K.
\end{array}\right.
$$
The Kronecker pairing $\langle\cdot,\cdot\rangle:H^*\Lambda\otimes H_*\Lambda\to\K$ descends to a pairing
$$
\langle\cdot,\cdot \rangle : \ol H^*\Lambda\otimes \ol H_*\Lambda\to \K,
$$ 
which we will also refer to as the \emph{Kronecker pairing (between reduced loop homology and reduced loop cohomology)}. 

We consider the splittings
\begin{equation} \label{eq:splitting_homology}
\wh H_*\Lambda =\ol H_*\Lambda\oplus \ol H^{1-*}\Lambda
\end{equation}
and 
\begin{equation} \label{eq:splitting_cohomology}
\wh H^*\Lambda=\ol H^*\Lambda\oplus \ol H_{1-*}\Lambda.
\end{equation}
These splittings are canonical if $H_1M=0$. 

Recall from~\cite{CHO-PD} that $\wh \H_*\Lambda=\wh H_{*+n}\Lambda$ is a topological Frobenius algebra. We denote $\boldeps$ the counit, $\boldmu$ the product, and $\boldp=-\boldeps\boldmu$ the Frobenius pairing. 

\begin{proposition} \label{prop:Kronecker-Frobenius} Assume $H_1M=0$. We then have the equality 
$$
\langle\cdot,\cdot\rangle = \boldp
$$
on $\ol H^*\Lambda \otimes \ol H_*\Lambda\subset \wh H_*\Lambda\otimes \wh H_*\Lambda$.
\end{proposition}

\begin{proof}
The Kronecker pairing is naturally identified with the restriction to $\ol H^*\Lambda\otimes\ol H_*\Lambda$ of the canonical pairing 
$$
\ev : \wh H^*\Lambda \otimes \wh H_*\Lambda\to \K. 
$$
This is a perfect topological pairing, which realizes $\wh H^*\Lambda$ as the topological dual of $\wh H_*\Lambda$.
\footnote{In a symplectic interpretation, this is defined by a count of cylinders, with one positive (homological) input and one negative (cohomological) input. This is implicit in the TQFT structure from~\cite{CHO-PD}, and explicit in~\cite{Ritter} in the setup of symplectic homology.} 

Let $\boldp$ be the pairing on $\wh H_*\Lambda$ and define $\vec \boldp:\wh H_*\Lambda\stackrel\simeq\longrightarrow \wh H^{1-*}\Lambda$ by the equality 
$$
\boldp=\ev(\vec \boldp\otimes 1)
$$ 
as in~\cite[Lemma~5.3]{CHO-algebra}. Then $\vec\boldp$ is the Poincaré duality isomorphism~\cite{CHO-PD}, and~\cite[Proposition~7.1]{CHO-reducedSH} shows that, under the assumption that $H_1M=0$, the map $\vec\boldp$ simply exchanges the factors in the decompositions~\eqref{eq:splitting_homology} and~\eqref{eq:splitting_cohomology}. The conclusion follows.
\end{proof}

\begin{remark}
For general $M$, the Kronecker pairing can be expressed in terms of the Frobenius pairing and the secondary continuation map as defined in~\cite[Proposition~7.1]{CHO-reducedSH}. We do not spell out this variant of Proposition~\ref{prop:Kronecker-Frobenius} because we only need it for simply connected CROSS. 
\end{remark}

\section{The index of a periodic orbit of an autonomous Hamiltonian}\label{sec:basepoint}

Let $\varphi_t$ be the flow of an autonomous Hamiltonian on a symplectic manifold $M$ of dimension $2n$, and let $\gamma$ be a nonconstant periodic orbit. Given a homotopy class of symplectic trivializations of $TM$ along $\gamma$, the index $i(\gamma)$ and mean index $\hat{i}(\gamma)$ are defined from the Bott-Long index for paths of symplectic matrices by linearizing the flow along the orbit with a given starting point. We show in this appendix that they do not depend on the choice of starting point. Specifically, we show that different choices of starting point give rise to paths of symplectic matrices whose classes in $\widetilde{\mathrm{Sp}}(2n)$ are conjugate. 

The question that we address is local around $\gamma$, and therefore we can assume without loss of generality that the flow is complete, $M=\R^{2n}$, and the trivialization of $TM$ along $\gamma$ is the canonical one. Let $T>0$ be the period of $\gamma$. 

Let $x\in \gamma$ and $s\in\R$ be fixed. We consider the paths $P:[0,T]\to \Sp(2n)$, $P(t)=D\varphi_t(x)$ and $P_s:[0,T]\to\Sp(2n)$, $P_s(t)=D\varphi_t(\varphi_s(x))$. These paths start at the identity and therefore define classes in $\widetilde{\Sp}(2n)$. 

\begin{lemma}
The classes $[P], [P_s]\in\widetilde{\Sp}(2n)$ are conjugate. 
\end{lemma}

\begin{proof}
The relation $\varphi_t\varphi_s=\varphi_s\varphi_t$ implies by differentiation at $x$ that 
$$
P_s(t)=D\varphi_s(\varphi_t(x)) P(t) D\varphi_s(x)^{-1}. 
$$
This path is homotopic with fixed endpoints to the path 
$$
Q_s(t) = D\varphi_s(x) P(t) D\varphi_s(x)^{-1},\qquad t\in [0,T].
$$
Indeed, we write 
$$
D\varphi_s(\varphi_t(x)) P(t) D\varphi_s(x)^{-1}= D\varphi_s(\varphi_t(x)) D\varphi_s(x)^{-1} D\varphi_s(x) P(t) D\varphi_s(x)^{-1},
$$ 
and the loop $t\mapsto D\varphi_s(\varphi_t(x)) D\varphi_s(x)^{-1}$ based at the identity is contractible via the homotopy $t\mapsto D\varphi_{rs}(\varphi_t(x)) D\varphi_{rs}(x)^{-1}$, $r\in [0,1]$. 

We find that $[P_s]=[Q_s]$, and $Q_s$ is conjugate to $P$. 
\end{proof}

Since $[P]=[P_s]$, all their invariants coincide, in particular 
$$
i(P)=i(P_s), \qquad \hat{i}(P)=\hat{i}(P_s). 
$$
Here $\hat{i}$ is the mean index, and $i$ is the Bott-Long index, alias $i_1$ in~\cite[Definitions~5.4.2 or 6.1.10]{Long-book}. That $i$ is invariant by conjugation is a consequence of its invariance under homotopies of paths such that the nullity at the endpoint is constant~\cite[Theorem~6.2.3]{Long-book}.

\bibliographystyle{abbrv}
\bibliography{000_Cross}

\end{document}